\documentclass[12pt]{article}
\usepackage[top=1.0in,bottom=1.0in,left=0.9in,right=0.9in]{geometry}
\usepackage{amsmath,amssymb}
\usepackage{mathrsfs}
\usepackage{bm}
\usepackage{enumerate}
\usepackage{multirow}
\usepackage{amsthm}
\usepackage{latexsym}
\usepackage{extarrows}
\usepackage{graphicx}
\usepackage{epsfig}
\usepackage{subfigure}
\usepackage{psfrag}
\usepackage{caption}
\usepackage{float}
\newtheorem{theorem}{Theorem}[section]

\newtheorem{lemma}[theorem]{Lemma}

\newtheorem{example}{Example}[section]

\numberwithin{equation}{section}
\numberwithin{figure}{section}
\numberwithin{table}{section}
\renewcommand{\vec}[1]{\mbox{\boldmath \small $#1$}}
    \renewcommand{\qed}{\hfill   \Box}
 \def\Box{\mbox{ }\rule[0pt]{1.5ex}{1.5ex}}
\allowdisplaybreaks

\begin{document}

\baselineskip=2pc

\begin{center}
{\bf Physical-constraints-preserving Lagrangian finite volume schemes for one- and two-dimensional
special relativistic hydrodynamics}
\end{center}


\centerline{Dan Ling$^1$, 
Junming Duan\footnote{HEDPS, CAPT \& LMAM, School of Mathematical Sciences,
Peking University, Beijing 100871, China.
E-mail: {\tt dan$\_$ling@math.pku.edu.cn}; {\tt duanjm@pku.edu.cn}},
Huazhong Tang$^{1,}$\footnote{School of Mathematics and Computational Science, Xiangtan University, Hunan Province, Xiangtan 411105, P.R. China. E-mail: {\tt hztang@math.pku.edu.cn}}
}

\vspace{.15in}

\centerline{\bf Abstract}

\bigskip

\baselineskip=1.4pc

This paper studies the physical-constraints-preserving (PCP) Lagrangian finite volume schemes
for one- and two-dimensional special relativistic hydrodynamic (RHD) equations.
First, the PCP property (i.e. preserving the positivity of the rest-mass density and the pressure and the bound of the velocity) is proved for the first-order accurate Lagrangian scheme with the HLLC Riemann solver and  forward Euler time discretization. The key is  that the intermediate states in the HLLC Riemann solver are shown to be admissible or PCP when the HLLC wave speeds are estimated suitably.
Then, the higher-order accurate schemes are proposed by using the high-order accurate strong stability preserving (SSP) time discretizations and the scaling PCP limiter as well as the WENO reconstruction.  
Finally, several one- and two-dimensional numerical experiments are
conducted to demonstrate the
accuracy and the effectiveness of the PCP Lagrangian schemes in solving the special RHD problems  involving strong discontinuities, or large Lorentz factor,
or low rest-mass density or low pressure etc.

\vspace{.05in}

\vfill


\newpage
\baselineskip=2pc

\section{Introduction}
The paper is concerned with the study of  the physical-constraints-preserving (PCP) Lagrangian schemes
(which preserve the positivity of the rest-mass density and pressure, and the bound of
the fluid velocity) for one- and two-dimensional special relativistic hydrodynamic (RHD) equations.
The $d$ dimensional governing equations of the special RHDs is a
system of first-order quasi-linear partial differential equations, see e.g. \cite{LL1987},
and in the laboratory frame, it can be written
in the divergence form
\begin{equation}\label{eq1}
\frac{\partial \vec{U}}{\partial t}+\sum\limits_{\ell=1}^d\frac{\partial\vec{F}_\ell(\vec{U})}{\partial x_\ell}=0,
\end{equation}
with the conservative vector $\vec{U}$  and the flux $\vec{F}_\ell$,   defined respectively by
\begin{equation}\label{eq2}
\vec{U}=(D, \vec{m}, E)^T, \quad
\vec{F}_\ell=\left(Du_\ell, \vec{m}u_\ell+p\vec{e}_\ell, (E+p)u_\ell \right)^T,
\end{equation}
where $D$, $\vec{m}=(m_1,\cdots,m_d)$, $E$ and $p$ are the mass, momentum and
total energy relative to the laboratory frame and the gas pressure, respectively,
$\vec{e}_\ell$ is the row vector denoting the $\ell$th row of the unit matrix of size $d$.
 In the Lagrangian framework, (\ref{eq1}) can be expressed as the   integral form
\begin{equation}\label{eq4}
\frac{d}{dt}\int_{\varOmega(t)}\vec{U}dV
+\int_{\partial\varOmega(t)}\vec{\mathcal F}_{\vec{n}}(\vec{U})ds=\vec{0},~
\vec{\mathcal F}_{\vec{n}}(\vec{U})=\left(0,p\vec{n}^T,p\vec{u}\cdot\vec{n}^T\right)^T,
\end{equation}
where $\varOmega(t)$ is the mass volume with the  boundary $\partial\varOmega(t)$ and
$\vec{n}$ denotes the unit outward normal of $\partial\varOmega(t)$.
The conservative vector $\vec{U}$ can be explicitly related to
the primitive variables $\vec{V}=(\rho,\vec{u},p)^{T}$ by
\begin{equation}\label{eq3}
D=\rho\gamma, \quad\vec{m}=Dh\gamma\vec{u}, \quad E=Dh\gamma-p.
\end{equation}
Here $\rho$ denotes the proper rest-mass density,
and the Lorentz factor $\gamma=(1-|\vec{u}|^2)^{-1/2}$ and the specific enthalpy
$h=1+e+\frac{p}{\rho}$ when units are normalized so that the speed of light is $c=1$.
The symbol $e$ is the specific internal energy and the fluid velocity $\vec{u}=(u_1,\cdots,u_d)$.
The system (\ref{eq1}) should be closed by an additional equation of state (EOS). Our discussion in this paper will be restricted to the perfect gas, whose EOS is formulated as follows
\begin{equation}\label{eq55}
p=(\Gamma-1)\rho e,
\end{equation}
with the adiabatic index $\Gamma\in(1, 2]$.
Such restriction on $\Gamma$ is reasonable under the compressibility assumptions and
the adiabatic index is taken as 5/3 for the mildly relativistic case and 4/3 for
the ultra-relativistic case. For the EOS \eqref{eq55}, the sound speed $c_s$ is given by
\begin{equation}\label{cs}
c_s=\sqrt{ {\Gamma p}/{(\rho h)}},
\end{equation}
which satisfies the following inequality
\begin{equation}
c_s^2=\frac{\Gamma p}{\rho h}=\frac{\Gamma p}{\rho+\frac{p}{\Gamma-1}+p}=\frac{(\Gamma-1)\Gamma p}{(\Gamma-1)\rho+\Gamma p}<\Gamma-1<1.
\end{equation}
It is an important result and will be used in proving the PCP property of
the HLLC Riemann solver.

The inverse of  \eqref{eq3} cannot be explicitly given and involves solving  the nonlinear equation, e.g. the
pressure equation
\begin{equation}\label{pre}
E+p=D\gamma+\frac{\Gamma}{\Gamma-1}p\gamma^2,
\end{equation}
for the EOS \eqref{eq55}, where the Lorentz factor $\gamma$ has been expressed as $\gamma=\big(1-|\vec{m}|^2/(E+p)^2\big)^{-1/2}$. Once $p$ is obtained by solving the equation \eqref{pre}, the rest-mass density $\rho$, the specific enthalpy $h$,
and the velocity $\vec u$ can be orderly calculated as follows
\begin{equation}\label{add1}
\rho=\frac{D}{\gamma},\ h=1+\frac{\Gamma p}{(\Gamma-1)\rho},\ \vec u=\frac{\vec m}{Dh}.
\end{equation}

The system \eqref{eq1} takes into account the relativistic description for the dynamics
of the fluid (gas) at nearly speed of light. The relativistic fluid flow appears in investigating
numerous astrophysical phenomena from stellar to galactic scales, e.g. coalescing neutron stars, core
collapse supernovae, formation of black holes, active galactic nuclei,
superluminal jets and gamma-ray bursts,  etc.
Due to the relativistic effect, especially the appearance of the Lorentz factor, the nonlinearity of the system \eqref{eq1}
becomes much stronger than the non-relativistic case so that its analytic treatment is extremely difficult and challenging even though
some literature studied the exact solution to the RHD equations in the special case, for
instance the one-dimensional Riemann problem or the isentropic problem \cite{marti1,pant,lora}.
A primary and powerful approach to improve our understanding of the physical
mechanisms in the RHDs is the numerical simulation. In comparison with
the non-relativistic case, the numerical difficulties are coming from strongly
nonlinear coupling between the RHD equations \eqref{eq1}, which leads to no explicit
expression of the primitive variable vector $\vec V = (\rho, \vec u, p)^T$ and the flux $\vec F_{\ell}$ in
terms of $\vec U$, and some physical constraints such as $\rho > 0, p > 0$, $E\geq D$ and $|\vec u| < c = 1$, etc.

The pioneering work in this field may trace back to the finite difference code
via the artificial viscosity
for the spherically symmetric general RHD equations in the Lagrangian coordinates \cite{may1,may2}.
The first attempt to solve RHD equations numerically in the Eulerian coordinates was made in \cite{wilson} by using an
explicit finite difference method with monotonic transport and artificial viscosity technique. After those, the numerical
study of the RHD equations has attracted more attention, and various  exact or approximate Riemann solvers have been successively proposed, for instance the HLL (Harten-Lax-van Leer) method \cite{schneider},
the flux corrected transport method \cite{duncan}, the two-shock solvers \cite{balsara,dai},
the Roe solver \cite{eulderink}, the upwind scheme \cite{falle}, the kinetic schemes \cite{yang,kunik}, the flux-splitting method \cite{donat}, the HLLC (HLL-Contact) scheme
\cite{mignone}, and so forth.
 In addition, some higher-order accurate  shock-capturing  schemes  were also developed
 for solving the RHD equations, such as
ENO (essentially non-oscillatory) and weighted ENO (WENO) methods \cite{dolezal,zanna,tchekhovskoy},
  finite volume local evolution Galerkin method
\cite{wu2014b},
 piecewise parabolic methods
\cite{marti3,mignone2}, adaptive mesh refinement method \cite{zw}, discontinuous Galerkin (DG) method \cite{radice},
direct Eulerian GRP (generalized Riemann problem) schemes \cite{yz1,yz2,wu2014,wu2016}, adaptive moving mesh methods \cite{he1,he2},
space-time conservation element and solution element method \cite{qamar} and Runge-Kutta DG methods with WENO
 limiter \cite{zhao} and so on. The readers are  also
referred to the review articles \cite{marti2,font,marti2015}
and references therein.

Unfortunately, in general, it cannot be proved that the above schemes
are PCP, in other words, their solutions  belong to the
physically admissible set ${\mathcal G}=\{\vec{U}=(D,\vec m,E)^T~\big|$ $\rho>0,~p>0, ~|\vec u|< c=1\}$.
Although they have been used to simulate some RHD flows successfully,
they are still confronted with enormous risks of  the calculation failure
 in solving the RHD problems with large Lorentz factor or low density or pressure, or strong discontinuity.
If the negative  density or pressure, or the larger velocity than the speed of light are numerically obtained,
the eigenvalues of the Jacobian matrix or the Lorentz factor may become imaginary, leading directly to the ill-posedness of the discrete problem.
In practice, such  nonphysical numerical solutions are usually simply replaced with a ``close'' and ``physical'' one by performing recalculation with more diffusive schemes and smaller CFL number until the numerical solutions become physical, see e.g. \cite{hughes,zw}.
Obviously, to some extent such operations are not
scientifically reasonable, so it is necessary and significant to develop high-order accurate physical-constraints-preserving (PCP) numerical schemes, whose solutions can satisfy those intrinsic physical constraints or properties: the positivity of the rest-mass density and the pressure and the bound of the velocity.
Recently, there exist some works on the PCP numerical schemes
for the RHD equations, such as the finite difference WENO schemes \cite{wu2015}, non-central DG methods with
bound-preserving limiter \cite{qin}, and central DG methods \cite{wu2017}.
Moreover, the admissible states and physical-constraints-preserving numerical schemes were first well-studied for the special relativistic magnetohydrodynamics \cite{wu2017b,wu2018}, and the provably  physical-constraint-preserving methods were successfully designed for the general RHDs \cite{wu2017a}.
Motivated by \cite{wu2017b,wu2018}, positivity-preserving  numerical schemes  was  analyzed  for the non-relativistic ideal magnetohydrodynamics in \cite{wu2018a,wu2018b}.
It is worthy of mentioning
that all of the aforementioned schemes with the PCP property are  formulated in the Eulerian framework.

The aim of this paper is to study the PCP Lagrangian schemes with the HLLC solver for the one- and two-dimensional RHD equations \eqref{eq1} with $d=1,2$.
The HLLC approximate Riemann solver has  a simpler form in the Lagrangian framework.
First, we  prove that  the intermediate states in the HLLC Riemann solver are admissible or PCP (that is, the
rest-mass density and pressure are positive and the fluid velocity magnitude is less than the
speed of light) when the HLLC wave speeds are estimated suitably.
Next, we derive the first-order PCP Lagrangian scheme with the HLLC solver  and then
propose the higher-order accurate PCP Lagrangian scheme via the WENO reconstruction
and the scaling PCP limiter as well as the strong stability preserving (SSP) high order time discretization.

The paper is organized as follows.
Section \ref{section:1D-PCP}  gives the 1D PCP Lagrangian schemes.
The PCP properties of intermediate states in the HLLC Riemann solver are   proved,
and then the first- and high-order PCP Lagrangian
schemes are proposed.
Section \ref{section:2D-PCP} presents the 2D PCP Lagrangian
schemes. It needs more techniques to derive the PCP properties of the HLLC Riemann solver
for the $x_i$-split 2D RHD equations,
because an important difference from the 1D RHD equations comes from a purely multi-dimensional
relativistic feature that the flow regions across the shock or rarefaction wave in the split 2D RHD equations are nonlinearly coupled through the Lorentz factor which is also built in terms of the tangential velocities.
Some  numerical experiments  are conducted in Section \ref{section:NumResult} to demonstrate the
 performance such as the conservation, PCP property, accuracy of our Lagrangian schemes.
 Finally some remarkable conclusions are summarized in Section \ref{section:conclusion}.

\section{One-dimensional PCP Lagrangian schemes}\label{section:1D-PCP}
This section considers the Lagrangian finite volume schemes for the special relativistic hydrodynamics (RHD) equations \eqref{eq1} or \eqref{eq4}  with $d=1$, here
the notations $\vec{ F}_1$, $u_1$ and $x_1$ are replaced with $\vec{ F}$, $u$ and $x$, respectively,
and $\vec{\mathcal F}_{\vec{n}}(\vec{U})=\left(0,p,pu\right)^T=:\vec{\mathcal F}(\vec{U})$.

Assume that the time interval $\{t>0\}$ is discretized as:
$t_{n+1}=t_n+\Delta t^n$, $n=0,1,2,\cdots$,
 where $\Delta t^n$ is the time step size at $t=t_n$ and will be determined
 by the CFL type condition.
The (dynamic) computational domain at $t=t_n$  is divided into $N$  cells: $I_i^n =[x^n_{i-\frac{1}{2}},x^n_{i+\frac{1}{2}}]$ with the sizes $\Delta x_i^n =x^n_{i+\frac{1}{2}}-x^n_{i-\frac{1}{2}}$
for $i=1,\cdots,N$.
The mesh node $x^n_{i+\frac{1}{2}}$ is moved with
the fluid velocity $u^n_{i+\frac{1}{2}}$, that is,
$$
x^{n+1}_{i+\frac{1}{2}}=x^n_{i+\frac{1}{2}}+\Delta t^n u^n_{i+\frac{1}{2}},\ i=1,\cdots,N.
$$
For the RHD system \eqref{eq1} or \eqref{eq4} with $d=1$,
 the  Lagrangian finite volume  scheme with the Euler forward time discretization
\begin{equation}\label{eq8}
\overline{\vec{U}}_{i}^{n+1}\Delta x_i^{n+1}=\overline{\vec{U}}_{i}^{n}\Delta x_i^{n}-\Delta t^n
\left(\widehat{\vec{\mathcal F}}(\vec{U}_{i+\frac{1}{2}}^{-},\vec{U}_{i+\frac{1}{2}}^{+})-\widehat{\vec{\mathcal F}}(\vec{U}_{i-\frac{1}{2}}^{-},
\vec{U}_{i-\frac{1}{2}}^{+})\right),
\end{equation}
where $\overline{\vec{U}}_i^{n}$ and $\overline{\vec{U}}_i^{n+1}$ are approximate
cell average values of the conservative vectors  at $t_n$ and $t_{n+1}$ over the cells $I_i^n$ and $I_i^{n+1}$,
respectively, the numerical flux
$\widehat{\vec{\mathcal F}}$ is a function of two variables and satisfies the consistency
 $\widehat{\vec{\mathcal F}}(\vec U,\vec U)= \left(0,p,pu\right)^T$,
and $\vec{U}_{i+\frac{1}{2}}^{\mp}$ are the left- and right-limited approximations of $\vec{U}$
 at $x_{i+\frac{1}{2}}$, respectively, and obtained by using the initial reconstruction technique and $\{\overline{\vec{U}}_i^{n}\}$.

In this paper, the numerical flux $\widehat{\vec{\mathcal F}}(\vec U^-,\vec U^+)$ is computed by means of  the HLLC Riemann solver of \eqref{eq1} with $d=1$ \cite{mignone}, but in a coordinate system that is moving with the intermediate wave (i.e. contact discontinuity) speed $s^*$ (see below).
Let us consider the  Riemann problem
\begin{equation}\label{eq11}
 \vec{U}_t+\widetilde{\vec{F}}(\vec{U})_{\xi}=0,\ \
 \vec{U}(\xi,0)=\left\{\begin{array}{ll}
\vec{U}^{-},&\xi<0,\\
\vec{U}^{+},&\xi>0,
\end{array}\right.
\end{equation}
where $\xi=x-s^* t$ and
\begin{equation}\label{eq12}
\vec{U}=(D,m,E)^T,\quad \widetilde{\vec{F}}(\vec{U})=\left(
D(u-s^\ast), m(u-s^\ast)+p, E(u-s^\ast)+pu
\right)^T.
\end{equation}
\begin{figure}[h!]
\centering
\includegraphics[width=4.2in]{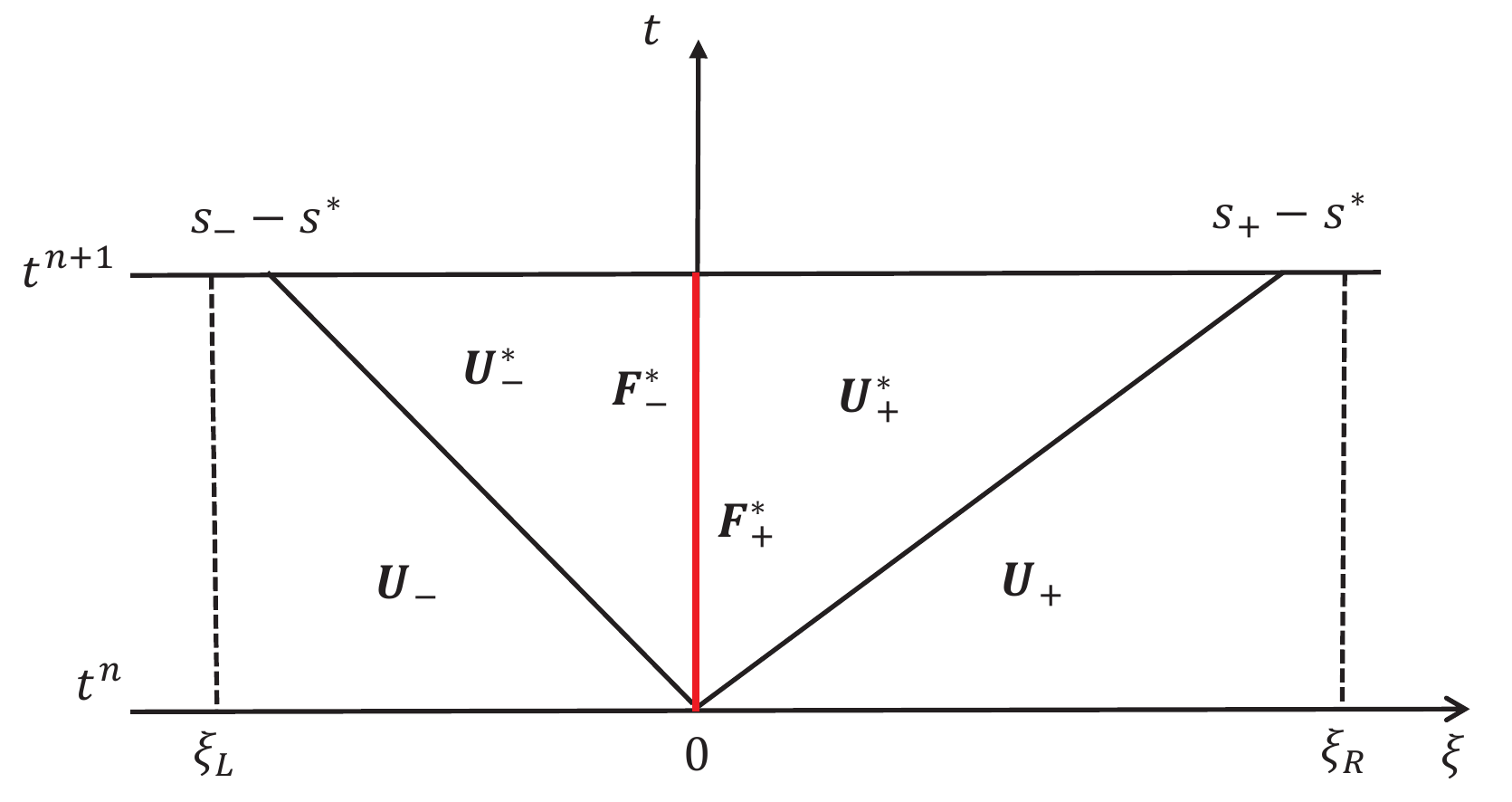}
\caption{Structure of the exact or HLLC approximate solution of the Riemann problem \eqref{eq11}.}
\label{fig-hllc}
\end{figure}
 Figure \ref{fig-hllc} shows corresponding wave structure of the exact or HLLC approximate solution of the Riemann problem \eqref{eq11}.
 There are three wave families associated with the eigenvalues of $\partial \widetilde{\vec{F}}/\partial \vec{U}$,
 and the contact discontinuity is always in accordance with the vertical axis $\xi=0$.
Two constant intermediate states between the left and right  acoustic waves are
denoted by $\vec{U}^{\ast,-}$ and $\vec{U}^{\ast,+}$, respectively, corresponding fluxes
are denoted by $\widetilde{\vec F}^{\ast,\pm}$.
The symbols $s^{-}$ and $s^{+}$ denote the smallest and largest speeds of the acoustic waves, respectively,
and are specifically set as
\begin{equation}\label{ws}
s^{-}=\min\big(s_{\min}(\vec{U}^{-}),s_{\min}(\vec{U}^{+})\big),~~s^{+}=\max\big(s_{\max}(\vec{U}^{-}),s_{\max}(\vec{U}^{+})\big),
\end{equation}
where $s_{\max}(\vec{U})$ and $s_{\min}(\vec{U})$ are the maximum and minimum eigenvalues of
the Jacobian matrix $\partial {\vec F}/\partial {\vec U}$, and
can be  explicitly given by
\begin{equation}\label{eq31}
s_{\min}(\vec{U})=\frac{u-c_s}{1-c_su},
~~~s_{\max}(\vec{U})=\frac{u+c_s}{1+c_su}.
\end{equation}
Integrating
the conservation law in \eqref{eq11} over the control volume
$[\xi_L,0]\times[t_n,t_{n+1}]$, the left portion of Figure \ref{fig-hllc}, gives
 the flux $\widetilde{\vec {F}}^{\ast,-}$ along the $t$-axis as follows
\begin{equation}\label{eq50}
\widetilde{\vec {F}}^{\ast,-}=\widetilde{\vec {F}}(\vec U^-)-(s^{-}-s^{\ast})\vec{U}^{-}
-\frac{1}{\Delta t^n}\int_{\Delta t^n(s^{-}-s^{\ast})}^{0}\vec{U}(\xi,t_{n+1})d\xi.
\end{equation}
 Similarly, performing the same operation on the control volume $[0,\xi_R]\times[t_n,t_{n+1}]$ gives
\begin{equation}\label{eq51}
\widetilde{\vec {F}}^{\ast,+}=\widetilde{\vec {F}}(\vec U^+)-(s^{+}-s^{\ast})\vec{U}^{+}
+\frac{1}{\Delta t^n}\int_{0}^{\Delta t^n(s^{+}-s^{\ast})}\vec{U}(\xi,t_{n+1})d\xi.
\end{equation}
By means of the the definitions of the intermediate states $\vec{U}^{\ast,\pm}$ \cite{toro}
\begin{equation*}
\vec{U}^{\ast,\pm} =\frac{1}{\Delta t^n(s^{\pm}-s^{\ast})}\int_{0}^{\Delta t^n(s^{\pm}-s^{\ast})}\vec{U}(\xi,t_{n+1})d\xi,
\end{equation*}
one can rewrite \eqref{eq50} and \eqref{eq51} as
\begin{equation}\label{eq53}
\widetilde{\vec {F}}^{\ast,\pm}=\widetilde{\vec {F}}(\vec U^{\pm})-(s^{\pm}-s^{\ast})(\vec{U}^{\pm}-\vec{U}^{\ast,\pm}),
\end{equation}
which are the Rankine-Hugoniot conditions for the right and left waves in Figure \ref{fig-hllc}.
If  assuming
$\vec{U}^{\ast,\pm}=(D^{\ast,\pm},m^{\ast,\pm},E^{\ast,\pm})^T$ and
$\widetilde{\vec {F}} ^{\ast,\pm}=\big(D^{\ast,\pm}(u^{\ast,\pm}-s^{\ast}),m^{\ast,\pm}(u^{\ast,\pm}-s^{\ast})+p^{\ast,\pm},
E^{\ast,\pm}(u^{\ast,\pm}-s^{\ast})+p^{\ast,\pm}u^{\ast,\pm}\big)^T$,
  \eqref{eq53} can explicitly give
\begin{equation}\label{add7}
\begin{aligned}
&D^{\ast,\pm}(s^{\pm}-u^{\ast,\pm})=D^{\pm}(s^{\pm}-u^{\pm}),\\
&m^{\ast,\pm}(s^{\pm}-u^{\ast,\pm})=m^{\pm}(s^{\pm}-u^{\pm})+p^{\ast,\pm}-p^{\pm},\\
&E^{\ast,\pm}(s^{\pm}-u^{\ast,\pm})=E^{\pm}(s^{\pm}-u^{\pm})+p^{\ast,\pm}u^{\ast,\pm}-p^{\pm}u^{\pm}.\\
\end{aligned}
\end{equation}
From those, one  obtains
\begin{equation}\label{eq15}
(A^{\pm}+s^{\pm}p^{\ast,\pm})u^{\ast,\pm}=B^{\pm}+p^{\ast,\pm},
\end{equation}
where
\begin{equation}\label{eq67}
A^{\pm}=s^{\pm}E^{\pm}-m^{\pm},~~~B^{\pm}=m^{\pm}(s^{\pm}-u^{\pm})-p^{\pm}.
\end{equation}

Because the  system \eqref{add7} is underdetermined,
one can impose  the  conditions on the continuity of the pressure and velocity across the contact discontinuity,
i.e.
\begin{equation}\label{add8}
p^{\ast,-}=p^{\ast,+}=:p^{\ast},~~u^{\ast,-}=u^{\ast,+}=:u^{\ast} (=s^{\ast}).
\end{equation}
Combing \eqref{eq15} with \eqref{add8} gives
\begin{equation}\label{eq16}
(1-s^{\pm}s^\ast)p^\ast=s^\ast A^{\pm}-B^{\pm}.
\end{equation}
Eliminating $p^*$ gives a quadratic equation in terms of $s^\ast$  as follows
\begin{equation}\label{eq17}
\mathcal{C}_0+\mathcal{C}_1s^\ast+\mathcal{C}_2(s^\ast)^2=0,
\end{equation}
where
\begin{equation*}
\mathcal{C}_0=B ^{+}-B^{-},~~\mathcal{C}_1=A^{-}+s^{+}B^{-}-A^{+}-s^{-}B^{+},~~
\mathcal{C}_2=s^{-}A^{+}-s^{+}A^{-}.
\end{equation*}
From \eqref{eq17}, one has
\begin{equation}\label{eq19}
s^{\ast}=u^{\ast}=\frac{-\mathcal{C}_1-\sqrt{\mathcal{C}_1^2-4\mathcal{C}_0\mathcal{C}_2}}{2\mathcal{C}_2},
\end{equation}
since the wave speed $s^\ast$ should be less than
 the speed of light $c=1$ \cite{mignone}.
Then  \eqref{eq16} further gives
\begin{equation}\label{eq66}
p^{\ast}=\frac{s^\ast A^{-}-B^{-}}{1-s^{-}s^\ast}=\frac{s^\ast A^{+}-B^{+}}{1-s^{+}s^\ast}.
\end{equation}

From \eqref{add7}, one can obtain $\vec U^{\ast,\pm}$ and then
substitute $\vec U^{\ast,\pm}$ into
\eqref{eq53} to give the HLLC flux (approximating $\widetilde{\vec{F}}(\vec{U})$) in the Lagrangian framework
$\widehat{\vec{\mathcal{F}}}=\widetilde{\vec {F}}^{\ast,+}=\widetilde{\vec {F}}^{\ast,-}$.
It is worth noticing that $p^{\ast}$ in \eqref{eq66} and   $u^\ast$  in
 \eqref{eq19}
are different from the pressure and the velocity calculated from the resulting $\vec U^{\ast,\pm}$ via \eqref{pre} and \eqref{add1},
so in general
$p^{\ast}\neq p(\vec{U}^{\ast,\pm})$, $u^{\ast}\neq u(\vec{U}^{\ast,\pm})$  and
$\widetilde{\vec {F}}^{\ast,\pm}\neq \widetilde{\vec{F}}(\vec{U}^{\ast,\pm})$.

For the above HLLC solver, the following conclusion holds.
\begin{lemma}\label{lem3}
For the given $\vec{U}^{\pm}=(D^{\pm},m^{\pm},E^{\pm})^{T}$, if the wave speed  $s^{\pm}$ are estimated  in \eqref{ws}  and $s^{\ast}$ is the speed of the contact discontinuity,
then $A^{\pm}$ and $B^{\pm}$  defined in \eqref{eq67} satisfy
 the following
inequalities
\begin{enumerate}[(\romannumeral1)]
\item $A^{-}<0,~~A^{+}>0$;
\item $A^{-}-B^{-}<0,~~~A^{+}+B^{+}>0$;
\item $s^{-}A^{-}-B^{-}>0,~~s^{+}A^{+}-B^{+}>0$;
\item $s^{+}A^{-}-B^{-}<0,~~s^{-}A^{+}-B^{+}<0$;
\item $s^{-}<u^{\pm}<s^{+},~~s^{-}<s^{\ast}<s^{+}$;
\item $4(A^{\pm})^2-(s^{\pm}A^{\pm}+B^{\pm})^2>0$;
\item $(A^{\pm})^2-(B^{\pm})^2-(D^{\pm})^2(s^{\pm}-u^{\pm})^2>0$.
\end{enumerate}
\end{lemma}
\noindent
The proof is presented in Appendix  \ref{appendix-A}.

\subsection{First-order accurate scheme}

For the first-order accurate Lagrangian scheme, $\vec{U}_{i\pm\frac{1}{2}}^{\mp}$ are calculated
 from the reconstructed piecewise constant function,
namely, $\vec{U}_{i\pm\frac{1}{2}}^{\mp}$ in the scheme \eqref{eq8} are
defined by
\begin{equation}\label{eq9}
\vec{U}_{i+\frac{1}{2}}^{-}=\overline{\vec{U}}_{i}^{n}, \quad
\ \vec{U}_{i+\frac{1}{2}}^{+}=\overline{\vec{U}}_{i+1}^{n}.
\end{equation}
Those form some local Riemann problems at the endpoints of each cell naturally, see the diagrammatic sketch in Figure \ref{scheme}.
The endpoints of the cell   will be evolved by
\begin{equation}\label{eq8-2}
\begin{aligned}
&x_{i+\frac{1}{2}}^{n+1}=x_{i+\frac{1}{2}}^{n}+\Delta t^nu_{i+\frac{1}{2}}^{\ast},\ \
u_{i+\frac{1}{2}}^{\ast}=s_{i+\frac{1}{2}}^{\ast}.
\end{aligned}
\end{equation}

Similar to the Godunov scheme, under a suitable CFL condition (see the coming Theorem \ref{theorem2.3}),
the waves in two neighboring local Riemann problems (e.g. centered at the points $x_{i-\frac{1}{2}}^n$ and $x_{i+\frac{1}{2}}^n$) do not interact with each other within
a time step, so
  $\overline{\vec{U}}^{n+1}_i$ in the scheme \eqref{eq8} can be equivalently
derived  by  exactly integrating the approximate Riemann solutions over
the cell $[x^{n+1}_{i-\frac{1}{2}}, x^{n+1}_{i+\frac{1}{2}}]$, i.e.
\begin{equation}\label{eq21}
\overline{\vec{U}}^{n+1}_i=\frac{1}{\Delta x_i^{n+1}}\left(\int_{x_{i-\frac{1}{2}}^{n+1}}^{x_1}R_h(x/t,\overline{\vec{U}}_{i-1}^{n},
\overline{\vec{U}}_{i}^{n})dx+\int_{x_1}^{x_2}\overline{\vec{U}}_i^ndx+
\int_{x_2}^{x_{i+\frac{1}{2}}^{n+1}}R_h(x/t,\overline{\vec{U}}_i^n,\overline{\vec{U}}_{i+1}^n)dx\right),
\end{equation}
where $R_h(x/t,\overline{\vec{U}}^n_{j-1},\overline{\vec{U}}^n_j)$  is the approximate Riemann
solution related to the states $\overline{\vec{U}}^n_{j-1}$
and $\overline{\vec{U}}^n_j$,  $j=i,i+1$. For the above HLLC Riemann solver, the integration of $R_h(x/t,\overline{\vec{U}}^n_{i-1},\overline{\vec{U}}^n_i)$
will be equal to $(x_1-x_{i-\frac{1}{2}}^{n+1})\vec{U}^{\ast,+}$ with  $\vec{U}^{\ast,+}$ is
computed from  $\overline{\vec{U}}^n_{i-1}$ and
$\overline{\vec{U}}^n_i$,
while the integration of  $R(x/t,\overline{\vec{U}}^n_i,\overline{\vec{U}}^n_{i+1}$)  will be $(x_{i+\frac{1}{2}}^{n+1}-x_2)\vec{U}^{\ast,-}$ with
$\vec{U}^{\ast,-}$ derived from  $\overline{\vec{U}}^n_i$ and
$\overline{\vec{U}}^n_{i+1}$.  Thus the first-order HLLC scheme is equivalent to
\begin{equation}\label{eq:hllc1d}
\overline{\vec{U}}^{n+1}_i=\frac{1}{\Delta x_i^{n+1}}\left((x_1-x_{i-\frac{1}{2}}^{n+1})\vec{U}^{\ast,+}
+(x_2-x_1)\overline{\vec{U}}_{i}^{n}+(x_{i+\frac{1}{2}}^{n+1}-x_2)\vec{U}^{\ast,-}\right).
\end{equation}
\begin{figure}[h!]
\centering
\includegraphics[width=5.0in]{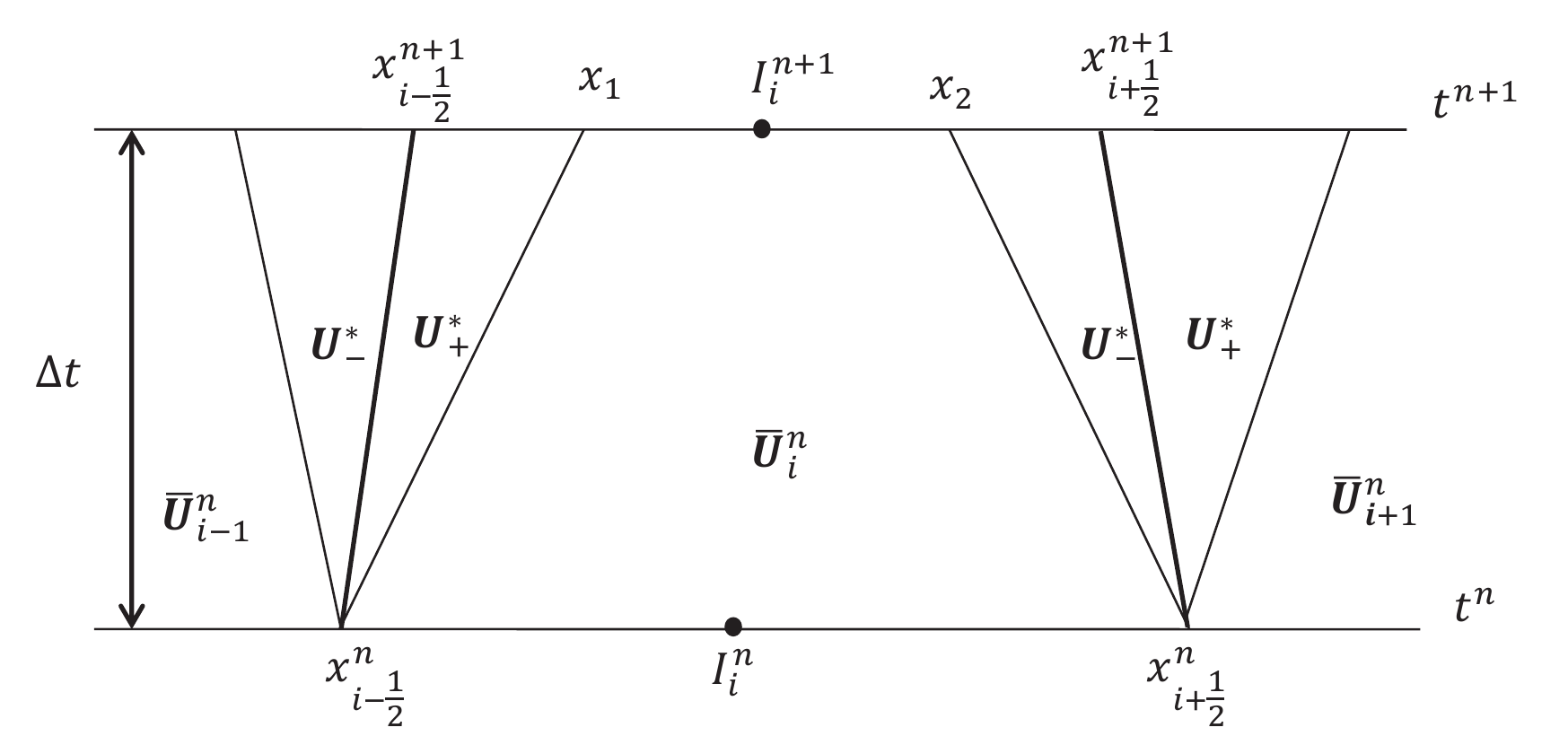}
\caption{Diagrammatic sketch of the neighboring Riemann problems.}
\label{scheme}
\end{figure}
Therefore in order to prove that  the first-order HLLC scheme \eqref{eq:hllc1d} is PCP,
one just needs to show the intermediate states $\vec{U}^{\ast,-},\vec{U}^{\ast,+}\in \mathcal G$ due to the convexity of $\mathcal G$.

\begin{theorem}\label{theorem2.2}
For $\vec{U}^{\pm}\in \mathcal G$ and the wave speeds $s^\pm$ defined in \eqref{ws},
the intermediate states $\vec{U}^{\ast,\pm}$ obtained by the HLLC Riemann solver are PCP,
that is to say,  $\vec{U}^{\ast,\pm}\in \mathcal G$.
\end{theorem}

\begin{proof}
Following \cite{wu2015}, one has to prove
$D^{\ast,\pm}>0$, $E^{\ast,\pm}>0$, and
$(E^{\ast,\pm})^2-(D^{\ast,\pm})^2-(m^{\ast,\pm})^2>0$.

(\romannumeral1)
Due to (\romannumeral4) of Lemma \ref{lem3} and the first equality in \eqref{add7}, it is easy to get $D^{\ast,\pm}>0$.

(\romannumeral2)
For the left or  right state, from (\ref{add7}) and (\ref{eq15}),  one has
\begin{align}
\mathcal{A}^{\pm}&=E^{\ast,\pm}(s^{\pm}-s^{\ast})(1-s^{\pm}s^{\ast})\notag\\
&=(E^{\pm}(s^{\pm}-u^{\pm})+p^{\ast}s^{\ast}-p^{\pm}u^{\pm})(1-s^{\pm}s^{\ast})\notag\\
&=A^{\pm}\left(s^{\ast}-\frac{s^{\pm}A^{\pm}+B^{\pm}}{2A^{\pm}}\right)^2+\frac{4(A^{\pm})^2-(s^{\pm}A^{\pm}+B^{\pm})^2}{4A^{\pm}},\label{eq90}
\end{align}
where $A^{\pm}$ and $B^{\pm}$ are defined in \eqref{eq67}.
Using (\romannumeral1) of Lemma \ref{lem3}  gives
\begin{equation}\label{eq68}
\begin{aligned}
\mathcal{A}^{-}&=A^{-}\left(s^{\ast}-\frac{s^{-}A^{-}+B^{-}}{2A^{-}}\right)^2+\frac{4(A^{-})^2-(s^{-}A^{-}+B^{-})^2}{4A^{-}}
\le\frac{4(A^{-})^2-(s^{-}A^{-}+B^{-})^2}{4A^{-}},\\
\mathcal{A}^{+}&=A^{+}\left(s^{\ast}-\frac{s^{+}A^{+}+B^{+}}{2A^{+}}\right)^2+\frac{4(A^{+})^2-(s^{+}A^{+}+B^{+})^2}{4A^{+}}
\ge\frac{4(A^{+})^2-(s^{+}A^{+}+B^{+})^2}{4A^{+}}.\\
\end{aligned}
\end{equation}
Combining \eqref{eq90},~\eqref{eq68} with (\romannumeral1) and (\romannumeral5) in Lemma \ref{lem3} further  yields
$\mathcal{A}^{-}<0$ and $\mathcal{A}^{+}>0$, which imply
$$E^{\ast,-}>0,~~~E^{\ast,+}>0.$$

(\romannumeral3)
Define
$$\mathcal{B}^{\pm}=\big[(E^{\ast,\pm})^2-(D^{\ast,\pm})^2-(m^{\ast,\pm})^2\big]
(s^{\pm}-s^\ast)^2(A^{\pm}+s^{\pm}p^\ast)^2.$$
Using (\ref{eq15}) gives
\begin{align}
\mathcal{B}^{\pm}&=(A^{\pm}+s^{\pm}p^\ast)^2\big[(A^{\pm}
+p^\ast s^\ast)^2-(D^{\pm})^2(s^{\pm}-u^{\pm})^2-(B^{\pm}+p^\ast)^2\big]\notag\\
&=(A^{\pm}+s^{\pm}p^\ast)^2\bigg((p^\ast)^2(s^\ast)^2+2A^{\pm}p^{\ast}s^{\ast}-(2B^{\pm}+p^{\ast})p^{\ast}+\mathcal{K}^{\pm}\bigg)\notag\\
&=[1-(s^{\pm})^2](p^\ast)^4+2B^{\pm}(1-(s ^{\pm})^2)(p^{\ast})^3
+2s^{\pm}A^{\pm}\mathcal{K}^{\pm}p^\ast+(A^{\pm})^2\mathcal{K}^{\pm}\notag\\
&~~~+\bigg((A^{\pm}-s^{\pm}B^{\pm})^2+(1-(s^{\pm})^2)(B^{\pm})^2+(s^{\pm})^2\mathcal{K}^{\pm}\bigg)(p^\ast)^2\notag\\
&=[1-(s^{\pm})^2](B^{\pm}+p^{\ast})^2+(A^{\pm}-s^{\pm}B^{\pm})^2(p^{\ast})^2
+\mathcal{K}^{\pm}(A^{\pm}+s^{\pm}p^{\ast})^2,
\label{eq25}
\end{align}
where
$$\mathcal{K}^{\pm}=(A^{\pm})^2-(B^{\pm})^2-(D^{\pm})^2(s^{\pm}-u^{\pm})^2.$$
The conclusion (\romannumeral6) in Lemma \ref{lem3} can tells us $\mathcal{K}^{\pm}>0$ and then
$\mathcal{B}^{\pm}>0$, which imply
$$(E^{\ast,\pm})^2-(D^{\ast,\pm})^2-(m^{\ast,\pm})^2>0.$$
The proof is completed.
\end{proof}

Based the above discussion, one can draw the following conclusion.
\begin{theorem}\label{theorem2.3}
If  $\{\overline{\vec{U}}_i^{n}\in \mathcal G,\forall i=1,
\cdots,N\}$ and the wave speeds $s^\pm$ estimated by \eqref{ws},
then $\overline{\vec{U}}_i^{n+1}$ obtained by
the first-order  Lagrangian scheme \eqref{eq8} with
\eqref{eq9} and HLLC solver
belong to the admissible state set $\mathcal G$
under   the following time step restriction
\begin{equation}\label{eq:dt1}
\Delta t^n\le\lambda\min\limits_{i}\left\{\Delta x_i^n/\max(|s_{\min}(\overline{\vec{U}}_i^n)|,
|s_{\max}(\overline{\vec{U}}_i^n)|)\right\},
\end{equation}
where the CFL number $\lambda\le\frac{1}{2}$.
\end{theorem}

\subsection{High-order accurate scheme}

This section will develop the one-dimensional high-order accurate PCP Lagrangian finite volume scheme with the HLLC solver. It consists of three parts: the high-order accurate initial reconstruction,
the scaling PCP limiter, and the high-order accurate
time discretization.

For the known cell-average values $\{\overline{\vec{U}}_{i}^{n}\}$ of the solutions of the RHD equations \eqref{eq1} or \eqref{eq4}  with $d=1$,
by the aid of  the local characteristic decomposition \cite{zhao},
the  WENO reconstruction technique is applied to get the high-order approximations of the point values $\vec{U}_{i-\frac{1}{2}}^{+}$ and
$\vec{U}_{i+\frac{1}{2}}^{-}$,
and then they can be used to give  the HLLC flux $\widetilde{\vec {F}}^{\ast,\pm}$
and   intermediate states $\vec{U}^{\ast,\pm}$.
For a $(K+1)$th-order accurate finite volume WENO scheme, as soon as the point values
$\vec{U}_{i-\frac{1}{2}}^{+}$ and $\vec{U}_{i+\frac{1}{2}}^{-}$ are obtained by the WENO reconstruction at $t=t_n$,
one can also give a polynomial vector $\vec{U}_i^n(x)$ of degree $K$ in principle such that $\vec{U}_i^n(x_{i-\frac{1}{2}})=\vec{U}_{i-\frac{1}{2}}^{+},
\vec{U}_i^n(x_{i+\frac{1}{2}})=\vec{U}_{i+\frac{1}{2}}^{-}$,   its cell average value over $I_i^n$ is equal to $\overline{\vec{U}}_i^n$, and $\vec{U}^n_i(x)$ is a $(K+1)$th-order   accurate approximation
to $\vec U(x,t_n)$ in $I_i^n$. Such a polynomial vector can  be gotten by using the Hermite type reconstruction technique \cite{zhang2} and satisfy exactly
the $L$-point  Gauss-Lobatto quadrature rule with $2L-3\ge K$, i.e.
\begin{equation}\label{add100}
\overline{\vec{U}}_i^n
=\frac{1}{\Delta x_i^n}\int_{x_{i-\frac{1}{2}}^n}^{x_{i+\frac{1}{2}}^n}\vec{U}_i^n(x)dx
=\sum\limits_{\alpha=1}^{L}\omega_\alpha\vec{U}_{i}^n(x_i^{\alpha}),
\end{equation}
which gives a split of $\overline{\vec{U}}_i^n$, where $\omega_{1},\cdots, \omega_{L}$ are the quadrature weights
satisfying $\omega_1=\omega_L>0$. Practically, it does not need to explicitly obtain such a polynomial vector since
the mean value theorem tell us that there exists some $x^{\ast}\in I_i^n$
such that $\vec{U}_i^n(x^{\ast})=\frac{1}{1-2\omega_1}\sum\limits_{\alpha=2}^{L-1}\omega_\alpha\vec{U}_{i}^n(x_i^{\alpha})=:\vec{U}^{\ast\ast}_i$
\cite{zhang1}. At this time, the split \eqref{add100} can be simply replaced with
\begin{equation}\label{eq2.25}
\overline{\vec{U}}_i^n
=\omega_1\vec{U}_{i-\frac{1}{2}}^{+}+\omega_1\vec{U}_{i+\frac{1}{2}}^{-}+
(1-2\omega_1)\vec{U}_i^{\ast\ast},
\end{equation}
and $\vec{U}^{\ast\ast}_i$ can be calculated by
\begin{equation}
\vec{U}^{\ast\ast}_i=\frac{\overline{\vec{U}}_i^n-\omega_1\vec{U}_{i-\frac{1}{2}}^{+}-\omega_1\vec{U}_{i+\frac{1}{2}}^{-}}{1-2\omega_1}.
\end{equation}

By adding and subtracting the term $\Delta t^n\widehat{\vec{\mathcal{F}}}(\vec{U}_{i-\frac{1}{2}}^{+},\vec{U}_{i+\frac{1}{2}}^{-})$ and using
the split \eqref{eq2.25},
the  scheme \eqref{eq8} with high-order WENO reconstruction can be reformulated as follows
\begin{equation}
\begin{aligned}
\overline{\vec{U}}_i^{n+1}\Delta x_i^{n+1}
&=\big(\omega_1\vec{U}_{i-\frac{1}{2}}^{+}+\omega_1\vec{U}_{i+\frac{1}{2}}^{-}+(1-2\omega_1)\vec{U}_i^{\ast\ast}\big)\Delta x_i^n\\
&~~~-\Delta t^n\left(\widehat{\vec{\mathcal{F}}}(\vec{U}_{i+\frac{1}{2}}^{-},\vec{U}_{i+\frac{1}{2}}^{+})
-\widehat{\vec{\mathcal{F}}}(\vec{U}_{i-\frac{1}{2}}^{+},\vec{U}_{i+\frac{1}{2}}^{-})
+\widehat{\vec{\mathcal{F}}}(\vec{U}_{i-\frac{1}{2}}^{+},\vec{U}_{i+\frac{1}{2}}^{-})
-\widehat{\vec{\mathcal{F}}}(\vec{U}_{i-\frac{1}{2}}^{-},\vec{U}_{i-\frac{1}{2}}^{+})\right)\\
&=(1-2\omega_1)\vec{U}_i^{\ast\ast}\Delta x_i^n
+\omega_1\left\{\vec{U}_{i-\frac{1}{2}}^{+}\Delta x_i^n-\frac{\Delta t^n}{\omega_1}
\left(\widehat{\vec{\mathcal{F}}}(\vec{U}_{i-\frac{1}{2}}^{+},\vec{U}_{i+\frac{1}{2}}^{-})
-\widehat{\vec{\mathcal{F}}}(\vec{U}_{i-\frac{1}{2}}^{-},\vec{U}_{i-\frac{1}{2}}^{+})\right)\right\}\\
&~~~+\omega_1\left\{\vec{U}_{i+\frac{1}{2}}^{-}\Delta x_i^n-\frac{\Delta t^n}{\omega_1}
\left(\widehat{\vec{\mathcal{F}}}(\vec{U}_{i+\frac{1}{2}}^{-},\vec{U}_{i+\frac{1}{2}}^{+})
-\widehat{\vec{\mathcal{F}}}(\vec{U}_{i-\frac{1}{2}}^{+},\vec{U}_{i+\frac{1}{2}}^{-})\right)\right\}\\
&=(1-2\omega_1)\vec{U}_i^{\ast\ast}\Delta x_i^n+\omega_1\mathcal{H}_1+\omega_1\mathcal{H}_L,
\end{aligned}
\end{equation}
where
\begin{equation}
\begin{aligned}
\mathcal{H}_1&=\vec{U}_{i-\frac{1}{2}}^{+}\Delta x_i^n-\frac{\Delta t^n}{\omega_1}
\left(\widehat{\vec{\mathcal{F}}}(\vec{U}_{i-\frac{1}{2}}^{+},\vec{U}_{i+\frac{1}{2}}^{-})
-\widehat{\vec{\mathcal{F}}}(\vec{U}_{i-\frac{1}{2}}^{-},\vec{U}_{i-\frac{1}{2}}^{+})\right),\\
\mathcal{H}_L&=\vec{U}_{i+\frac{1}{2}}^{-}\Delta x_i^n-\frac{\Delta t^n}{\omega_1}
\left(\widehat{\vec{\mathcal{F}}}(\vec{U}_{i+\frac{1}{2}}^{-},\vec{U}_{i+\frac{1}{2}}^{+})
-\widehat{\vec{\mathcal{F}}}(\vec{U}_{i-\frac{1}{2}}^{+},\vec{U}_{i+\frac{1}{2}}^{-})\right).
\end{aligned}
\end{equation}
Because $\mathcal{H}_1$ and $\mathcal{H}_L$ do exactly mimic the first-order scheme
 \eqref{eq8} with (\ref{eq9}) and the HLLC solver,
 one can know $\mathcal{H}_1,\mathcal{H}_L\in \mathcal G$
 if the two boundary values $\vec{U}_{i-\frac{1}{2}}^{+},\vec{U}_{i+\frac{1}{2}}^{-}\in \mathcal G$,  the wave speeds $s^\pm$ estimated by \eqref{ws}, and the time stepsize  $\Delta t^n$ satisfies the  restriction
\begin{equation}\label{eq56}
\Delta t^n\le\lambda~\omega_1\min\limits_{i}
\left\{\Delta x_i^n/\max\limits(|s_{\min}(\vec{U}_{i\pm\frac{1}{2}}^{\mp})|,|s_{\max}(\vec{U}_{i\pm\frac{1}{2}}^{\mp})|)\right\},
\end{equation}
with $\lambda\le\frac{1}{2}$.
Therefore, besides the time step restriction  \eqref{eq56},
the  sufficient condition for
$\overline{\vec{U}}_i^{n+1}\in  \mathcal G$ is
\begin{equation}\label{eq57}
\vec{U}_{i-\frac{1}{2}}^{+},~\vec{U}_{i+\frac{1}{2}}^{-},~\vec{U}_{i}^{\ast\ast}\in \mathcal G, \quad \forall i=1,\dots, N,
\end{equation}
which can be ensured by using a scaling limiter, presented in the next
subsection. Hence  the aforementioned results can be summarized as follows.

\begin{theorem}
If $\vec{U}_{i-\frac{1}{2}}^{+},\vec{U}_{i+\frac{1}{2}}^{-},\vec{U}_{i}^{\ast\ast}\in \mathcal G$
and the wave speeds $s^\pm$ estimated by \eqref{ws},
then $\overline{\vec{U}}_i^{n+1}$ obtained by the high-order
Lagrangian scheme \eqref{eq8} with
the WENO reconstruction and the HLLC solver belongs to the admissible
state set $\mathcal G$ under the time stepsize restriction \eqref{eq56} with $\lambda\le\frac{1}{2}$.
\end{theorem}

\subsubsection{Scaling PCP limiter}\label{subsection2.2.1}

This section uses  the scaling PCP limiter, which has been used in \cite{cheng,zhang,wu2017},
to limit $\vec{U}_{i-\frac{1}{2}}^{+},\vec{U}_{i+\frac{1}{2}}^{-},\vec{U}_{i}^{\ast\ast}$
such that the limited values $\widetilde{\vec{U}}_{i-\frac{1}{2}}^{+},\widetilde{\vec{U}}_{i+\frac{1}{2}}^{-},
\widetilde{\vec{U}}_i^{\ast\ast}$ belong to $\mathcal G$ when  $\overline{\vec{U}}_i^{n}\in \mathcal{G}$.
For the sake of brevity, the superscript $n$ will be omitted in this section and a small parameter
 $\varepsilon$ is taken as $10^{-13}$.
Such limiting procedure can be implemented as follows.

First, let us enforce the positivity of the mass density.
For each cell $I_i$, define the limiter
\begin{equation}
\theta_i^1=\min\left\{1,\frac{\overline{D}_i-\varepsilon}
{\overline{D}_i-D_{\min}}\right\},~~D_{\min}=\min(D_{i-\frac{1}{2}}^{+},D_i^{\ast\ast},D_{i+\frac{1}{2}}^{-}),
\end{equation}
and  then limit the  mass density as follows
\begin{equation}\label{eq60}
\widehat{D}_{i\pm\frac{1}{2}}^{\mp}=\overline{D}_i+\theta_{i}^{1}(D_{i\pm\frac{1}{2}}^{\mp}-\overline{D}_i),
~~\widehat{D}_{i}^{\ast\ast}=\overline{D}_i+\theta_{i}^{1}(D_{i}^{\ast\ast}-\overline{D}_i).
\end{equation}
Define
$\widehat{\vec{U}}_{i\pm\frac{1}{2}}^{\mp}=(\widehat{D}_{i\pm\frac{1}{2}}^{\mp},
m_{i\pm\frac{1}{2}}^{\mp},E_{i\pm\frac{1}{2}}^{\mp})^{\mathrm{T}}$
and
$\widehat{\vec{U}}_i^{\ast\ast}=(\widehat{D}_i^{\ast\ast},
m_i^{\ast\ast},E_i^{\ast\ast})^{\mathrm{T}}$.

Then, enforce the positivity of  $q(\vec{U})=E-\sqrt{D^2+m^2}$.  For each cell $I_i$, define the limiter
\begin{equation}\label{eq61}
\theta_i^2=\min\left\{1,\frac{q(\overline{\vec{U}}_i)
-\varepsilon}{q(\overline{\vec{U}}_i)-q_{\min}}\right\},\quad
q_{\min}=\min\left(q(\widehat{\vec{U}}_{i-\frac{1}{2}}^{+}),q(\widehat{\vec{U}}_i^{\ast\ast}),
q(\widehat{\vec{U}}_{i+\frac{1}{2}}^{-})\right),
\end{equation}
and then limit the conservative vectors as follows
\begin{equation}\label{eq62}
\widetilde{\vec{U}}_{i\pm\frac{1}{2}}^{\mp}=\overline{\vec{U}}_i+\theta_i^2(\widehat{\vec{U}}_{i\pm\frac{1}{2}}^{\mp}-\overline{\vec{U}}_i),
~~\widetilde{\vec{U}}_{i}^{\ast\ast}=\overline{\vec{U}}_i+\theta_i^2(\widehat{\vec{U}}_{i}^{\ast\ast}-\overline{\vec{U}}_i).
\end{equation}

It is easy to check that
$\widetilde{\vec{U}}_{i-\frac{1}{2}}^{+},\widetilde{\vec{U}}_{i+\frac{1}{2}}^{-},
\widetilde{\vec{U}}_i^{\ast\ast}\in \mathcal G$.
Moreover, the above scaling PCP limiter does not destroy the original high order accuracy
in smooth regions, see  the detailed discussion in \cite{zhang}.

\subsubsection{High-order time discretization}\label{subsubsection2.2.2}

To get a globally high-order accurate scheme in time and space,
we can further employ the strong stability preserving (SSP) Runge-Kutta (RK) method to replace the explicit Euler time discretization in  \eqref{eq8} and \eqref{eq8-2}.
Similar to  \cite{cheng},
for instance, to obtain a third-order accurate scheme in time,
a third-stage SSP, explicit RK method may be used
for the time discretization as follows.

Stage 1:
\begin{equation}\label{eq63}
\begin{aligned}
&x_{i+\frac{1}{2}}^{(1)}=x_{i+\frac{1}{2}}^n+\Delta t^nu_{i+\frac{1}{2}}^{\ast},
~~~\Delta x_{i}^{(1)}=x_{i+\frac{1}{2}}^{(1)}-x_{i-\frac{1}{2}}^{(1)},\\
&\overline{\vec{U}}_i^{(1)}\Delta x_{i}^{(1)}
=\overline{\vec{U}}_i^{n}\Delta x_{i}^{n}-\Delta t^n\mathcal{L}(\overline{\vec{U}}^n;i);\\
\end{aligned}
\end{equation}
\par
Stage 2:
\begin{equation}\label{eq64}
\begin{aligned}
&x_{i+\frac{1}{2}}^{(2)}=\frac{3}{4}x_{i+\frac{1}{2}}^n+
\frac{1}{4}\big(x_{i+\frac{1}{2}}^{(1)}+\Delta t^nu_{i+\frac{1}{2}}^{(1),\ast}\big),
~~~\Delta x_i^{(2)}=x_{i+\frac{1}{2}}^{(2)}-x_{i-\frac{1}{2}}^{(2)},\\
&\overline{\vec{U}}_i^{(2)}\Delta x_{i}^{(2)}=\frac{3}{4}\overline{\vec{U}}_i^{n}\Delta x_{i}^{n}
+\frac{1}{4}\left(\overline{\vec{U}}_i^{(1)}\Delta x_{i}^{(1)}-\Delta t^n\mathcal{L}(\overline{\vec{U}}^{(1)};i)\right);
\end{aligned}
\end{equation}
\par
Stage 3:
\begin{equation}\label{eq65}
\begin{aligned}
&x_{i+\frac{1}{2}}^{n+1}=\frac{1}{3}x_{i+\frac{1}{2}}^n
+\frac{2}{3}\big(x_{i+\frac{1}{2}}^{(2)}+\Delta t^nu_{i+\frac{1}{2}}^{(2),\ast}\big),
~~~\Delta x_i^{n+1}=x _{i+\frac{1}{2}}^{n+1}-x_{i-\frac{1}{2}}^{n+1},\\
&\overline{\vec{U}}_i^{n+1}\Delta x_{i}^{n+1}=\frac{1}{3}\overline{\vec{U}}_i^{n}\Delta x_{i}^{n}
+\frac{2}{3}\left(\overline{\vec{U}}_i^{(2)}\Delta x_{i}^{(2)}-\Delta t^n\mathcal{L}(\overline{\vec{U}}^{(2)};i)\right),
\end{aligned}
\end{equation}
where $\mathcal{L}(\overline{\vec{U}};i)=\widehat{\vec{\mathcal{F}}}(\vec{U}_{i+\frac{1}{2}}^{-},\vec{U}_{i+\frac{1}{2}}^{+})-
\widehat{\vec{\mathcal{F}}}(\vec{U}_{i-\frac{1}{2}}^{-},\vec{U}_{i-\frac{1}{2}}^{+})$.

The  scaling PCP limiter described in Section \ref{subsection2.2.1}
needs to be performed at each stage of the above RK method
to limit the value of ${\vec{U}}_{i-\frac{1}{2}}^{+}, {\vec{U}}_{i+\frac{1}{2}}^{-},
{\vec{U}}_i^{\ast\ast}$. Because each stage of the above SSP RK method is a convex
combination of the forward Euler time discretization and $\mathcal G$ is convex,
so is the above SSP RK method when the forward Euler method is conservative, stable and PCP.

\section{Two-dimensional PCP Lagrangian schemes}\label{section:2D-PCP}
This section considers the Lagrangian finite volume  schemes for the special relativistic hydrodynamics (RHD) equations \eqref{eq1} or \eqref{eq4}  with $d=2$. Here
the notations $(m_1,m_2)$, $(u_1,u_2)$ and $(x_1,x_2)$ are replaced with $\vec m=(m_x,m_y)$, $\vec u=(u_x,u_y)$ and $\vec x=(x,y)$, respectively.

Assume that the time interval $\{t>0\}$ is divided into:
$t_{n+1}=t_n+\Delta t^n$, $n=0,1,2,\cdots$,
 where the time step size $\Delta t^n$  will be determined
 by the CFL type condition.
The (dynamic) computational domain $\Omega$ at $t=t_n$  is partitioned into $N_x\times N_y$  quadrilateral cells: $I_{ij}^n$  with the boundary $\partial I_{ij}^n$ and four vertices
$\{(x_{i\pm\frac{1}{2}}^n,y_{j\pm\frac{1}{2}}^n),(x_{i\pm\frac{1}{2}}^n,y_{j\mp\frac{1}{2}}^n)\}$, $i=1,\cdots, N_x,j=1,\cdots, N_y$; and then
the conservative Lagrangian finite volume
scheme with the forward Euler time discretization
for the governing equations \eqref{eq1} with
$d=2$ can be given as follows
\begin{align}\nonumber
\overline{\vec{U}}_{ij}^{n+1}A_{ij}^{n+1}&=\overline{\vec{U}}_{ij}^{n}A_{ij}^{n}
-\Delta t^n\int_{\partial I_{ij}^n}\widehat{\vec{\mathcal{F}}}_{\vec{n}_{ij}}(\vec{U}^{-}_{ij},\vec{U}^{+}_{ij})ds\\
&=\overline{\vec{U}}_{ij}^{n}A_{ij}^{n}-\Delta t^n\sum\limits_{m=1}^4\int_{\partial I_{ij}^{n,m}}
\widehat{\vec{\mathcal{F}}}_{\vec{n}_{ij}^m}(\vec{U}_{m,ij}^{-},\vec{U}_{m,ij}^{+})ds,
 \label{eq72}\end{align}
where $\overline{\vec{U}}^n_{ij}=(\overline{D}_{ij}^n,\overline{\vec{m}}_{ij}^n,\overline{E}_{ij}^n)^{T}$
is the cell average approximation of $\vec U$ at $t_n$ over the cell $I^n_{ij}$,
$A^n_{ij}=\int_{I^n_{ij}}dxdy$ denotes  the area of  $I^n_{ij}$,
$\vec{U}^{-}_{ij}=(D^{-}_{ij},\vec{m}^{-}_{ij},E^{-}_{ij})^T$
and $\vec{U}^{+}_{ij}=(D^{+}_{ij},\vec{m}^{+}_{ij},E^{+}_{ij})^T$
are the reconstructed limits of $\vec U$
on the boundary $\partial I_{ij}^n$ from the inside and outside of the cell $I^n_{ij}$,  respectively; $\partial I^{n,m}_{ij}$ is the $m$th cell edge of $\partial I_{ij}^n$,
and $|l_{ij}^m|$ and $\vec{n}^m_{ij}$ are its  length and unit outward normal vector from the inside  to the outside of $I^{n}_{ij}$, respectively.
Here $\widehat{\vec{\mathcal{F}}}$ denotes the numerical flux, satisfying
the consistency condition $\widehat{\vec{\mathcal{F}}}_{\vec{n}}(\vec{U},\vec{U})=
(0,p\vec{n}^T,p\vec{u}\cdot\vec{n}^T)^T$ and the conservation property
$\widehat{\vec{\mathcal{F}}}_{\vec{n}}(\vec{U}^{-},\vec{U}^{+})=
-\widehat{\vec{\mathcal{F}}}_{-\vec{n}}(\vec{U}^{+},\vec{U}^{-})$.
 The  flux $\widehat{\vec{\mathcal{F}}}_{\vec{n}_{ij}^m}$ will be
obtained by solving the local 1D Riemann problem along the  vector normal to the edge $\partial I^{n,m}_{ij}$ with the HLLC Riemann solver.
Different from the one-dimensional numerical flux $\widehat{\vec{\mathcal F}}(\vec U^-,\vec U^+)$ presented in Section \ref{section:1D-PCP}, the flux $\widehat{\vec{\mathcal{F}}}_{\vec{n}_{ij}^m}(\vec U^-,\vec U^+)$ has to be obtained by the HLLC Riemann solver of the split 2D system for \eqref{eq1} with $d=2$
 in a moving coordinate system.
 The readers are referred to \cite{yz2} for the emphasis on the differences between
 the system \eqref{eq1} with $d=1$ and the split 2D system for \eqref{eq1} with $d=2$.
In view of that the  current HLLC Riemann solver is essentially the same as that in Section \ref{section:1D-PCP} except for the special attention to the nonlinear coupling from the tangential velocity, some details of the  current HLLC Riemann solver will be omitted here.

Let $\vec{U}^{\mp}=(D^{\mp},\vec{m}^{\mp},E^{\mp})^T$
be  the left and right states of the local Riemann problem of the
split 2D system for \eqref{eq1} with $d=2$ in the direction $\vec n$ as shown in Figure \ref{cell-2d}.
At present,
the wave structure of the exact or HLLC approximate solution of the local Riemann problem
is the same as  that in Figure \ref{fig-hllc}.
Two constant intermediate states between the left and right  acoustic waves are
denoted by  $\vec{U}^{\ast,\pm}=(D^{\ast,\pm},\vec{m}^{\ast,\pm},E^{\ast,\pm})^T$,
respectively,
and the fluxes
are denoted by $\widetilde{\vec F}^{\ast,\pm}_{\vec{n}}$,
then for the left or right wave, the Rankine-Hugoniot conditions
$\widetilde{\vec {F}}^{\ast,\pm}_{\vec{n}}=\widetilde{\vec {F}}_{\vec{n}}(\vec U^{\pm})-(s^{\pm}-s^{\ast})(\vec{U}^{\pm}-\vec{U}^{\ast,\pm})$
can be similarly  given in the component form
\begin{align}
\begin{aligned}
&D^{\ast,\pm}(s^{\pm}-s^{\ast})=D^{\pm}(s^{\pm}-u_n^{\pm}),\\
&m_n^{\ast,\pm}(s^{\pm}-s^{\ast})=m_n^{\pm}(s^{\pm}-u_n^{\pm})+p^{\ast}-p^{\pm},\\
&m_{\tau}^{\ast,\pm}(s^{\pm}-s^{\ast})=m_{\tau}^{\pm}(s^{\pm}-u_n^{\pm}),\\
&E^{\ast,\pm}(s^{\pm}-s^{\ast})=E^{\pm}(s^{\pm}-u_n^{\pm})+p^{\ast}s^{\ast}-p^{\pm}u_n^{\pm},
\end{aligned}\label{eq28}\end{align}
where $m_n^{\pm}$ and $m_{\tau}^{\mp}$ are the normal and tangential
components of $\vec{m}^{\mp}$, i.e. $m_n=\vec{m}^T\cdot\vec{n}$ and $m_{\tau}=\vec{m}^T\cdot\boldsymbol{\tau}$,
$u_n=m_n/(E+p)$ is the normal component of the velocity vector,
and the wave speeds $s^{\pm}$ are estimated as follows
\begin{equation}\label{eq30}
s^{-}=\min\left(s_{\min}(\vec{U}^{-}),s_{\min}(\vec{U}^{+})\right),
~~s^{+}=\max\left(s_{\max}(\vec{U}^{-}),s_{\max}(\vec{U}^{+})\right),
\end{equation}
with
\begin{equation}\label{ws2d}
s_{\min}(\vec{U})=\frac{u_n-\sqrt{\sigma_s(1-u_n^2+\sigma_s)}}{1+\sigma_s},~~~
s_{\max}(\vec{U})=\frac{u_n+\sqrt{\sigma_s(1-u_n^2+\sigma_s)}}{1+\sigma_s},
\end{equation}
and $\sigma_s=c_s^2/[\gamma^2(1-c_s^2)]$.
\begin{figure}[h!]
\centering
\includegraphics[width=2.8in]{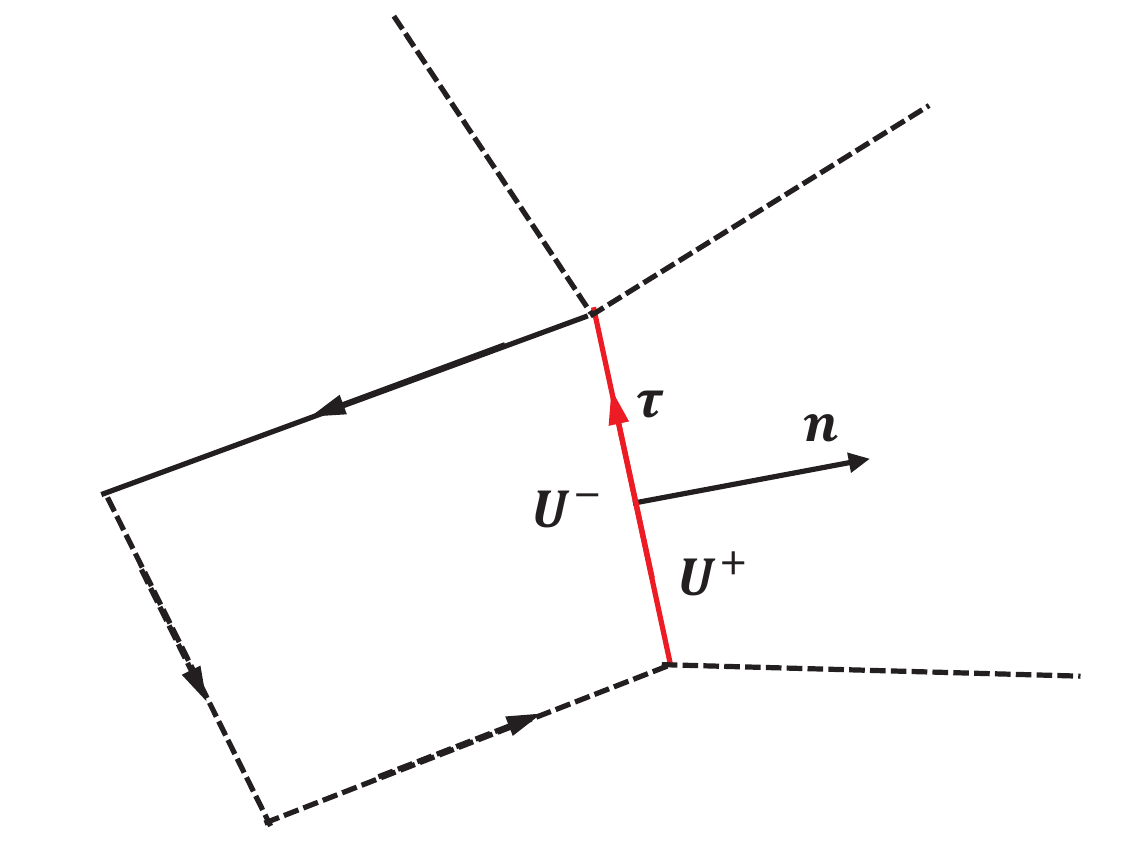}
\caption{A general quadrilateral cell. }
\label{cell-2d}
\end{figure}
Combining the last equation of \eqref{eq28} with the second one gives
the  expression of $s^{\ast}$ in terms
of $p^{\ast}$ as follows
\begin{equation}\label{eq29}
(A^{-}+s^{-}p^{\ast})s^{\ast}=B^{-}+p^{\ast},~~(A^{+}+s^{+}p^{\ast})s^{\ast}=B^{+}+p^{\ast},
\end{equation}
where
\begin{equation}\label{eq91}
A^{\pm}=s^{\pm}E^{\pm}-m_n^{\pm},~~~B^{\pm}=m_n^{\pm}(s^{\pm}-u_n^{\pm})-p^{\pm}.
\end{equation}
Similar to \eqref{eq17},  solving the system \eqref{eq29} gives
\begin{align}\label{eq-contact-speed}
s^{\ast}=&\frac{-\mathcal{C}_1-\sqrt{\mathcal{C}_1^2-4\mathcal{C}_0\mathcal{C}_2}}{2\mathcal{C}_2},\\
p^{\ast}=&\frac{s^{\ast}A^{-}-B^{-}}{1-s^{-}s^{\ast}}=\frac{s^{\ast}A^{+}-B^{+}}{1-s^{+}s^{\ast}},
\end{align}
where
\begin{equation}
\mathcal{C}_0=B^{+}-B^{-},~~\mathcal{C}_1=A^{-}+s^{+}B^{-}-A^{+}-s^{-}B^{+},
~~\mathcal{C}_2=s^{-}A^{+}-s^{+}A^{-}.
\end{equation}
Based on those,  the HLLC intermediate states $\vec{U}^{\ast,\pm}$
 can be gotten from \eqref{eq28}  and then
 the HLLC fluxes $\widetilde{\vec F}^{\ast,\pm}_{\vec{n}}$
 (approximating $\widetilde{\vec{F}}_{\vec{n}}(\vec{U})$) and the numerical flux in the Lagrangian framework
$\widehat{\vec{\mathcal{F}}}_{\vec{n}}=\widetilde{\vec {F}}^{\ast,+}_{\vec{n}}=\widetilde{\vec {F}}^{\ast,-}_{\vec{n}}$.

Similar to Lemma \ref{lem3},  the following conclusion can be drawn.
\begin{lemma}\label{lem4} For $\vec{U}^{\pm}=(D^{\pm},\vec{m}^{\pm},E^{\pm})^{T}$ and the wave speeds $s^{\pm}$    estimated in \eqref{eq30}, if  $s^{\ast}$ is the speed of the contact discontinuity and $A^{\pm},B^{\pm}$ are defined by \eqref{eq91}, then one has
\begin{enumerate}[(\romannumeral1)]
\item $A^{-}<0,~~A^{+}>0$;
\item $A^{-}-B^{-}<0,~~A^{+}+B_{+}>0$;
\item $s^{-}A^{-}-B^{-}>0,~~s^{+}A^{+}-B^{+}>0$;
\item $s^{+}A^{-}-B^{-}<0,~~s^{-}A^{+}-B^{+}<0$;
\item $s^{-}<u_n^{\pm}<s^{+},~~s^{-}<s^{\ast}<s^{+}$;
\item $4(A^{\pm})^2-(s^{\pm}A^{\pm}+B^{\pm})^2>0$;
\item $(A^{\pm})^2-(B^{\pm})^2-(D^{\pm})^2(s^{\pm}-u_n^{\pm})^2-(m_{\tau}^{\pm})^2(s^{\pm}-u_n^{\pm})^2>0$.
\end{enumerate}
\end{lemma}
\noindent
The proof is given in Appendix  \ref{appendix-B}. Based on this result, one can further get
the following theorem, which corresponds to  Theorem \ref{theorem2.2}.
\begin{theorem}
For $\vec{U}^{\pm}\in \mathcal G$, and the wave speeds $s^\pm$ defined in \eqref{eq30},
the intermediate states $\vec{U}^{\ast,\pm}$ obtained by the above HLLC Riemann solver
are PCP, that is to say, $\vec{U}^{\ast,\pm}\in \mathcal G$.
\end{theorem}

\begin{proof}Similarly,   one has to prove
$D^{\ast,\pm}>0$, $E^{\ast,\pm}>0$, and
$(E^{\ast,\pm})^2-(D^{\ast,\pm})^2-|\vec{m}^{\ast,\pm}|^2>0$.

(\romannumeral1) Due to the first equation in \eqref{eq28} and (\romannumeral5) in Lemma \ref{lem4}, it is easy to know that
$D^{\ast,\pm}>0$.

(\romannumeral2) According to (\ref{eq28}) and (\ref{eq29}), for the left or the right state one has
\begin{align}
\mathcal{A}^{\pm}&=E^{\ast,\pm}(s^{\pm}-s^{\ast})(1-s^{\pm}s^{\ast})\notag\\
&=(E^{\pm}(s^{\pm}-u_n^{\pm})+p^{\ast}s^{\ast}-p^{\pm}u_n^{\pm})(1-s^{\pm}s^{\ast})\notag\\
&=(A^{\pm}+p^{\ast}s^{\ast})(1-s^{\pm}s^{\ast})
=A^{\pm}\left(s^{\ast}-\frac{s^{\pm}A^{\pm}+B^{\pm}}{2A^{\pm}}\right)^2+\frac{4(A^{\pm})^2-(s^{\pm}A^{\pm}+B^{\pm})^2}{4A^{\pm}}.\label{eq92}
\end{align}
Using the conclusion (\romannumeral1) in   Lemma \ref{lem4} gives
\begin{equation}\label{eq71}
\begin{aligned}
\mathcal{A}^{-}&=A^{-}\left(s^{\ast}-\frac{s^{-}A^{-}+B^{-}}{2A^{-}}\right)^2+\frac{4(A^{-})^2-(s^{-}A^{-}+B^{-})^2}{4A^{-}}
\le\frac{4(A^{-})^2-(s_{-}A_{-}+B_{-})^2}{4A^{-}},\\
\mathcal{A}^{+}&=A^{+}\left(s^{\ast}-\frac{s^{+}A^{+}+B^{+}}{2A^{+}}\right)^2+\frac{4(A^{+})^2-(s^{+}A^{+}+B^{+})^2}{4A^{+}}
\ge\frac{4(A^{+})^2-(s^{+}A^{+}+B^{+})^2}{4A^{+}}.\\
\end{aligned}
\end{equation}
Combining \eqref{eq92}, \eqref{eq71} with Lemma \ref{lem4}  can easily show
$\mathcal{A}^{-}<0$ and $\mathcal{A}^{+}>0$,
which implies $E^{\ast,-}>0$ and $E^{\ast,+}>0$.

(\romannumeral3) Define
$\mathcal{B}^{\pm}=\big[(E^{\ast,\pm})^2-(D^{\ast,\pm})^2-|\vec{m}^{\ast,\pm}|^2\big]
(s^{\pm}-s^\ast)^2(A^{\pm}+s^{\pm}p^\ast)^2$,
and then use  \eqref{eq28} and \eqref{eq29} to give
\begin{align*}
\mathcal{B}^{\pm}&=(A^{\pm}+s^{\pm}p^\ast)^2\big[(A^{\pm}
+p^\ast s^\ast)^2-(B^{\pm}+p^\ast)^2-(D^{\pm})^2(s^{\pm}-u_n^{\pm})^2-(m_{\tau}^{\pm})^2(s^{\pm}-u_n^{\pm})^2\big] \\
&=(A^{\pm}+s^{\pm}p^\ast)^2\bigg((p^\ast)^2(s^\ast)^2+2A^{\pm}p^\ast s^{\ast}-(2B^{\pm}+p^\ast)p^\ast+\mathcal{K}^{\pm}\bigg) \\
&=[1-(s^{\pm})^2](p^\ast)^4+2B^{\pm}[1-(s^{\pm})^2](p^{\ast})^3+2s^{\pm}A^{\pm}
\mathcal{K}^{\pm}p^\ast+(A^{\pm})^2\mathcal{K}^{\pm} \\
&~~+\bigg((A^{\pm}-s^{\pm}B^{\pm})^2+[1-(s^{\pm})^2](B^{\pm})^2+(s^{\pm})^2\mathcal{K}^{\pm}\bigg)(p^\ast)^2\notag\\
&=[1-(s^{\pm})^2](p^\ast)^2(B^{\pm}+p^{\ast})^2+(A^{\pm}-s^{\pm}B^{\pm})^2(p^{\ast})^2
+\mathcal{K}^{\pm}(A^{\pm}+s^{\pm}p^{\ast})^2,
\end{align*}
where
$$\mathcal{K}^{\pm}=(A^{\pm})^2-(B^{\pm})^2-(D^{\pm})^2(s^{\pm}-u_n^{\pm})^2-(m_{\tau}^{\pm})^2(s^{\pm}-u_n^{\pm})^2.$$
Finally, using (\romannumeral7) in Lemma \ref{lem4} gives $\mathcal{K}^{\pm}>0$. Thus one has
$\mathcal{B}^{\pm}>0$, which implies $(E^{\ast,\pm})^2-(D^{\ast,\pm})^2-|\vec{m}^{\ast,\pm}|^2>0$.
The proof is completed.
\end{proof}

\subsection{First-order accurate scheme}

For the first-order accurate Lagrangian scheme,
$\vec{U}_{m,ij}^{-}$ and $\vec{U}_{m,ij}^{+}$ in \eqref{eq72} are taken as the cell average values
of the conservative vector $\vec U$ in the cell  $I_{ij}^n$ and its neighboring cell
sharing the cell edge $\partial I_{ij}^{n,m}$,
$m=1,2,3,4$. The scheme \eqref{eq4} becomes
\begin{equation}\label{eq73}
\overline{\vec{U}}_{ij}^{n+1}A_{ij}^{n+1}=\overline{\vec{U}}_{ij}^{n}A_{ij}^{n}-\Delta t^n\sum\limits_{m=1}^4
\widehat{\vec{\mathcal{F}}}_{\vec{n}_{ij}^m}(\overline{\vec{U}}_{ij}^{n},\overline{\vec{U}}_m^{\text{ext}(I^n_{ij})})|l_{ij}^m|,
\end{equation}
where $\overline{\vec{U}}_m^{\text{ext}(I^n_{ij})}$ is the cell average of $\vec{U}$ over the neighboring cell
of $I^n_{ij}$
sharing the  edge $\partial I^{n,m}_{ij}$ with $I^n_{ij}$.
The vertices of the cell  will be evolved by
\begin{equation}\label{eq8-2b}
\begin{aligned}
&\vec x_{i+\frac{1}{2},j+\frac12}^{n+1}=\vec x_{i+\frac{1}{2},j+\frac12}^{n}+\Delta t^n\vec u_{i+\frac{1}{2},j+\frac12}^{n},
\end{aligned}
\end{equation}
where the velocity $\vec u_{i+\frac{1}{2},j+\frac12}^{n}$
is calculated as follows:
$$
\vec u_{i+\frac{1}{2},j+\frac12}^{n}=\frac14 \left(
  \vec u_{i+\frac{1}{2},j+1}^{\ast}+\vec u_{i+\frac{1}{2},j}^{\ast}
+\vec u_{i+1,j+\frac12}^{\ast}+\vec u_{i,j+\frac12}^{\ast} \right),
$$
where $\vec u_{i+\frac{1}{2},\ell}^{\ast}$ and $\vec u_{\ell',j+\frac12}^{\ast}$,
  $\ell=j,j+1$, $\ell'=i,i+1$, are  fluid velocities at midpoints of four cell edges
  with  a common vertex (e.g. the point P in Figure \ref{cell-2d}), respectively.
Here take the calculation of $\vec u_{i+\frac{1}{2},j}^{\ast}$ as an example.
The velocity $\vec u_{i+\frac{1}{2},j}^{\ast}$ is gotten by using the local rotation transformation
of $(u_n,u_\tau)_{i+\frac{1}{2},j}^{\ast}$, where
$(u_n)_{i+\frac{1}{2},j}^{\ast}=s_{i+\frac{1}{2},j}^{\ast}$,  $(u_\tau)_{i+\frac{1}{2},j}^{\ast}=\frac{1}{2}
\left( (u_\tau)_{i,j}^{n}+(u_\tau)_{i+1,j}^{n}\right)$, and $s_{i+\frac{1}{2},j}^{\ast}$ is the
speed of contact discontinuity in the HLLC solver, see \eqref{eq-contact-speed}.

Because the flux $\vec{\mathcal{F}}_{\vec{n}}(\vec{U})$ in \eqref{eq4} can be written as follows
\begin{equation*}
\vec{\mathcal{F}}_{\vec{n}}(\vec{U})=\begin{pmatrix}
0\\p\vec{n}\\p\vec{u}^T\cdot\vec{n}
\end{pmatrix}=\widetilde{\vec{\mathcal{F}}}(\vec{U})\vec{n},~~
\widetilde{\vec{\mathcal{F}}}(\vec{U})=\begin{pmatrix}
0&0\\p&0\\0&p\\pu_x&pu_y
\end{pmatrix},
\end{equation*}
and   the geometrical relation $\sum\limits_{m=1}^{4}\vec{n}_{ij}^m|l_{ij}^m|=\vec{0}$ holds,
  utilizing the flux consistency gives
\begin{equation*}
\sum\limits_{m=1}^{4}\widehat{\vec{\mathcal{F}}}_{\vec{n}_{ij}^m}(\overline{\vec{U}}_{ij}^{n},\overline{\vec{U}}_{ij}^{n})|l_{ij}^m|
=\sum\limits_{m=1}^{4}\vec{\mathcal{F}}_{\vec{n}_{ij}^m}(\overline{\vec{U}}_{ij}^{n})|l_{ij}^m|
=\sum\limits_{m=1}^{4}\widetilde{\vec{\mathcal{F}}}(\overline{\vec{U}}_{ij}^{n})\vec{n}_{ij}^m|l_{ij}^m|
=\widetilde{\vec{\mathcal{F}}}(\overline{\vec{U}}_{ij}^{n})\sum\limits_{m=1}^{4}
\vec{n}_{ij}^m|l_{ij}^m|=\vec{0}.
\end{equation*}
Adding the identity $\vec{0}=\Delta t^n\sum\limits_{m=1}^{4}\widehat{\vec{\mathcal{F}}}_{\vec{n}_{ij}^m}(\overline{\vec{U}}_{ij}^{n},\overline{\vec{U}}_{ij}^{n}
)|l_{ij}^m|$ into the scheme \eqref{eq73}  gives
\begin{align*}
\overline{\vec{U}}_{ij}^{n+1}A_{ij}^{n+1}&=\overline{\vec{U}}_{ij}^{n}A_{ij}^{n}-\Delta t^n\sum\limits_{m=1}^4
\widehat{\vec{\mathcal{F}}}_{\vec{n}_{ij}^m}(\overline{\vec{U}}_{ij}^{n},\overline{\vec{U}}_m^{\text{ext}(I_{ij})})|l_{ij}^m|
+\Delta t^n\sum\limits_{m=1}^{4}\widehat{\vec{\mathcal{F}}}_{\vec{n}_{ij}^m}(\overline{\vec{U}}_{ij}^{n},\overline{\vec{U}}_{ij}^{n}
)|l_{ij}^m|
\\
&=\overline{\vec{U}}_{ij}^{n}A_{ij}^{n}-\Delta t^n\sum\limits_{m=1}^4
\left(\widehat{\vec{\mathcal{F}}}_{\vec{n}_{ij}^m}(\overline{\vec{U}}_{ij}^{n},\overline{\vec{U}}_m^{\text{ext}(I_{ij})})
-\widehat{\vec{\mathcal{F}}}_{\vec{n}_{ij}^m}(\overline{\vec{U}}_{ij}^{n},\overline{\vec{U}}_{ij}^{n})\right)|l_{ij}^m|
\\
&=\sum\limits_{m=1}^4|l_{ij}^m|\left[\overline{\vec{U}}_{ij}^{n}\widetilde{A}_{ij}^{n}-\Delta t^n
\left(\widehat{\vec{\mathcal{F}}}_{\vec{n}_{ij}^m}(\overline{\vec{U}}_{ij}^{n},\overline{\vec{U}}_m^{\text{ext}(I_{ij})})
-\widehat{\vec{\mathcal{F}}}_{\vec{n}_{ij}^m}(\overline{\vec{U}}_{ij}^{n},\overline{\vec{U}}_{ij}^{n})\right)\right]
=\sum\limits_{m=1}^4|l_{ij}^m|\mathcal{H}_m,
\end{align*}
where
\begin{equation*}
\widetilde{A}_{ij}^{n}=\frac{A_{ij}^{n}}{\sum\limits_{m=1}^4|l_{ij}^m|},~~~
\mathcal{H}_m=\overline{\vec{U}}_{ij}^{n}\widetilde{A}_{ij}^{n}-\Delta t^n
\left(\widehat{\vec{\mathcal{F}}}_{\vec{n}_{ij}^m}(\overline{\vec{U}}_{ij}^{n},\overline{\vec{U}}_m^{\text{ext}(I_{ij})})
-\widehat{\vec{\mathcal{F}}}_{\vec{n}_{ij}^m}(\overline{\vec{U}}_{ij}^{n},\overline{\vec{U}}_{ij}^{n})\right).
\end{equation*}
It means that the scheme \eqref{eq73} can be expressed as a combination of $\mathcal{H}_m$ with positive coefficients,
while $\mathcal{H}_m$ has the same form as the first order scheme \eqref{eq8} with \eqref{eq9} and HLLC solver.
Thus the analysis in Section \ref{section:1D-PCP}
can be applied to get the sufficient condition for that $\mathcal{H}_m$ is in the
admissible state set $\mathcal{G}$. Summarizing  those draws the following conclusion.
\begin{theorem}
For the first order finite volume Lagrangian scheme \eqref{eq73},
if $\{\overline{\vec{U}}_{ij}^{n}\in \mathcal{G},\forall i=1,
\cdots,N_x;j=1,\cdots,N_y\}$, then we have $\overline{\vec{U}}_{ij}^{n+1}\in \mathcal{G}$ for all $i=1,\cdots,N_x;j=1,\cdots,N_y$
under the wave speed estimates \eqref{eq30} and the following time step restriction
\begin{equation}\label{eq:dt2}
\Delta t^n\le\lambda\min\limits_{i,j}\left\{\frac{A_{ij}^{n}}{\sum\limits_{m=1}^4|l_{ij}^m|}
/\max\big(|s_{\min}(\overline{\vec{U}}_{ij}^n)|,|s_{\max}(\overline{\vec{U}}_{ij}^n)|\big)\right\},
\end{equation}
where the CFL number $\lambda\le\frac{1}{2}$.
\end{theorem}

\subsection{High-order accurate scheme}
In view of that the Lagrangian methods on the
the quadrilateral grid with  straight edge can be at most second order accurate anyway \cite{cheng2}, this section focuses on developing the two-dimensional second-order accurate PCP Lagrangian finite volume schemes with the HLLC solver, based on the  initial reconstruction,
the scaling PCP limiter, and the second order Runge-Kutta time discretization.

The following discusses the initial reconstruction and the scaling PCP limiter. At this time,
the two-dimensional  Lagrangian finite volume scheme \eqref{eq72}  should be replaced with
\begin{equation}\label{eq80}
\overline{\vec{U}}_{ij}^{n+1}A_{ij}^{n+1}=\overline{\vec{U}}_{ij}^{n}A_{ij}^{n}-\Delta t^n\sum\limits_{m=1}^{4}
\sum\limits_{\alpha=1}^{L}\omega_{\alpha}\widehat{\vec{\mathcal{F}}}_{\vec{n}_{ij}^m}\big((\vec{U}_{m,ij}^{\alpha})^{-},
(\vec{U}_{m,ij}^{\alpha})^{+}\big)|l_{ij}^m|,
\end{equation}
which is derived by approximating  the line integral in \eqref{eq72}  via the $L$-point Gauss-Lobatto quadrature,
where $\{(\vec{U}_{m,ij}^{\alpha})^{-},(\vec{U}_{m,ij}^{\alpha})^{+}\}$ are the left and right limit values from the inside and outside of the cell $I^n_{ij}$, respectively, approximating  the
conservative variable $\vec{U}$ at $\vec{x}_{m,ij}^{\alpha}$, $\alpha=1,\dots,L$, which are the $L$
Gauss-Lobatto quadrature points mapped onto ${\partial I_{ij}^{n,m}}$.
Those limit values are obtained by using the  high-order WENO reconstruction  used in \cite{cheng3}.
To be specific, for the edge $\partial I^{n,1}_{ij}$ and the five known cell-average values $\{\overline{\vec{U}}_{i,j}^{n},\overline{\vec{U}}_{i-1,j}^{n},\overline{\vec{U}}_{i+1,j}^{n},\overline{\vec{U}}_{i,j-1}^{n},\overline{\vec{U}}_{i,j+1}^{n}\}$
of the solutions of the RHD equations \eqref{eq1} or \eqref{eq4}  with $d=2$,
we first transform those cell average values $\{\overline{\vec{U}}_{i,j}^{n}\}$ into $\{\overline{\mathbf {U}}_{i,j}^{n}\}$ by using the local rotational transformation
from the $(x,y)$ coordinates to the local $(\xi,\eta)$ coordinates,
where $\xi$ and $\eta$ are in the $\bm{n}^{n,1}_{ij}$ and $\bm{\tau}^{\tau,1}_{ij}$ directions, respectively,
see Figure \ref{fig:WENO},
and then  calculate the values of corresponding characteristic variables $\vec{W}_{i,j}^n=\vec L\overline{\mathbf U}_{i,j}^{n},\vec{W}_{i-1,j}^n=\vec L\overline{\mathbf U}_{i-1,j}^{n},
\vec{W}_{i+1,j}^n=\vec L\overline{\mathbf U}_{i+1,j}^{n},\vec{W}_{i,j+1}^n=\vec L\overline{\mathbf U}_{i,j+1}^{n},\vec{W}_{i,j-1}^n=\vec L\overline{\mathbf U}_{i,j-1}^{n}$, where $\vec L=\vec L(\overline{\vec{U}}_{i,j}^{n})$ is the left eigen matrix of $\partial \vec F_{\vec n}/\partial \vec U$.
Let us consider the following four stencils
$S_1=\{I^{n}_{i,j},I^{n}_{i+1,j},I^{n}_{i,j+1}\}$, $S_2=\{I^{n}_{i,j},I^{n}_{i-1,j},I^{n}_{i,j+1}\}$,
$S_3=\{I^{n}_{i,j},I^{n}_{i+1,j},I^{n}_{i,j-1}\}$, and $S_4=\{I^{n}_{i,j},I^{n}_{i-1,j},I^{n}_{i,j-1}\}$.

\begin{figure}[!ht]
	\centering
	\includegraphics[width=0.5\textwidth]{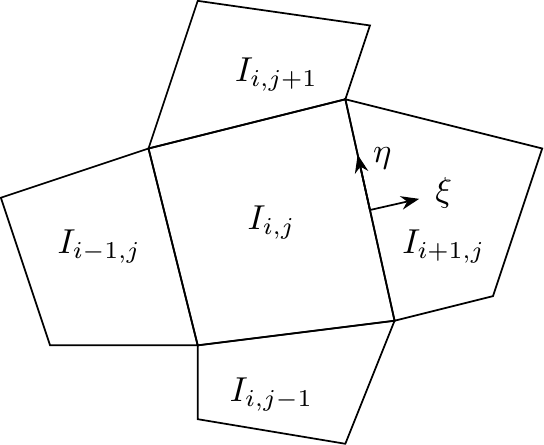}
	\caption{The local coordinates and the stencils in the WENO reconstruction.}
	\label{fig:WENO}
\end{figure}
For the  stencil $S_q$, $q=1,2,3,4$, one can obtain a linear polynomial
$$\widetilde{\vec{W}}_{q}(\xi,\eta)=\vec{a}_q(\xi-\xi_{ij})+\vec{b}_q(\eta-\eta_{ij})+\vec{c}_q,$$
where $(\xi_{ij},\eta_{ij})$ is the barycenter of $I^n_{ij}$
and the coefficients $\{\vec{a}_q,\vec{b}_q,\vec{c}_q\}$ are  determined
by preserving the cell average values, e.g. for $q=1$
\begin{align*}
&\int_{I_{i,j}}\widetilde{\vec{W}}_{1}(\xi,\eta)d\xi d\eta=\vec{W}^n_{i,j}A^n_{i,j},\\
&\int_{I_{i+1,j}}\widetilde{\vec{W}}_{1}(\xi,\eta)d\xi d\eta=\vec{W}^n_{i+1,j}A^n_{i+1,j},\\
&\int_{I_{i,j+1}}\widetilde{\vec{W}}_{1}(\xi,\eta)d\xi d\eta=\vec{W}^n_{i,j+1}A^n_{i,j+1}.
\end{align*}
Using those  can give the final linear polynomial
$\widetilde{\vec{W}}(\vec x)=\widetilde{\vec{W}}(\xi,\eta)=\vec{a}(\xi-\xi_{ij})+\vec{b}(\eta-\eta_{ij})+\vec{c}$
with the coefficients $\{\vec{a},\vec{b},\vec{c}\}$ determined by
\begin{equation*}
\vec{a}=\sum_{q=1}^4  {\omega}_q\vec{a}_q,~\vec{b}=\sum_{q=1}^4  {\omega}_q\vec{b}_q,~\vec{c}=\sum_{q=1}^4 {\omega}_q\vec{c}_q,
\end{equation*}
where the weights $\{ {\omega}_q\}$ are defined by
\begin{equation}
{\omega}_q=\dfrac{\widetilde{{\omega}}_q}{\sum_{r=1}^4 \widetilde{{\omega}}_r},\quad
\widetilde{{\omega}}_r=\dfrac{1}{[(|\vec{a}_r|^2+|\vec{b}_r|^2)A^n_{i,j}+\epsilon]^2}, \
\epsilon=10^{-6}.
\end{equation}
Using the polynomial $\widetilde{\vec{W}}(\vec x)$ calculates  its values
at the  Gauss-Lobatto point and then gives
$(\mathbf {U}_{1,ij}^{\alpha})^-=\vec L^{-1}  \widetilde{\vec{W}}({\vec x}^{\alpha}_{1,ij})$. Finally,
the values $(\vec{U}_{1,ij}^{\alpha})^-$ can be obtained by using the inverse rotation transformation.
It is worth noting that
as soon as the limit value $(\mathbf {U}_{m,ij}^{\alpha})^{-}$ and $(\mathbf{U}_{m,ij}^{\alpha})^{+}$  can give
the HLLC flux and HLLC intermediate states
corresponding to the local 1D Riemann problem located at the point $\vec{x}_{m,ij}^{\alpha}$, $\alpha=1,\dots,L$.

\begin{figure}[!ht]
\centering
\includegraphics[width=13.5cm]{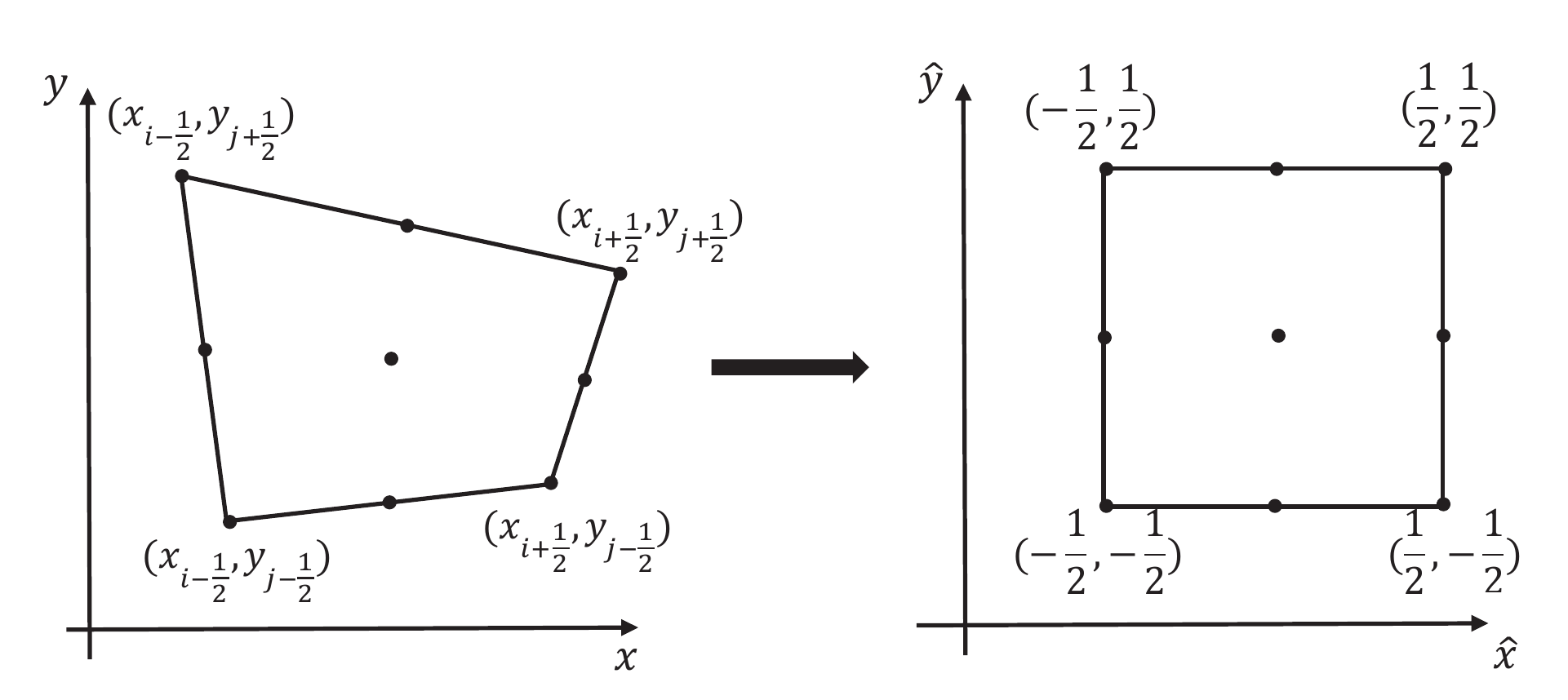}
\caption{The quadrilateral in the $(x,y)$ plane   mapped  into
the square in the $(\hat{x}, \hat{y})$ plane.}
\label{transformation}
\end{figure}

To give a decomposition similar to \eqref{eq2.25}, we use a coordinate transformation
$\vec x=\vec x_{ij}(\hat{x}, \hat{y})$ to transform the quadrilateral cell $I_{ij}^n$  in the $(x,y)$
plane  to the unit square $\hat{I}_0=[-\frac{1}{2},\frac{1}{2}]
\times[-\frac{1}{2},\frac{1}{2}]$ in the $(\hat{x}, \hat{y})$ plane,
see   Figure \ref{transformation},
then   define the set of the 2D  Gauss-Lobatto quadrature points in the cell $I_{ij}^n$  by
\begin{equation}\label{2dg-l points}
S_{ij}=\{(x_{\alpha},y_{\beta}),\alpha=1,\dots,L,\beta=1,\dots,L\},
\end{equation}
which are derived  by inversely transformed   the Gauss-Lobatto quadrature points in the unit square $\hat{I}_0$.
Because only the second-order accurate scheme is considered here,
one may  apply the tensor product Simpson quadrature rule, in which the quadrature points consist of the cell vertices, the mid-points of each edge and the
cell center, see Figure \ref{transformation}, i.e.~$L=3$, $\omega_1=\omega_3=\frac{1}{6}$, and $\omega_2=\frac{2}{3}$.
Based on those,  the term $\overline{\vec{U}}^n_{ij}A^n_{ij}$   can be decomposed  into
\begin{align}
\overline{\vec{U}}^n_{ij}A^n_{ij}&=\sum\limits_{\alpha=1}^{3}\omega_{\alpha}\omega_1
 J_{ij}^{\alpha,1}
\vec{U}_{1,ij}^{\alpha,1}+\omega_{\ast}^{1}\vec{U}_{1,ij}^{\ast\ast}
=\sum\limits_{\alpha=1}^{3}\omega_3\omega_{\alpha}
 J_{ij}^{3,\alpha}
\vec{U}_{2,ij}^{3,\alpha}+\omega_{\ast}^{2}\vec{U}_{2,ij}^{\ast\ast}\notag\\
&=\sum\limits_{\alpha=1}^{3}\omega_{\alpha}\omega_3
J_{ij}^{\alpha,3}
\vec{U}_{3,ij}^{\alpha,3}+\omega_{\ast}^{3}\vec{U}_{3,ij}^{\ast\ast}
=\sum\limits_{\alpha=1}^{3}\omega_{1}\omega_\alpha
J_{ij}^{_1,\alpha}
\vec{U}_{4,ij}^{1,\alpha}+\omega_{\ast}^{4}\vec{U}_{4,ij}^{\ast\ast},\label{eq77}
\end{align}
for the four edges $\partial I^{n,m}_{ij}$, $m=1,2,3,4$,
where
$J_{ij}^{\alpha,\beta}=J_{ij}(\hat{x}_\alpha,\hat{y}_\beta)$,
$J_{ij}(\hat{x},\hat{y})=\left|\frac{\partial (x_{ij}(\hat{x},\hat{y}),y_{ij}(\hat{x},\hat{y}))}{\partial(\hat{x},\hat{y})}\right|$ is the Jacobian for the coordinate transformation, and
\begin{align}
\omega_{\ast}^{1}&=\sum\limits_{\alpha=1}^{3}\sum\limits_{\beta=2}^{3}\omega_\alpha\omega_\beta
J_{ij}^{\alpha,\beta},~~~
\omega_{\ast}^{2}=\sum\limits_{\alpha=1}^{2}\sum\limits_{\beta=1}^{3}\omega_\alpha\omega_\beta
J_{ij}^{\alpha,\beta},\notag\\
\omega_{\ast}^{3}&=\sum\limits_{\alpha=1}^{3}\sum\limits_{\beta=1}^{2}\omega_\alpha\omega_\beta
J_{ij}^{\alpha,\beta},~~~
\omega_{\ast}^{4}=\sum\limits_{\alpha=2}^{3}\sum\limits_{\beta=1}^{3}\omega_\alpha\omega_\beta
J_{ij}^{\alpha,\beta}.\label{omega}
\end{align}
From \eqref{eq77}, one has
\begin{align}
\overline{\vec{U}}^n_{ij}A_{ij}^n&=\frac{1}{4}\Big(\sum\limits_{\alpha=1}^{3}\widetilde{\omega}_{\alpha,1}
\vec{U}_{1,ij}^{\alpha,1}+\omega_{\ast}^{1}\vec{U}_{1,ij}^{\ast\ast}
+\sum\limits_{\alpha=1}^{3}\widetilde{\omega}_{3,\alpha}
\vec{U}_{2,ij}^{3,\alpha}+\omega_{\ast}^{2}\vec{U}_{2,ij}^{\ast\ast} \notag\\
&~~~+  \sum\limits_{\alpha=1}^{3}\widetilde{\omega}_{\alpha,3}
\vec{U}_{3,ij}^{\alpha,3}+\omega_{\ast}^{3}\vec{U}_{3,ij}^{\ast\ast}
+\sum\limits_{\alpha=1}^{3}\widetilde{\omega}_{1,\alpha}
\vec{U}_{4,ij}^{1,\alpha}+\omega_{\ast}^{4}\vec{U}_{4,ij}^{\ast\ast}\Big)\notag\\
&=\frac{1}{4}\sum\limits_{m=1}^{4}\omega_{\ast}^{m}\vec{U}_{m,ij}^{\ast\ast}
+\frac{1}{4}\sum\limits_{\alpha=1}^{3}\big(\widetilde{\omega}_{\alpha,1}\vec{U}_{1,ij}^{\alpha,1}
+\widetilde{\omega}_{3,\alpha}\vec{U}_{2,ij}^{3,\alpha}
+\widetilde{\omega}_{\alpha,3}\vec{U}_{3,ij}^{\alpha,3}
+\widetilde{\omega}_{1,\alpha}\vec{U}_{4,ij}^{1,\alpha}\big),\label{eq82}
\end{align}
where $\widetilde{\omega}_{\alpha,\beta}=\omega_{\alpha}\omega_{\beta}J_{ij}^{\alpha,\beta}$.

Moreover, since the point values
$\vec{U}_{1,ij}^{\alpha,1},\vec{U}_{2,ij}^{3,\alpha},\vec{U}_{3,ij}^{\alpha,3},\vec{U}_{4,ij}^{1,\alpha}$
can be obtained from the WENO reconstruction, one can directly compute $\vec{U}_{m,ij}^{\ast\ast}$ from \eqref{eq77}
\begin{align}
\vec{U}_{1,ij}^{\ast\ast}&=\frac{1}{\omega_{\ast}^{1}}\left(\overline{\vec{U}}_{ij}A_{ij}
-\sum\limits_{\alpha=1}^{3}\omega_{\alpha}\omega_1J_{ij}^{\alpha,1}\vec{U}_{1,ij}^{\alpha,1}\right),~~
\vec{U}_{2,ij}^{\ast\ast}=\frac{1}{\omega_{\ast}^{2}}\left(\overline{\vec{U}}_{ij}A_{ij}
-\sum\limits_{\alpha=1}^{3}\omega_3\omega_{\alpha}J_{ij}^{3,\alpha}\vec{U}_{2,ij}^{3,\alpha}\right),\notag\\
\vec{U}_{3,ij}^{\ast\ast}&=\frac{1}{\omega_{\ast}^{3}}\left(\overline{\vec{U}}_{ij}A_{ij}
-\sum\limits_{\alpha=1}^{3}\omega_{\alpha}\omega_3J_{ij}^{\alpha,3}\vec{U}_{3,ij}^{\alpha,3}\right),~~
\vec{U}_{4,ij}^{\ast\ast}=\frac{1}{\omega_{\ast}^{4}}\left(\overline{\vec{U}}_{ij}A_{ij}
-\sum\limits_{\alpha=1}^{3}\omega_1\omega_{\alpha}J_{ij}^{1,\alpha}\vec{U}_{4,ij}^{1,\alpha}\right).\label{ustar}
\end{align}

Consequently, by adding and subtracting
the term $\Delta t^n\sum\limits_{m=2}^{4}\sum\limits_{\alpha=1}^{3}\omega_{\alpha}\widehat{\vec{\mathcal{F}}}_{\vec{n}_{ij}^m}
\big((\vec{U}_{1,ij}^{\alpha})^{-},(\vec{U}_{m,ij}^{\alpha})^{-}\big)|l_{ij}^m|$
and using \eqref{eq82} and the fact that $\omega_1=\omega_3$,
$(\vec{U}_{1,ij}^{\alpha})^{-}=\vec{U}_{1,ij}^{\alpha,1}$,
$(\vec{U}_{2,ij}^{\alpha})^{-}=\vec{U}_{2,ij}^{3,\alpha}$,
$(\vec{U}_{3,ij}^{\alpha})^{-}=\vec{U}_{3,ij}^{\alpha,3}$,
and $(\vec{U}_{4,ij}^{\alpha})^{-}=\vec{U}_{4,ij}^{1,\alpha}$,
 the scheme \eqref{eq80} becomes
\begin{align}
\overline{\vec{U}}_{ij}^{n+1}A_{ij}^{n+1}&=\overline{\vec{U}}_{ij}^{n}A_{ij}^{n}-\Delta t^n\sum\limits_{m=1}^{4}
\sum\limits_{\alpha=1}^{3}\omega_{\alpha}\widehat{\vec{\mathcal{F}}}_{\vec{n}_{ij}^m}\big((\vec{U}_{m,ij}^{\alpha})^{-},
(\vec{U}_{m,ij}^{\alpha})^{+}\big)|l_{ij}^m|\notag\\
&=\frac{1}{4}\sum\limits_{m=1}^{4}\omega_{\ast}^{m}\vec{U}_{m,ij}^{\ast\ast}
+\frac{1}{4}\omega_1\sum\limits_{\alpha=1}^{3}\omega_{\alpha}\big(\mathcal{H}_1^{\alpha}+\mathcal{H}_2^{\alpha}
+\mathcal{H}_3^{\alpha}+\mathcal{H}_4^{\alpha}\big),\label{eq83}
\end{align}
where
\begin{equation}\label{eq84}
\begin{aligned}
\mathcal{H}_1^{\alpha}&=\vec{U}_{1,ij}^{\alpha,1}J_{ij}^{\alpha,1}-\frac{4\Delta t^n}{\omega_1}
\Big(\widehat{\vec{\mathcal{F}}}_{\vec{n}_{ij}^1}\big(\vec{U}_{1,ij}^{\alpha,1},(\vec{U}_{1,ij}^{\alpha})^{+}\big)|l_{ij}^1|
+\sum\limits_{m=2}^4\widehat{\vec{\mathcal{F}}}_{\vec{n}_{ij}^m}\big(\vec{U}_{1,ij}^{\alpha,1},
(\vec{U}_{m,ij}^{\alpha})^{-}\big)|l_{ij}^m|\Big),\\
\mathcal{H}_2^{\alpha}&=\vec{U}_{2,ij}^{3,\alpha}J_{ij}^{3,\alpha}-\frac{4\Delta t^n}{\omega_1}
\left(\widehat{\vec{\mathcal{F}}}_{\vec{n}_{ij}^2}\big(\vec{U}_{2,ij}^{3,\alpha},(\vec{U}_{2,ij}^{\alpha})^{+}\big)
-\widehat{\vec{\mathcal{F}}}_{\vec{n}_{ij}^2}(\vec{U}_{1,ij}^{\alpha,1},\vec{U}_{2,ij}^{3,\alpha})\right)|l_{ij}^2|,\\
\mathcal{H}_3^{\alpha}&=\vec{U}_{3,ij}^{\alpha,3}J_{ij}^{\alpha,3}-\frac{4\Delta t^n}{\omega_1}
\left(\widehat{\vec{\mathcal{F}}}_{\vec{n}_{ij}^3}\big(\vec{U}_{3,\alpha}^{\alpha,3},(\vec{U}_{3,ij}^{\alpha})^{+}\big)
-\widehat{\vec{\mathcal{F}}}_{\vec{n}_{ij}^3}(\vec{U}_{1,ij}^{\alpha,1},\vec{U}_{3,ij}^{\alpha,3})\right)|l_{ij}^3|,\\
\mathcal{H}_4^{\alpha}&=\vec{U}_{4,ij}^{1,\alpha}J_{ij}^{1,\alpha}-\frac{4\Delta t^n}{\omega_1}
\left(\widehat{\vec{\mathcal{F}}}_{\vec{n}_{ij}^4}\big(\vec{U}_{4,ij}^{1,\alpha},(\vec{U}_{4,ij}^{\alpha})^{+}\big)
-\widehat{\vec{\mathcal{F}}}_{\vec{n}_{ij}^4}(\vec{U}_{1,ij}^{\alpha,1},\vec{U}_{4,ij}^{1,\alpha})\right)|l_{ij}^4|.
\end{aligned}
\end{equation}
It is obvious that the equation of $\mathcal{H}_1^{\alpha}$
has the same type as the 2D first-order scheme \eqref{eq73}, $\alpha=1,2,3$,
while the equation of $\mathcal{H}_m^{\alpha}$
is similar to the 1D first-order scheme \eqref{eq8} with \eqref{eq9},  $\alpha=1,2,3,m=2,3,4$.
Meanwhile,  $\overline{\vec{U}}^{n+1}$
is a convex combination of $\vec{U}_{m,ij}^{\ast\ast}$ and $\mathcal{H}_m^{\alpha}$, $\alpha=1,2,3, m=1,\cdots,4$.
Thus, if those terms are PCP, then
$\overline{\vec{U}}^{n+1}\in \mathcal{G}$ due to the convexity of the admissible state set $\mathcal G$.

\begin{theorem}
If $\overline{\vec{U}}_{ij}^{n}, \vec{U}_{m,ij}^{\ast\ast}\in \mathcal{G}$ for all $m,i,j$,
and the HLLC wave speeds are estimated in \eqref{eq30},
then    the high-order finite volume Lagrangian scheme \eqref{eq80} is PCP, i.e.
$\overline{\vec{U}}_{ij}^{n+1}\in \mathcal{G}$ for all $i=1,\cdots,N_x;j=1,\cdots,N_y$,
under  the following time stepsize restriction
\begin{equation}\label{eq85}
\Delta t^n\le\frac{\omega_1}{4}\lambda\min\limits_{i,j,\alpha}\left\{\frac{J_{ij}}{\sum\limits_{m=1}^4|l_{ij}^m|}
/\max\limits_{\vec{W}_{ij}^{\alpha}\in\mathcal{P}}\big(|s_{\min}(\vec{W}_{ij}^{\alpha})|,|s_{\max}(\vec{W}_{ij}^{\alpha})|\big)\right\},
\end{equation}
where the CFL number $\lambda\le\frac{1}{2}$,
$J_{ij}=\min\limits_{\alpha=1,\cdots,3}\{J_{ij}^{\alpha,1},
J_{ij}^{3,\alpha},J_{ij}^{\alpha,3},J_{ij}^{1,\alpha}\}$,
and
$$\mathcal{P}=\{\vec{U}_{1,ij}^{\alpha,1},\vec{U}_{2,ij}^{3,\alpha},\vec{U}_{3,ij}^{\alpha,3},\vec{U}_{4,ij}^{1,\alpha}\}.$$
\end{theorem}

Before ending this section,
we discuss  how to
 to limit  $\vec{U}_{m,ij}^{\ast\ast}$ defined in
\eqref{ustar} and $(\vec{U}_{m,ij}^{\alpha})^{-}$ reconstructed by the WENO technique
 such that  the limited values $\widetilde{\vec{U}}_{m,ij}^{\ast\ast}$ and  $(\widetilde{\vec{U}}_{m,ij}^{\alpha})^{-}$ belong to $\mathcal G$ when  $\overline{\vec{U}}_{ij}^{n}\in \mathcal{G}$.
For the sake of brevity, the superscript $n$ will be omitted in this section and a small parameter
 $\varepsilon$ is taken as $10^{-13}$.
Similar to the one-dimensional case, the scaling PCP limiter can be implemented as follows.

First,  enforce the positivity of the mass density. For each cell $I_{ij}$, define
\begin{equation*}
\theta_{m,ij}^1=\min\left\{1,\frac{\overline{D}_{ij}-\varepsilon}{\overline{D}_{ij}
-D_{\min}^{m}}\right\},
~~~D_{\min}^{m}=\min\limits_{\alpha}\left\{(D_{m,ij}^{\alpha})^{-},D_{m,ij}^{\ast\ast}\right\},
\end{equation*}
and limit
\begin{equation*}
(\widehat{D}_{m,ij}^{\alpha})^{-}=\overline{D}_{ij}+\theta_{m,ij}^{1}\big((D_{m,ij}^{\alpha})^{-}-\overline{D}_{ij}\big),
~\widehat{D}_{m,ij}^{\ast\ast}=\overline{D}_{ij}+\theta_{m,ij}^{1}\big({D}_{m,ij}^{\ast\ast}-\overline{D}_{ij}\big).
\end{equation*}
Define
$(\widehat{\vec{U}}_{m,ij}^{\alpha})^{-}=((\widehat{D}_{m,ij}^{\alpha})^{-},(\vec{m}_{m,ij}^{\alpha})^{-},(E_{m,ij}^{\alpha})^{-})^{T}$.

Next, enforce the positivity of the term $q(\vec{U})=E-\sqrt{D^2+|\vec{m}|^2}$.
For each cell $I_{ij}$, compute
\begin{equation*}
\theta_{m,ij}^2=\min\left\{1,\frac{q_m(\overline{\vec{U}}_{ij})-\varepsilon}{q_m(\overline{\vec{U}}_{ij})
-q_{\min}^m}\right\},~~
q_{\min}^m=\min\limits_{\alpha}\big\{q_m(\widehat{\vec{U}}_{ij}^{\ast\ast}),
q_m((\widehat{\vec{U}}_{m,ij}^{\alpha})^{-})\big\},
\end{equation*}
and then  limit the point values
\begin{equation*}
(\widetilde{\vec{U}}_{m,ij}^{\alpha})^{-}=\overline{\vec{U}}_{ij}+\theta_{m,ij}^2
\big((\widehat{\vec{U}}_{m,ij}^{\alpha})^{-}-\overline{\vec{U}}_{ij}\big),
~~\widetilde{\vec{U}}_{m,ij}^{\ast\ast}=\overline{\vec{U}}_{ij}+\theta_{m,ij}^2
\big(\widehat{\vec{U}}_{m,ij}^{\ast\ast}-\overline{\vec{U}}_{ij}\big).
\end{equation*}
It is  easy to show  that all those limited values are  in the admissible state
set $\mathcal{G}$  when  $\overline{\vec{U}}_{ij}^{n}\in \mathcal{G}$.

In order to get a 2D Lagrangian scheme of second order accuracy both in space and time,
we replace the forward Euler time discretization with the second order Runge-Kutta time discretization in the scheme \eqref{eq80},
which can be implemented as follows:

Stage 1:
\begin{equation*}
\begin{aligned}
&\vec{x}_{i+\frac{1}{2},j+\frac{1}{2}}^{(1)}=\vec{x}_{i+\frac{1}{2},j+\frac{1}{2}}^{n}+\Delta t^n\vec{u}_{i+\frac{1}{2},j+\frac{1}{2}}^{n},\\
&\overline{\vec{U}}_{ij}^{(1)}A_{ij}^{(1)}
=\overline{\vec{U}}_{ij}^{n}A_{ij}^{n}-\Delta t^n\mathcal{L}(\overline{\vec{U}}^n;i,j);\\
\end{aligned}
\end{equation*}

Stage 2:
\begin{equation*}
\begin{aligned}
&\vec{x}_{i+\frac{1}{2},j+\frac{1}{2}}^{n+1}=\frac12\vec{x}_{i+\frac{1}{2},j+\frac{1}{2}}^{n}
+\frac{1}{2}\left(\vec{x}_{i+\frac{1}{2},j+\frac{1}{2}}^{(1)}+
\Delta t^n\vec{u}_{i+\frac{1}{2},j+\frac{1}{2}}^{(1)}\right),\\
&\overline{\vec{U}}_{ij}^{n+1}A_{ij}^{n+1}
=\frac{1}{2}\overline{\vec{U}}_{ij}^{n}A_{ij}^{n}+\frac{1}{2}\left(\overline{\vec{U}}_{ij}^{(1)}A_{ij}^{(1)}-\Delta t^n\mathcal{L}(\overline{\vec{U}}^{(1)};i,j)\right);\\
\end{aligned}
\end{equation*}
where $\mathcal{L}(\overline{\vec{U}};i,j)=\sum\limits_{m=1}^{4}
\sum\limits_{\alpha=1}^{3}\omega_{\alpha}\widehat{\vec{\mathcal{F}}}_{\vec{n}_{ij}^m}\big((\vec{U}_{m,ij}^{\alpha})^{-},
(\vec{U}_{m,ij}^{\alpha})^{+}\big)|l_{ij}^m|$. Here, the velocity of the vertex $\vec{u}^{n}_{i+\frac{1}{2},j+\frac{1}{2}}$
can be computed by
\begin{equation*}
\vec{u}^{n}_{i+\frac{1}{2},j+\frac{1}{2}}=\frac{1}{4}\sum\limits_{k=1}^{4}\vec{u}^{\ast,k}_{i+\frac{1}{2},j+\frac{1}{2}},
\end{equation*}
where $\vec{u}^{\ast,k}_{i+\frac{1}{2},j+\frac{1}{2}},k=1,\cdots,4$
are fluid velocities at nodes of four cell edges
sharing the common vertex $\vec{x}_{i+\frac{1}{2},j+\frac{1}{2}}^n$ (e.g. the point P in Figure \ref{cell-2d}), respectively.
Here take the calculation of $\vec u^{\ast,1}_{i+\frac{1}{2},j+\frac{1}{2}}$ as an example.
The velocity $\vec u_{i+\frac{1}{2},j+\frac{1}{2}}^{\ast,1}$ is gotten by using the local rotation transformation
of $(u_n,u_\tau)_{i+\frac{1}{2},j+\frac{1}{2}}^{\ast,1}$, where
$(u_n)_{i+\frac{1}{2},j+\frac{1}{2}}^{\ast,1}=s_{i+\frac{1}{2},j+\frac{1}{2}}^{\ast,1}$,
$(u_\tau)_{i+\frac{1}{2},j+\frac{1}{2}}^{\ast,1}=\frac{1}{2}
\left( (u_\tau^{n})_{i+\frac{1}{2},j+\frac{1}{2}}^{-,1}+(u_\tau^n)_{i+\frac{1}{2},j+\frac{1}{2}}^{+,1}\right)$,
and $s_{i+\frac{1}{2},j+\frac{1}{2}}^{\ast,1}$ is the
speed of contact discontinuity in the HLLC solver. And the computation of $\vec{u}^{(1)}_{i+\frac{1}{2},j+\frac{1}{2}}$
can be done by the similar way.

\section{Numerical results}\label{section:NumResult}
This section conducts some numerical experiments on several ultra-relativistic RHD problems with large Lorentz factor, or strong discontinuities, or low rest-mass density or pressure etc. to verify the accuracy, robustness and effectiveness of the studied PCP Lagrangian schemes.
It is worth stressing that those ultra-relativistic RHD problems seriously challenge the numerical schemes.
Unless otherwise stated, all the computations are restricted to the equation of state \eqref{eq55} with the adiabatic index $\Gamma=1.4$, and
the time step size $\Delta t$ of the  1D (resp. 2D)  first-order schemes is determined by \eqref{eq:dt1} (resp. \eqref{eq:dt2}), while it is decided by  \eqref{eq56} (resp. \eqref{eq85}) for the 1D (resp. 2D) high order schemes, where  the value of the parameter $\lambda$ is
 uniformly taken as $1/2$.

\subsection{1D case}

\begin{example}[Accuracy test]\label{example1d01}
It is to  test the accuracy and  PCP property of our schemes.
The initial density and pressure are given by
\begin{align*}
  &\rho(x,0)=\rho_{\text{ref}}+\alpha f(x), \quad p(x,0)=K\rho^\Gamma,\ \
  f(x)=\begin{cases}
    (x^2/L^2-1)^4, &  |x|<L,\\
    0,             &  |x|\geq L,
  \end{cases}
\end{align*}
where $\rho_{\text{ref}}=10^{-7}, K=10^{-1}, L=0.3, \alpha=1$, and $\Gamma=5/3$.
The initial velocity $u$ is specified by assuming that the Riemann
invariant
\begin{align*}
  J_-=\dfrac{1}{2}\ln\left(\dfrac{1+u}{1-u}\right)-\dfrac{1}{\sqrt{\Gamma-1}}\ln\left(\dfrac{\sqrt{\Gamma-1}+c_s}{\sqrt{\Gamma-1}-c_s}\right)
\end{align*}
is constant. It describes an isentropic pulse moving in a smooth domain, similar to
one  in \cite{zw}.  The computational domain is taken as $[-0.35,1]$, and the exact solution  can
be obtained by the method of characteristics.

 The errors $\varepsilon_{1}$, $\varepsilon_{2}$ and $\varepsilon_{\infty}$ and corresponding orders of convergence  at $t=0.02$  obtained by the first- and  third-order Lagrangian schemes are shown in Tables
\ref{tab:acc1D} and \ref{tab:acc1D-b},
where the errors are defined by
\begin{align*}
  &\varepsilon_{1} :=\int_{\Omega(t)}||\mathbf{U} -\mathbf{U}_h||_{1}~ dV,  \
  \varepsilon_{2} :=\sqrt{\int_{\Omega(t)}||\mathbf{U} -\mathbf{U}_h||_{2}^2~ dV},\
 \varepsilon_{\infty} :=\max_{\Omega(t)}||\mathbf{U} -\mathbf{U}_h||_{\infty},
\end{align*}
here $\mathbf{U} $ and $\mathbf{U}_h$ are the exact  and   numerical
solutions at $t$, respectively.  Table \ref{tab:acc1D-b} also
lists the proportions of the PCP limited cells at all time levels, denoted by $\Theta_N$.
 It is shown that the PCP limiter has been performed in
the higher-order accurate schemes because of the low pressure, and
 the higher-order PCP Lagrangian schemes  can achieve the theoretical accuracy.

\end{example}

\begin{table}[h!]
  \centering
  \begin{tabular}{r||c|c||c|c||c|c}
    \hline $N$ & $\varepsilon_{1}$ & Order & $\varepsilon_{2}$ & Order
    & $\varepsilon_{\infty}$ & Order   \\
    \hline
     20 & 1.687e-02 &  -   & 1.729e-02 &  -   & 3.350e-02 &  -      \\
     40 & 8.356e-03 & 1.01 & 8.523e-03 & 1.02 & 1.698e-02 & 0.98    \\
     80 & 4.260e-03 & 0.97 & 4.288e-03 & 0.99 & 8.418e-03 & 1.01   \\
    160 & 2.171e-03 & 0.97 & 2.169e-03 & 0.98 & 4.272e-03 & 0.98    \\
    320 & 1.098e-03 & 0.98 & 1.092e-03 & 0.99 & 2.139e-03 & 1.00   \\
    \hline
  \end{tabular}
  \caption{Example \ref{example1d01}: Errors and orders of convergence $t=0.02$ obtained by using
  the first-order scheme.}
  \label{tab:acc1D}
\end{table}

\begin{table}[h!]
  \centering
  \begin{tabular}{r||c|c||c|c||c|c||c}
    \hline $N$ & $\varepsilon_{1}$ & Order & $\varepsilon_{2}$ & Order
    & $\varepsilon_{\infty}$ & Order &  $\Theta_N$ \\
    \hline
    20 & 9.067e-02 &  -   & 1.038e-01 &  -   & 2.317e-01 &  -  & 25.0\% \\
    40 & 1.785e-02 & 2.34 & 2.289e-02 & 2.18 & 6.573e-02 & 1.82 & 20.8\% \\
    80 & 3.373e-03 & 2.40 & 4.356e-03 & 2.39 & 1.368e-02 & 2.26 & 12.8\% \\
    160 & 3.776e-04 & 3.16 & 4.932e-04 & 3.14 & 1.695e-03 & 3.01 & 6.49\% \\
    320 & 3.306e-05 & 3.51 & 4.151e-05 & 3.57 & 1.715e-04 & 3.31 & 3.12\% \\
    \hline
  \end{tabular}
  \caption{Same as Table \ref{tab:acc1D}, except for the third-order scheme.}
  \label{tab:acc1D-b}
\end{table}

\begin{example}[Blast wave interaction]\label{example1d02}
This is an initial-boundary-value problem for the 1D RHD equations  and has been studied in \cite{marti3,wu2015,yz1}. The same initial setup is considered here.
The computational domain is $[0,1]$ with outflow boundary conditions, and the initial condition is
\begin{equation}
  (\rho,u,p)=\begin{cases}
    (1,0,1000),    &0<x<0.1,\\
    (1,0,0.01),    & 0.1<x<0.9,\\
    (1,0,100),     &0.9<x<1.
  \end{cases}
\end{equation}
It is a severe test because the interaction is happened in a very narrow
region and there meet the low density and pressure, and large velocity in the domain.
If the PCP limiter is not employed, then the calculation of the second- or third-order schemes
result in failure as soon as  the computed pressure or density  becomes negative.
Figure \ref{fig:blast_wave} shows the numerical results at $t=0.43$
obtained by using the second- and third-order schemes.
It can be found that the solutions within the interval $[0.5,0.53]$ obtained by using the third-order scheme are in good agreement with the exact solution, and
the discontinuities are exactly and well captured on a coarse mesh with $400$ cells.
\end{example}
\begin{figure}[h!]
  \centering
  \subfigure[$\rho$]{
    \includegraphics[width=0.3\textwidth, trim=40 20 50 0, clip]{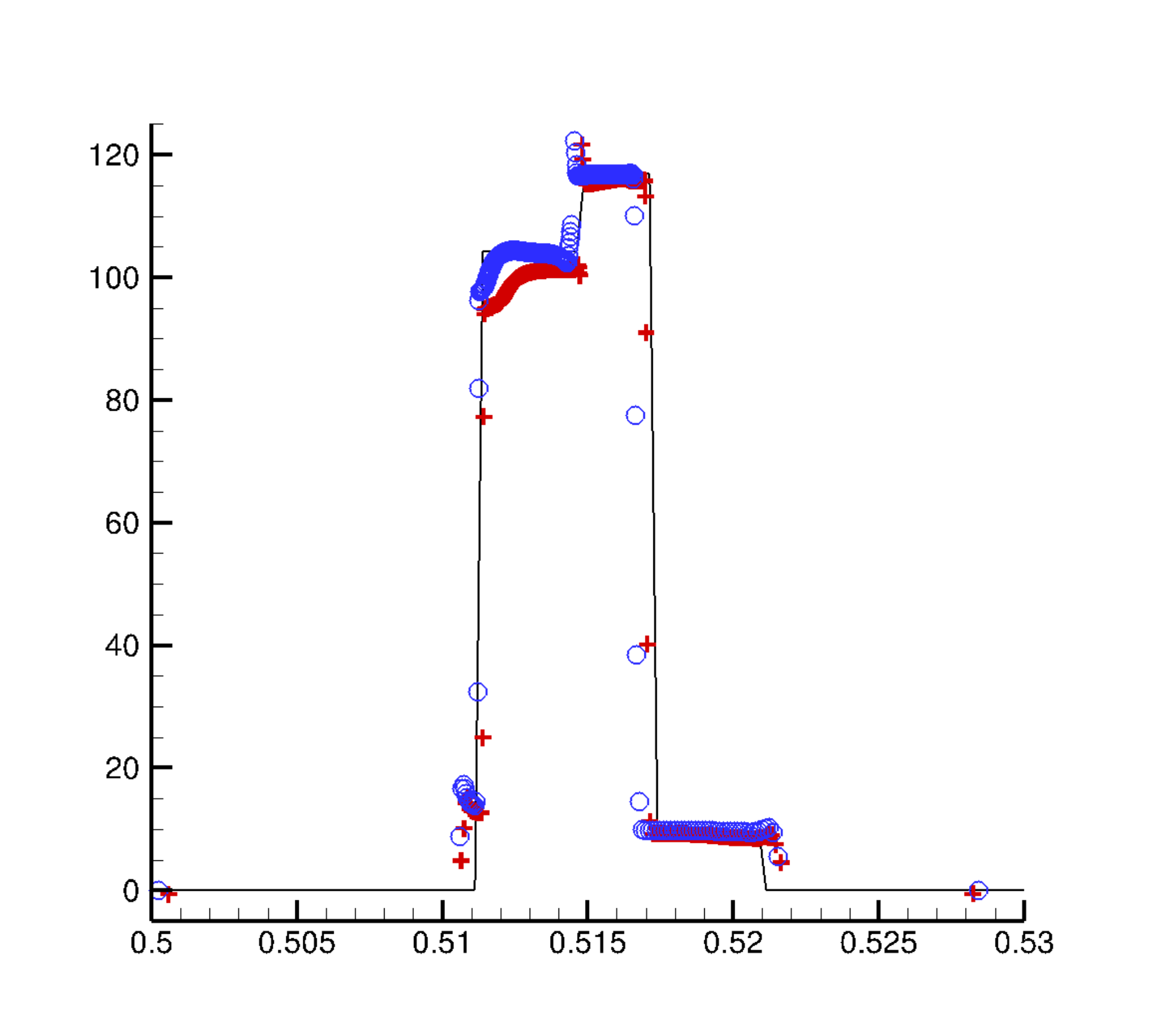}
  }
  \subfigure[$u$]{
    \includegraphics[width=0.3\textwidth, trim=40 20 50 0, clip]{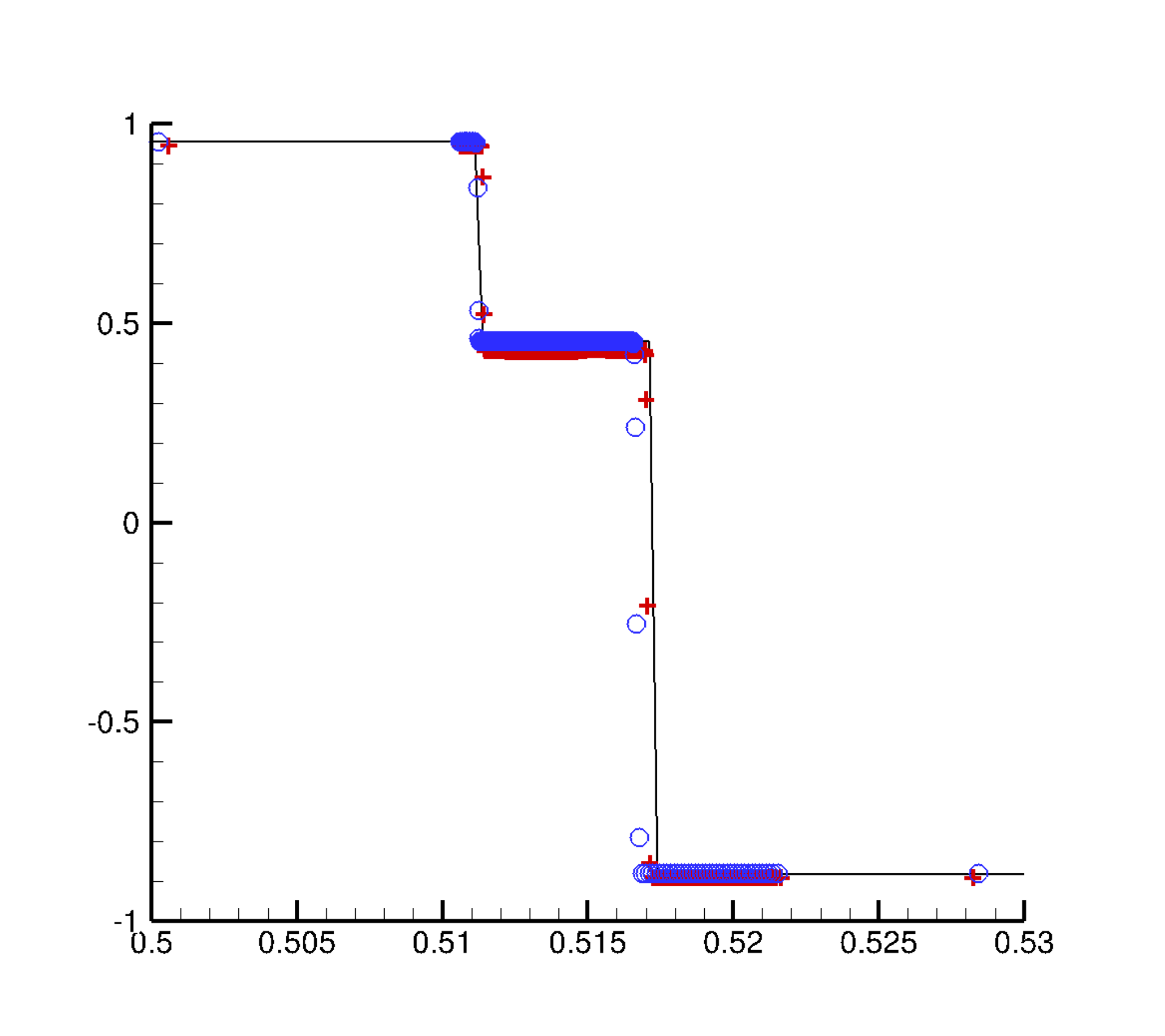}
  }
  \subfigure[$p$]{
    \includegraphics[width=0.3\textwidth, trim=40 20 50 0, clip]{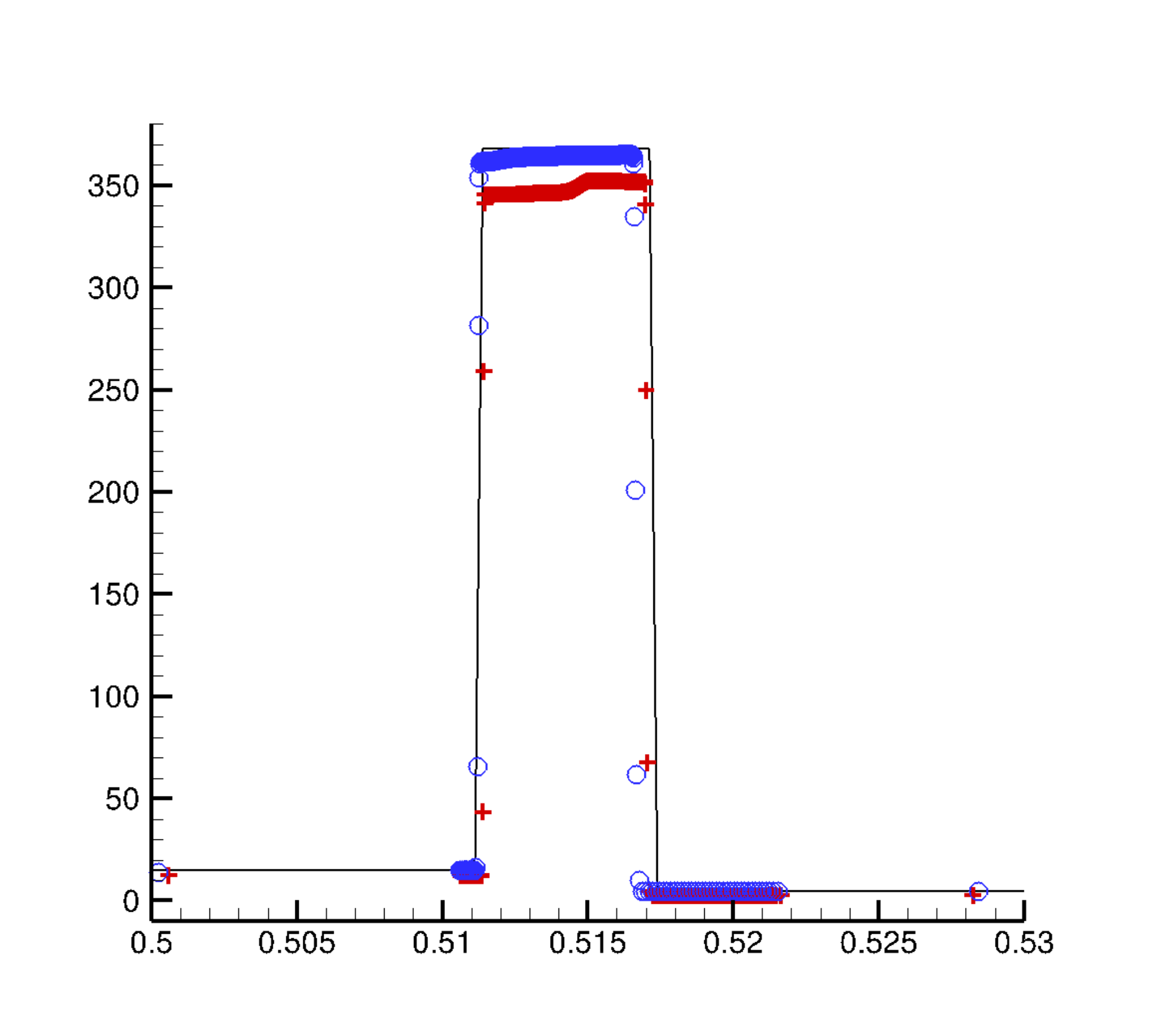}
  }
  \caption{Example \ref{example1d02}: Close-up of the solutions at $t=0.43$.
  The solid lines are the exact solutions, while the symbols ``$+$'' and  ``$\circ$'' denote the solutions obtained by using the second- and third-order Lagrangian schemes, respectively.}
  \label{fig:blast_wave}
\end{figure}

\subsection{2D case}

\begin{example}[Accuracy test]\label{example1}
Similar to \cite{shu},  a 2D relativistic isentropic vortex problem
is constructed here to test the accuracy of our Lagrangian schemes.
First, in a coordinate system $S$ with the spacetime
coordinates $(t,x,y)$,  a steady, relativistic isentropic vortex solution
$(\rho,u_x,u_y,p)$ of the 2D RHD equations is obtained as follows
\begin{align*}
  &\rho=(1-\alpha e^{1-r^2})^{\frac{1}{\Gamma-1}},\quad p=\rho^\Gamma,\quad
  (u_x,u_y)=(-{y},{x})f,\\
  &r=\sqrt{{x}^2+{y}^2}, \quad \alpha=\dfrac{(\Gamma-1)/\Gamma}{8\pi^2}   \epsilon^2,\\
  &\beta=\dfrac{\Gamma^2\alpha e^{1-r^2}}{2\Gamma-1-\Gamma\alpha e^{1-r^2}},\quad
  f=\sqrt{\dfrac{\beta}{1+\beta r^2}},
\end{align*}
where the vortex strength is $\epsilon=10.0828$ such that the
lowest density and lowest pressure  are $7.8\times 10^{-15}$ and
$1.78\times 10^{-20}$, respectively.
Next,  assume that a coordinate system $S'$ with the spacetime coordinates $(t',x',y')$
is in motion relative to the coordinate system $S$ with a constant velocity of magnitude
$w$ along the $(1,1)$ direction, from the perspective of an observer stationary
in $S$. Then the relation between the two coordinate systems is given by the Lorentz
transformation as
\begin{align*}
  &\gamma = \dfrac{1}{\sqrt{1-w^2}},\quad t=\gamma\big(t'+\dfrac{w}{\sqrt{2}}(x'+y')\big),\\
  &x=x'+\dfrac{\gamma-1}{2}(x'+y')+\dfrac{\gamma t'w}{\sqrt{2}},
  \quad y=y'+\dfrac{\gamma-1}{2}(x'+y')+\dfrac{\gamma t'w}{\sqrt{2}},
\end{align*}
and the transformation between the velocities is
\begin{align*}
u'&=\dfrac{1}{1-\frac{w(u_x+u_y)}{\sqrt{2}}}\left[\dfrac{u_x}{\gamma}-\dfrac{w}{\sqrt{2}}+\dfrac{\gamma
  w^2}{2(\gamma+1)}(u_x+u_y)\right],\\
v'&=\dfrac{1}{1-\frac{w(u_x+u_y)}{\sqrt{2}}}\left[\dfrac{u_y}{\gamma}-\dfrac{w}{\sqrt{2}}+\dfrac{\gamma
  w^2}{2(\gamma+1)}(u_x+u_y)\right].
\end{align*}
Using those transformations can give  a time-dependent solution
$(\rho',u',v',p')$ in the coordinate system $S'$
\begin{align*}
  &\rho'(x',y',t')=\rho(x(x',y',t'),y(x',y',t')),\\
  &p'(x',y',t')=p(x(x',y',t'),y(x',y',t')),\\
  &u_x'(x',y',t')=u_x', \quad u_y'(x',y',t')=u_y'.
\end{align*}
The vortex in the coordinate system $S'$ moves with a constant speed of magnitude $w$ in $(-1,-1)$ direction.
Unlike the non-relativistic case, the relativistic circular vortex is contracted in $(1,1)$
direction due to the Lorentz contraction, thus it becomes elliptic in the coordinate system $S'$.

Our numerical simulation is performed in an initial square $\Omega(0)=[-5,5]^2$
with $w=0.5$ and periodic boundary conditions, and the output time is $t=1$.
For such problem, the second order Lagrangian
scheme without the PCP limiter will break down because the
numerical solution cannot be guaranteed to be in the admissible state set during the
computation.
Tables \ref{tab:acc} and  \ref{tab:acc-b}
list the errors $\varepsilon_{1}$, $\varepsilon_{2}$ and $\varepsilon_{\infty}$ and
orders of convergence. Moreover, Table \ref{tab:acc-b} also
gives  the proportions of the PCP limited cells at all time levels, denoted by $\Theta_N$.
 Figure \ref{fig:vortex} plots the deformed mesh and shifted elliptic vortex
 in the primitive variables at $t=1$ obtained by the second-order PCP Lagrangian scheme.
 It is shown that  the first- and second-order  PCP Lagrangian schemes can achieve the expected theoretical accuracy, and the PCP limiter has been performed in the second-order accurate schemes.
\end{example}

\begin{table}[h!]
  \centering
  \begin{tabular}{r||c|c||c|c||c|c}
    \hline $N$ & $\varepsilon_{1}$ & Order & $\varepsilon_{2}$ & Order
    & $\varepsilon_{\infty}$ & Order   \\
    \hline
     20 & 1.287e-01 &  -   & 1.848e-01 &  -   & 9.417e-01 &  -      \\
     40 & 6.317e-02 & 1.03 & 8.755e-02 & 1.08 & 4.139e-01 & 1.19    \\
     80 & 3.148e-02 & 1.00 & 4.282e-02 & 1.03 & 2.056e-01 & 1.01    \\
    160 & 1.587e-02 & 0.99 & 2.139e-02 & 1.00 & 1.009e-01 & 1.03    \\
    320 & 7.946e-03 & 1.00 & 1.065e-02 & 1.01 & 4.918e-02 & 1.04    \\ \hline
  \end{tabular}
  \caption{Example \ref{example1}: Errors and orders of convergence at $t=1$.}
  \label{tab:acc}
\end{table}

\begin{table}[h!]
  \centering
  \begin{tabular}{r||c|c||c|c||c|c||c}
    \hline $N$ & $\varepsilon_{1}$ & Order & $\varepsilon_{2}$ & Order
    & $\varepsilon_{\infty}$ & Order & $\Theta_N$ \\ \hline
    20 & 8.131e-02 & -    & 1.303e-01 & -    & 6.264e-01 & -     &  1.52\%\\
    40  & 2.199e-02 & 1.89 & 3.769e-02 & 1.79 & 2.008e-01 & 1.64  &  0.38\%\\
    80  & 5.458e-03 & 2.01 & 1.008e-02 & 1.90 & 7.842e-02 & 1.36  &  0.19\%\\
    160 & 1.277e-03 & 2.10 & 2.393e-03 & 2.07 & 2.037e-02 & 1.94  &  0.082\%\\
    320 & 3.000e-04 & 2.09 & 5.559e-04 & 2.11 & 4.390e-03 & 2.21  &  0.015\%\\
    \hline
  \end{tabular}
  \caption{Same as Table \ref{tab:acc} except for the second-order scheme.}
  \label{tab:acc-b}
\end{table}

\begin{figure}[h!]
  \centering
  \subfigure[$\rho$]{
    \includegraphics[width=0.45\textwidth, trim=50 50 30 50, clip]{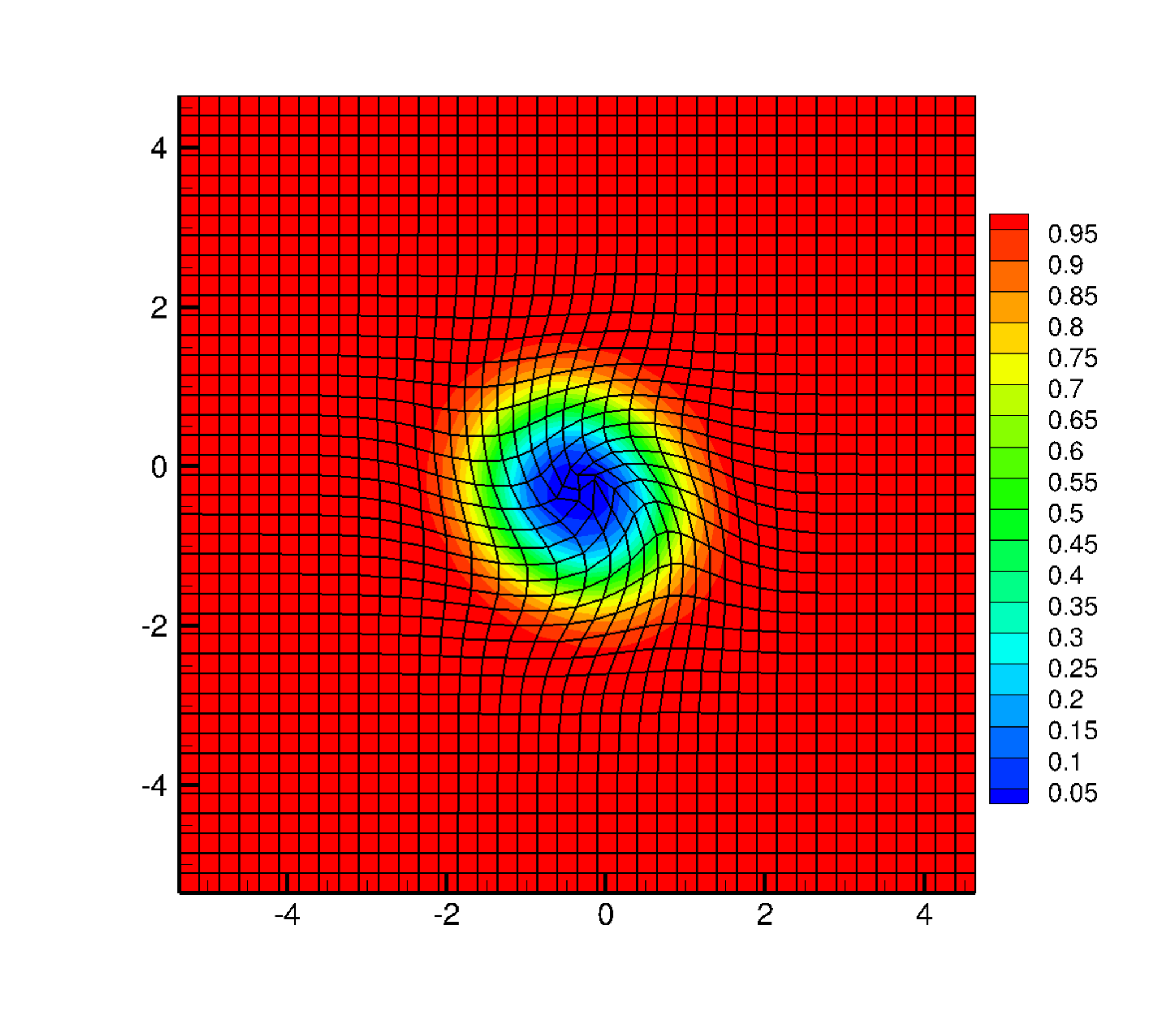}
  }
  \subfigure[$u_x$]{
    \includegraphics[width=0.45\textwidth, trim=50 50 30 50, clip]{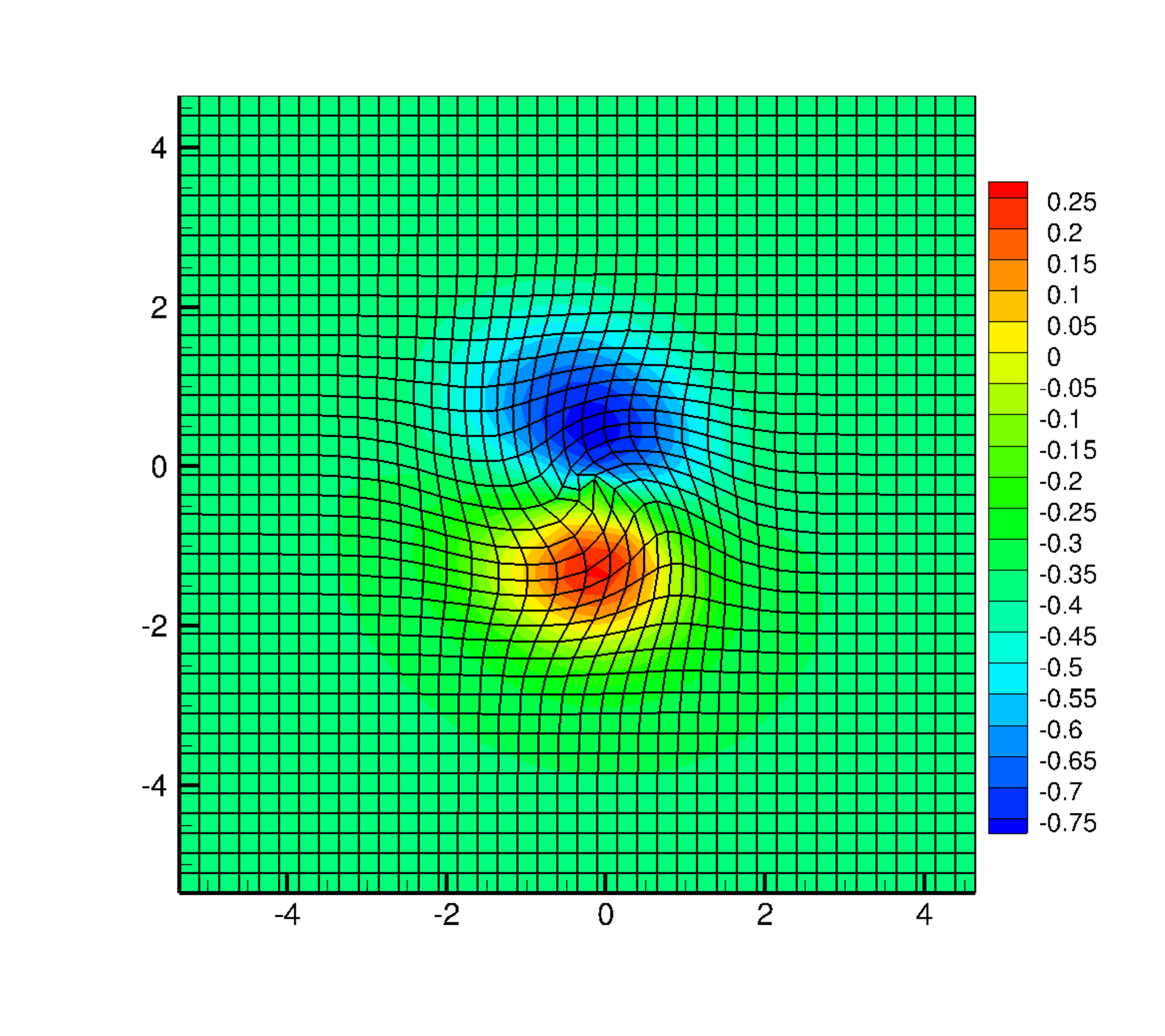}
  }
  \subfigure[$u_y$]{
    \includegraphics[width=0.45\textwidth, trim=50 50 30 50, clip]{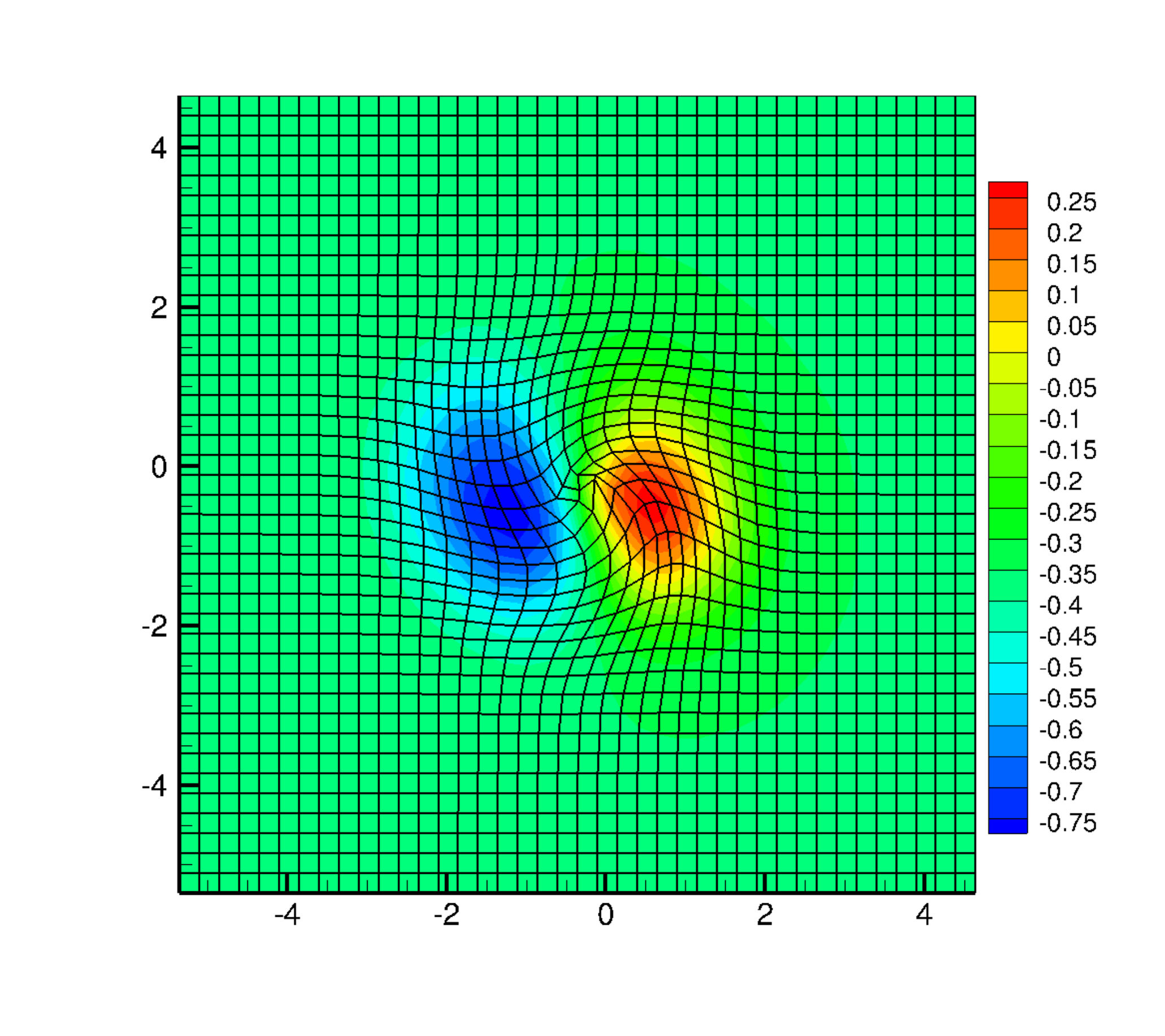}
  }
  \subfigure[$p$]{
    \includegraphics[width=0.45\textwidth, trim=50 50 30 50, clip]{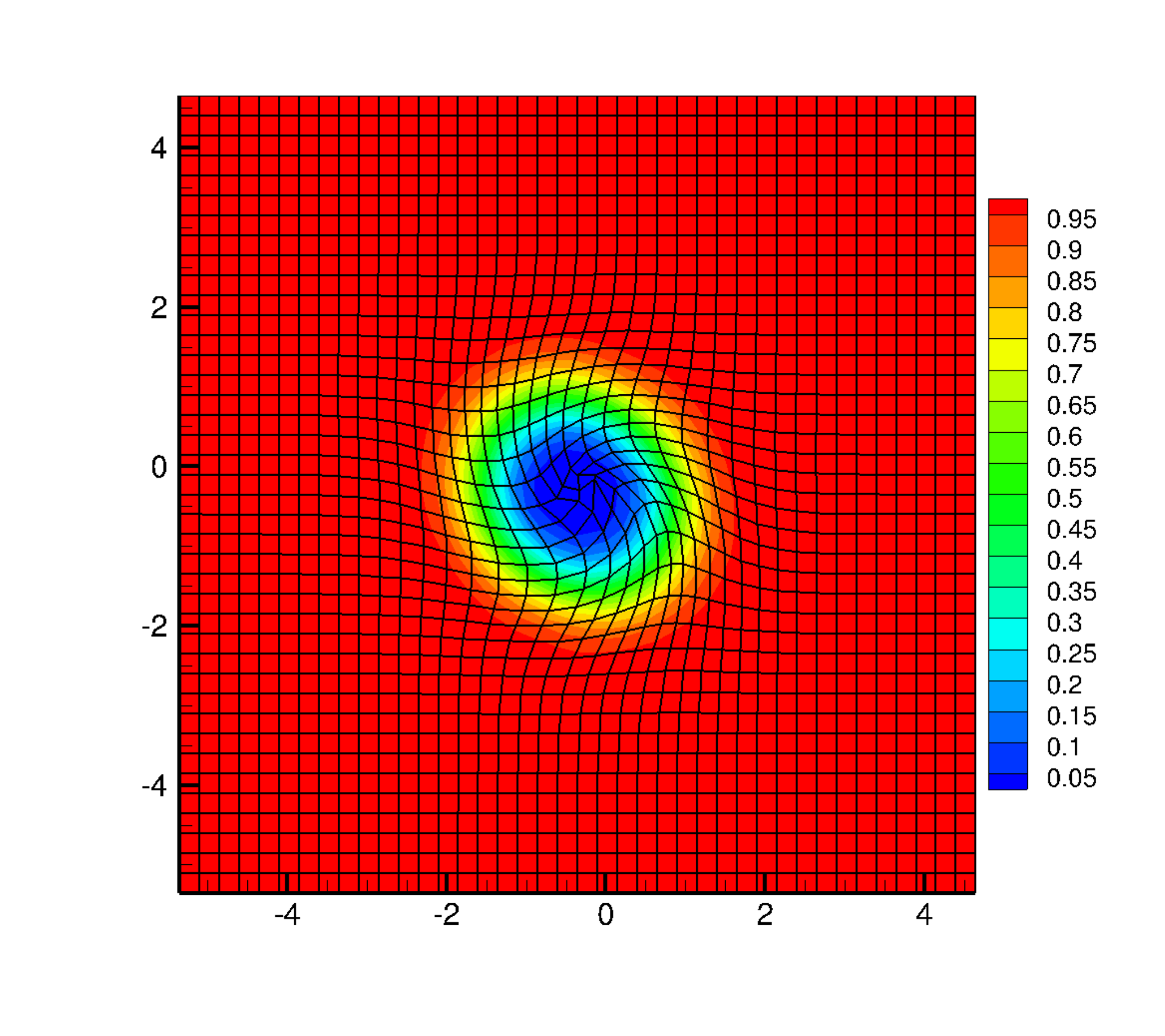}
  }
  \caption{Example \ref{example1}:  Mesh and primitive variables at $t=1$ obtained by using
   the second-order scheme.}
  \label{fig:vortex}
\end{figure}

\begin{example}[Blast problem on the Cartesian mesh]\label{example2}
Consider a blast problem in a square domain $[0,1]^2$ with reflective boundary
conditions.
The initial data are specified as follows
\begin{equation}
  (\rho,u_x,u_y,p)=\begin{cases}
    (10^{-10},0,0,1),    & r<0.5,\\
    (10^{-12},0,0,0.05), & r>0.5,
\end{cases}
\end{equation}
with $r=\sqrt{x^2+y^2}$.
The domain is initially divided into a uniform Cartesian mesh with $60\times 60$ cells.
If the second-order Lagrangian scheme is not  PCP, then  the negative density  may be numerically obtained so that the calculation will result in failure.

Figure \ref{fig:WB_cartesian} plots the  mesh and the density contour at  $t=0.4$
 obtained by using the second-order PCP Lagrangian scheme, and
the density and pressure  along the line $y=x$ obtained by the first-  and second-order PCP Lagrangian schemes.
The solid line in  Figure \ref{fig:WB_cartesian}(c)-(d)
is the reference solution obtained by using a second order TVD
    Eulerian scheme with Lax-Friedrichs flux in
    the cylindrical coordinate and with $10000$ cells, and
    clearly shows that the solution consists of a ``left-moving'' rarefaction wave, a ``right-moving'' contact
discontinuity and a ``right-moving'' shock wave.
It is seen that our PCP Lagrangian  schemes capture the contact discontinuity and the
shock wave with high resolution,  and the second-order scheme gives a better result
than the first-order.
\end{example}
\begin{figure}[h!]
  \centering
  \subfigure[Mesh]{
    \includegraphics[width=0.45\textwidth, trim=40 30 30 40, clip]{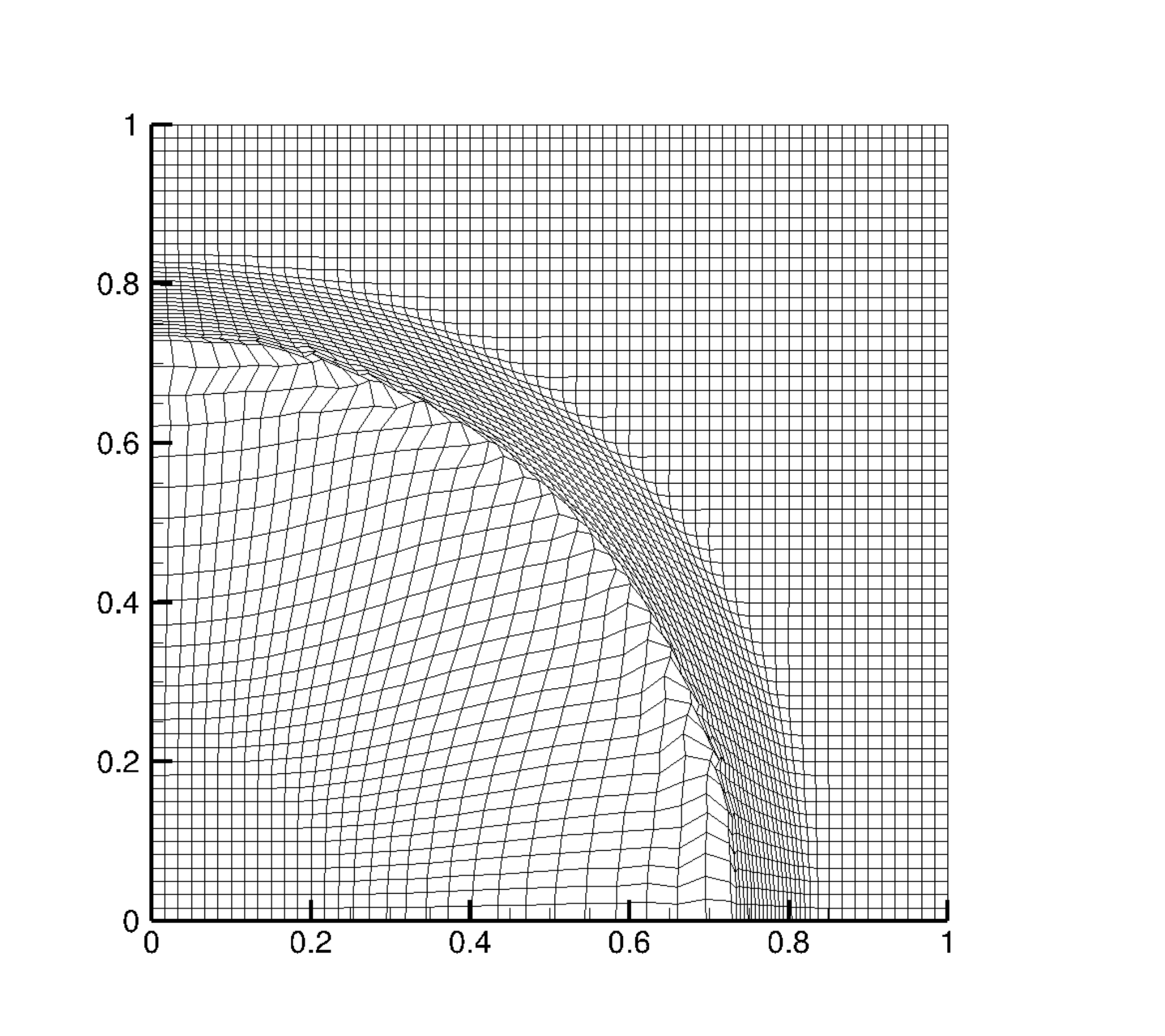}
  }
  \subfigure[$\rho(x,y)$]{
    \includegraphics[width=0.45\textwidth, trim=40 30 30 40, clip]{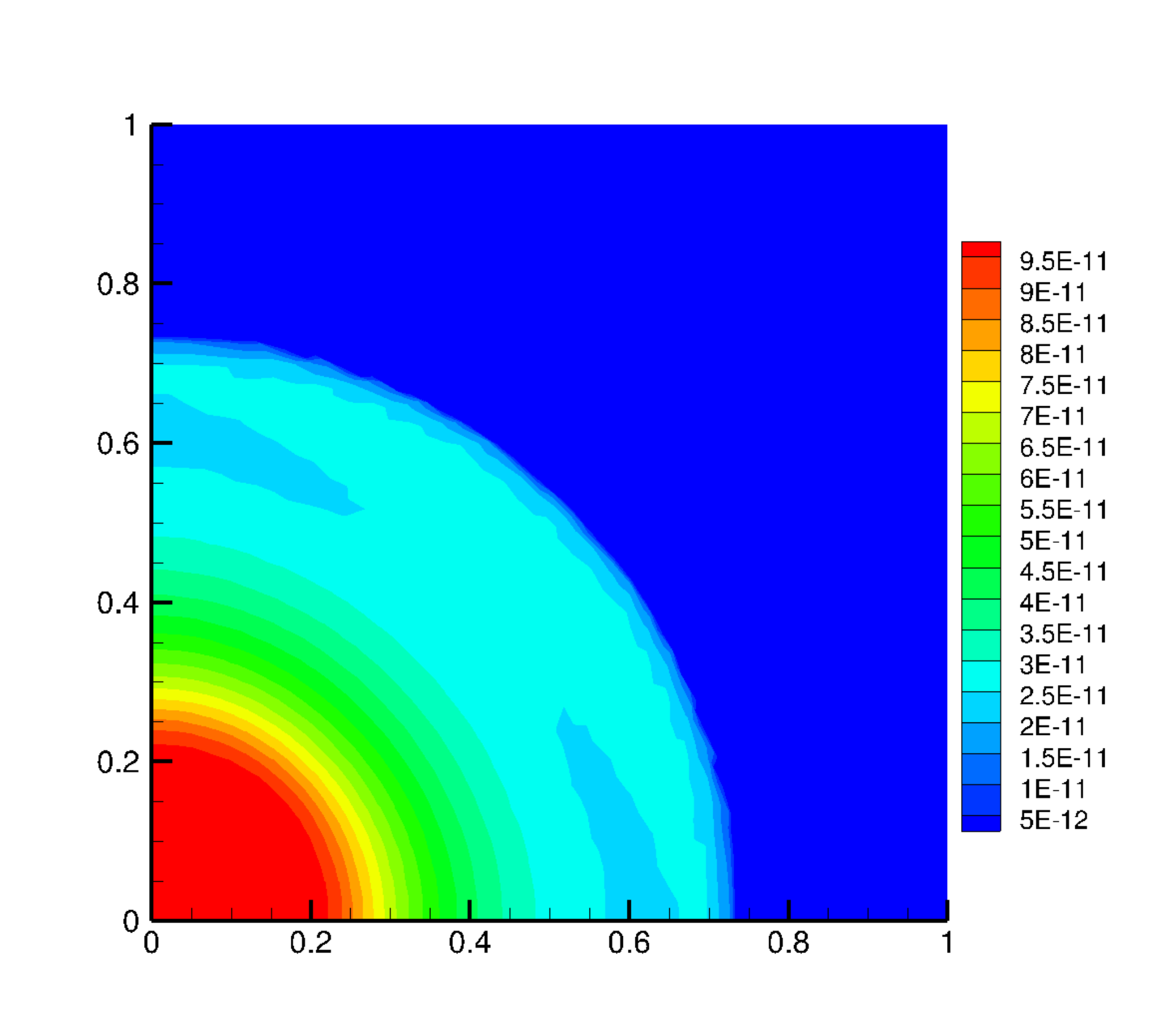}
  }
  \subfigure[$\rho$ along $y=x$]{
    \includegraphics[width=0.45\textwidth, trim=30 30 60 40, clip]{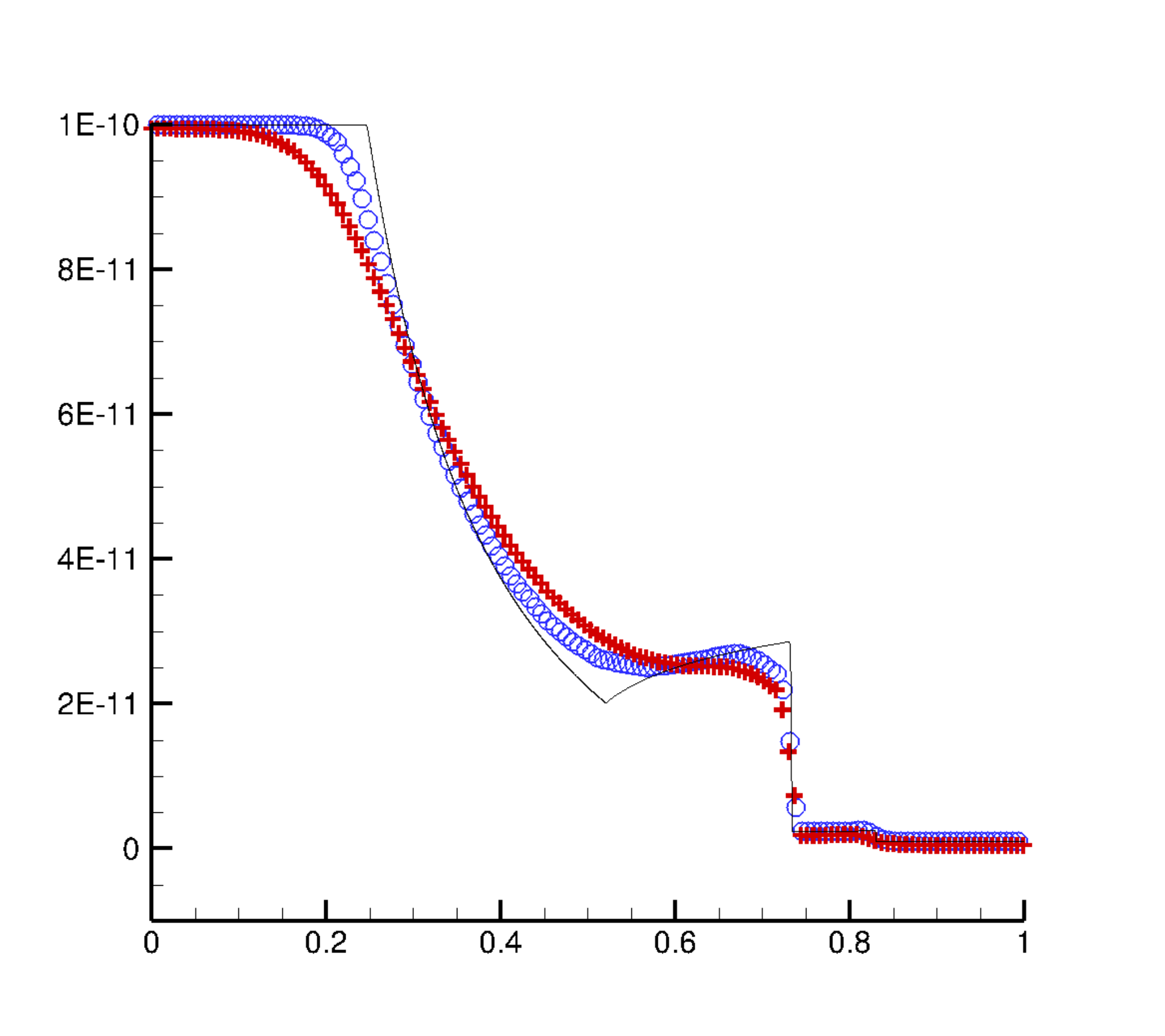}
  }
  \subfigure[$p$ along  $y=x$.]{
    \includegraphics[width=0.45\textwidth, trim=30 30 60 40, clip]{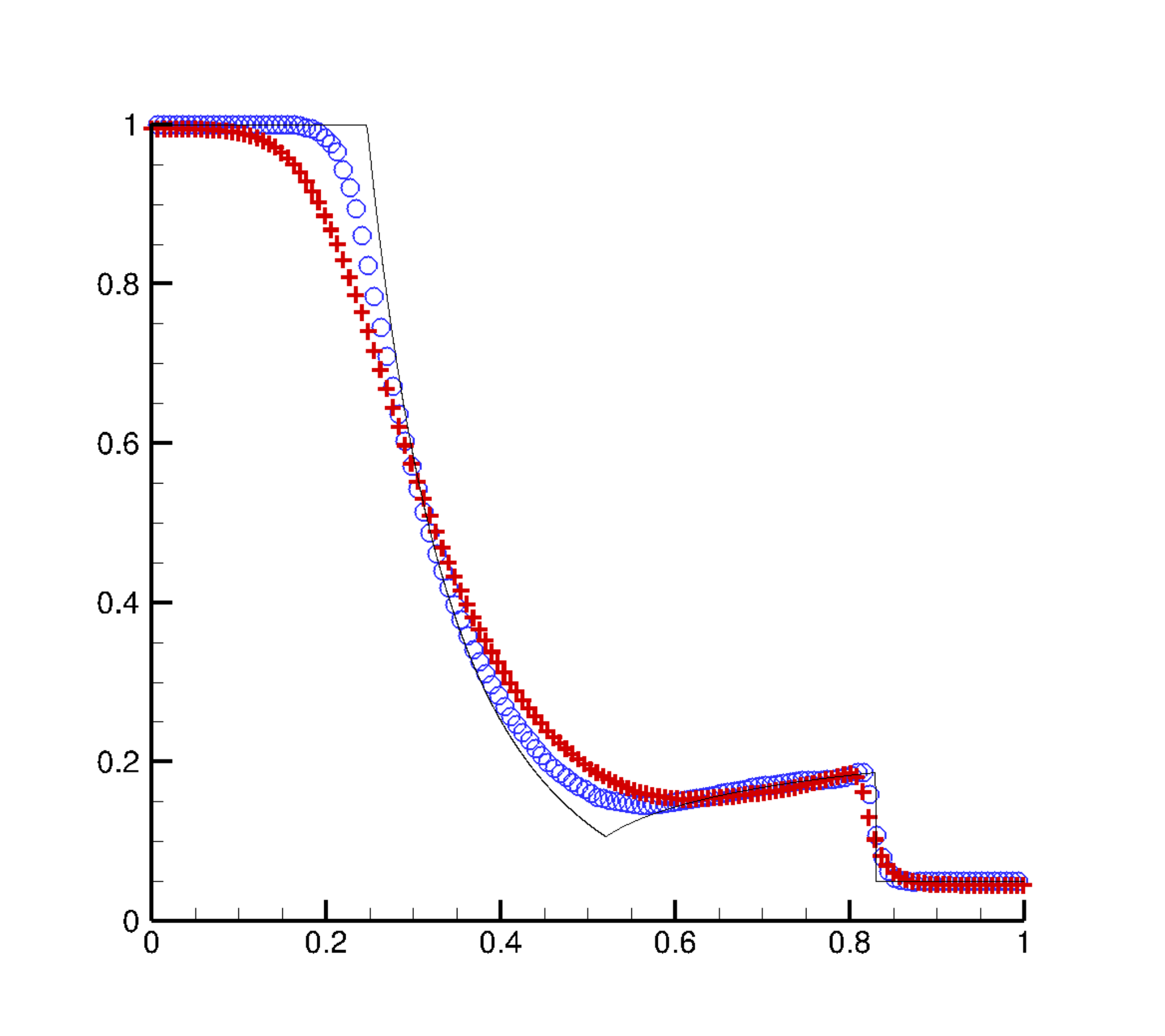}
  }
  \caption{Example \ref{example2}: Mesh and solutions at $t=0.4$.
    The   symbols
    ``$+$'' and ``$\circ$'' denote the solutions obtained by using the first-   and second-order PCP Lagrangian schemes, respectively.}
  \label{fig:WB_cartesian}
\end{figure}

\begin{example}[Blast problem on the polar mesh]\label{example3}
It is to solve the problem in Example \ref{example2} on
an equal-angled polar mesh with $100\times 10$ cells.
Figure \ref{fig3}(a)-(b) shows the mesh and density at $t=0.4$ obtained by using
the second-order PCP Lagrangian scheme, while Figure \ref{fig3}(c)-(d)
plots the density and pressure with respect to the radial radius  obtained by the first- and second-order schemes.
The results show that the contact discontinuity can be exactly resolved and
the second-order PCP Lagrangian scheme gives a better result
than the first-order. Moreover, the  results in  Figure \ref{fig3}  are obviously
better than those in Figure \ref{fig:WB_cartesian}.
\end{example}

\begin{figure}[h!]
  \centering
  \subfigure[Mesh]{
    \includegraphics[width=0.45\textwidth, trim=40 30 30 40, clip]{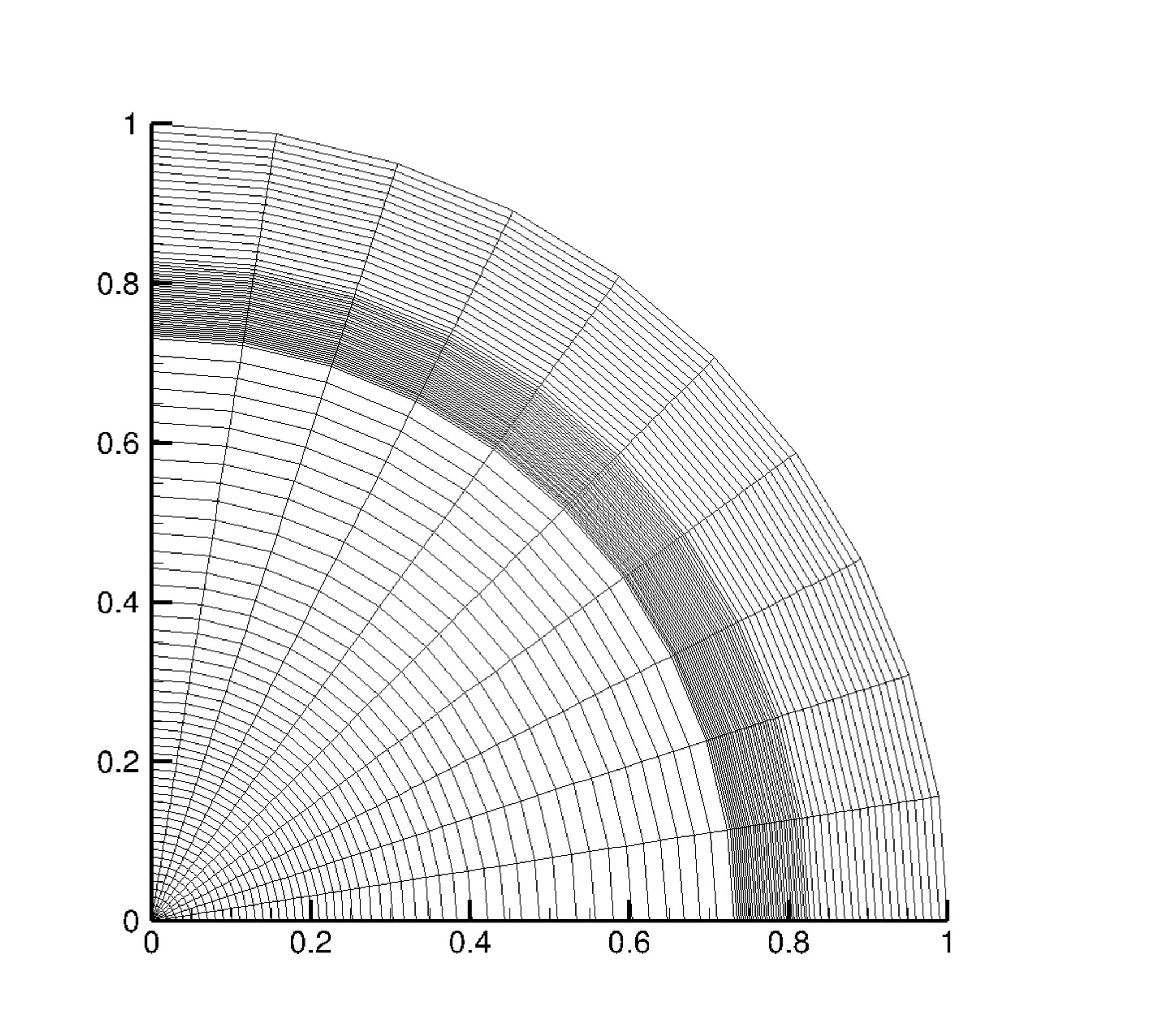}
  }
  \subfigure[$\rho(x,y)$]{
    \includegraphics[width=0.45\textwidth, trim=40 30 30 40, clip]{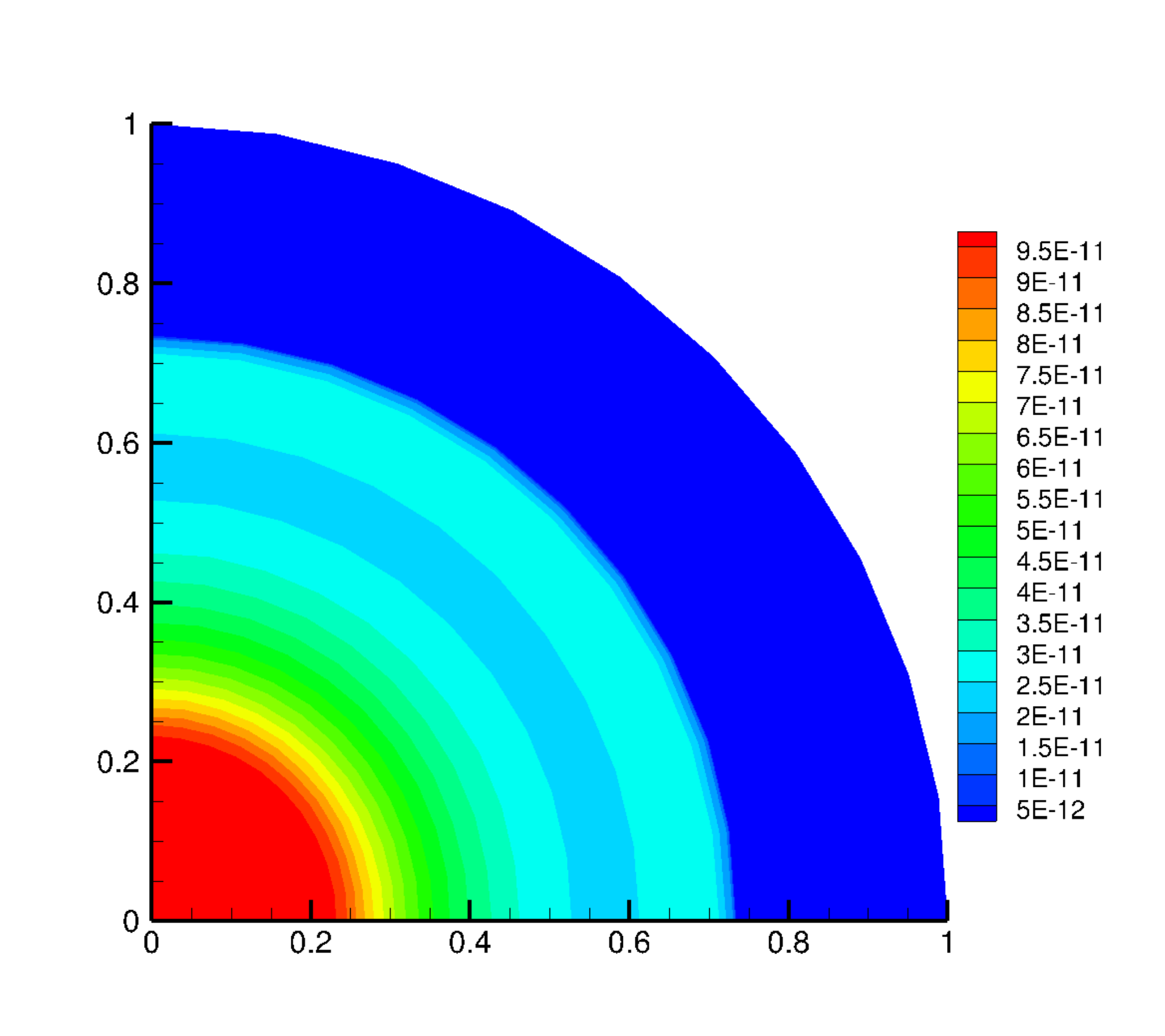}
  }
  \subfigure[$\rho(r)$]{
    \includegraphics[width=0.45\textwidth, trim=30 30 60 40, clip]{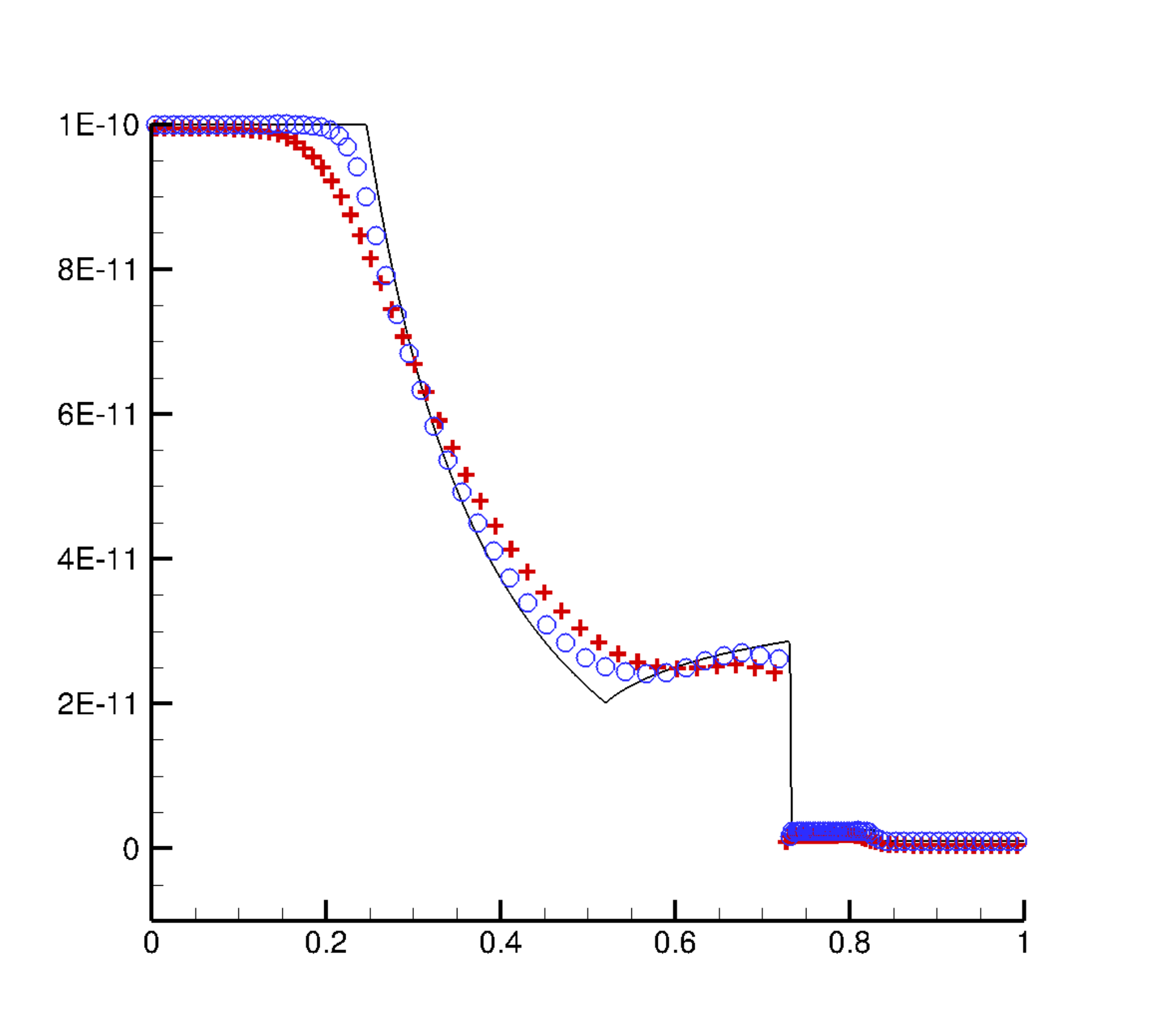}
  }
  \subfigure[$p(r)$]{
    \includegraphics[width=0.45\textwidth, trim=30 30 60 40, clip]{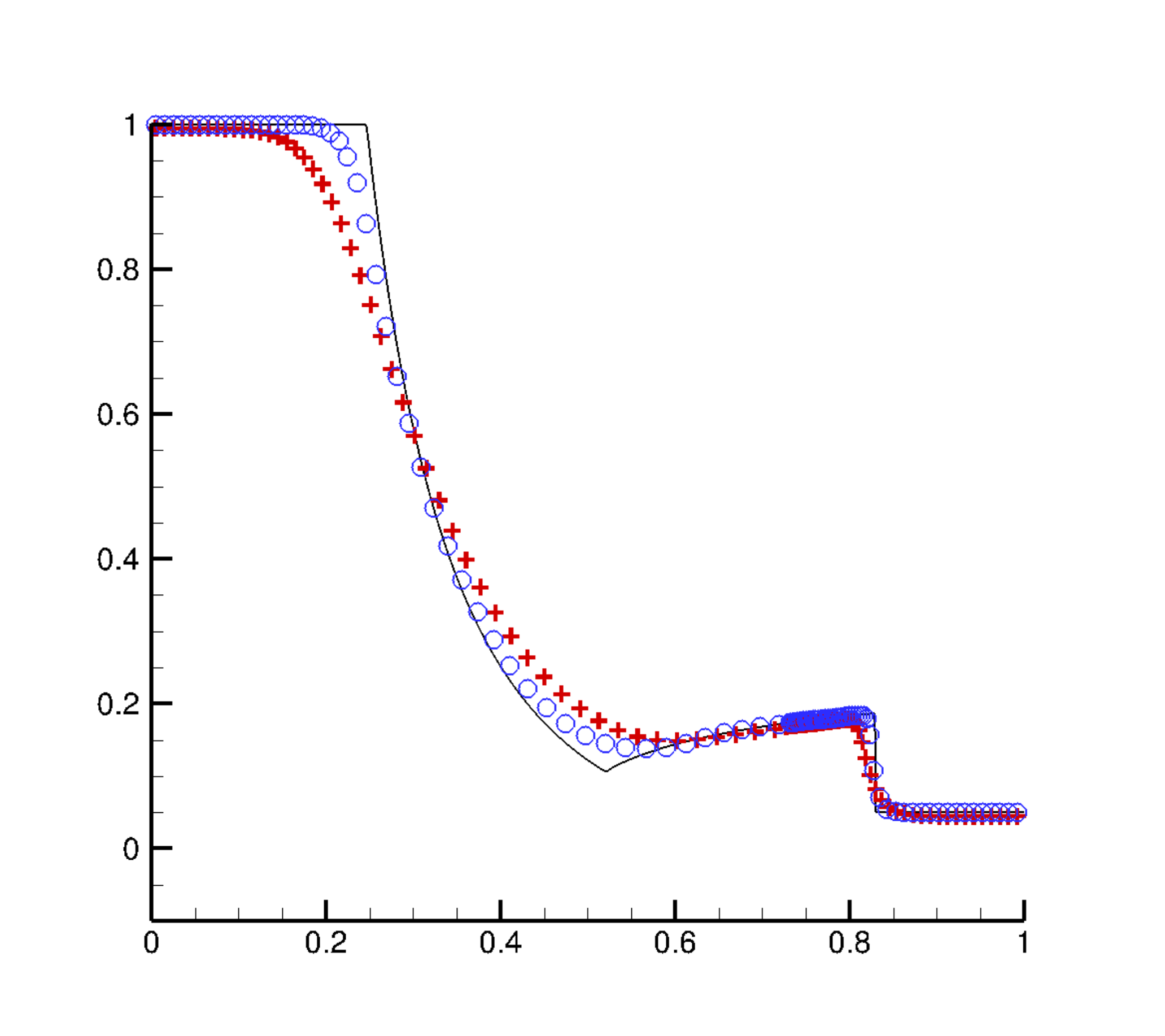}
  }
  \caption{Example \ref{example3}: Mesh  and solutions at $t=0.4$.
 The  symbols
    ``$+$'' and  ``$\circ$'' denote the solutions obtained by the first- and second-order PCP Lagrangian schemes, respectively.}
  \label{fig3}
\end{figure}

\begin{example}[Strong blast problem on the polar mesh]\label{example4}
Consider a strong blast problem with
the initial data
\begin{equation}
  (\rho,u_x,u_y,p)=\begin{cases}
    (1,0,0,1),       & r<0.5,\\
    (1,0,0,10^{-12}), & r>0.5.
  \end{cases}
\end{equation}
It is similar to Example \ref{example3} so that
the flow pattern is similar to that in Example \ref{example3}.
However, it is very challenging for the Lagrangian scheme
because there exists the extremely low pressure and the very narrow region
between the contact discontinuity and the shock wave.
If the scheme is not  PCP, then its calculation may result in failure.

 Figure \ref{fig:SB_polar}(a)-(c)
the pressure, density and the close-up of the density obtained by using
the first and the second order schemes  with $100\times 10$ cells, where
    the solid line is the reference solution obtained by a second-order TVD Eulerian scheme with Lax-Friedrichs flux in  the cylindrical coordinate with $10000$ cells.
    Figure \ref{fig:SB_polar}(d) shows the mesh with $100\times 10$ cells at $t=0.4$ obtained by using
the second-order PCP Lagrangian scheme.
It can be seen that our PCP schemes can capture the narrow region between the
contact discontinuity and the shock wave well, and  the second order scheme is better
than the first-order.
\end{example}

\begin{figure}[h!]
  \centering

  \subfigure[$\rho(r)$]{
    \includegraphics[width=0.45\textwidth, trim=40 20 30 40, clip]{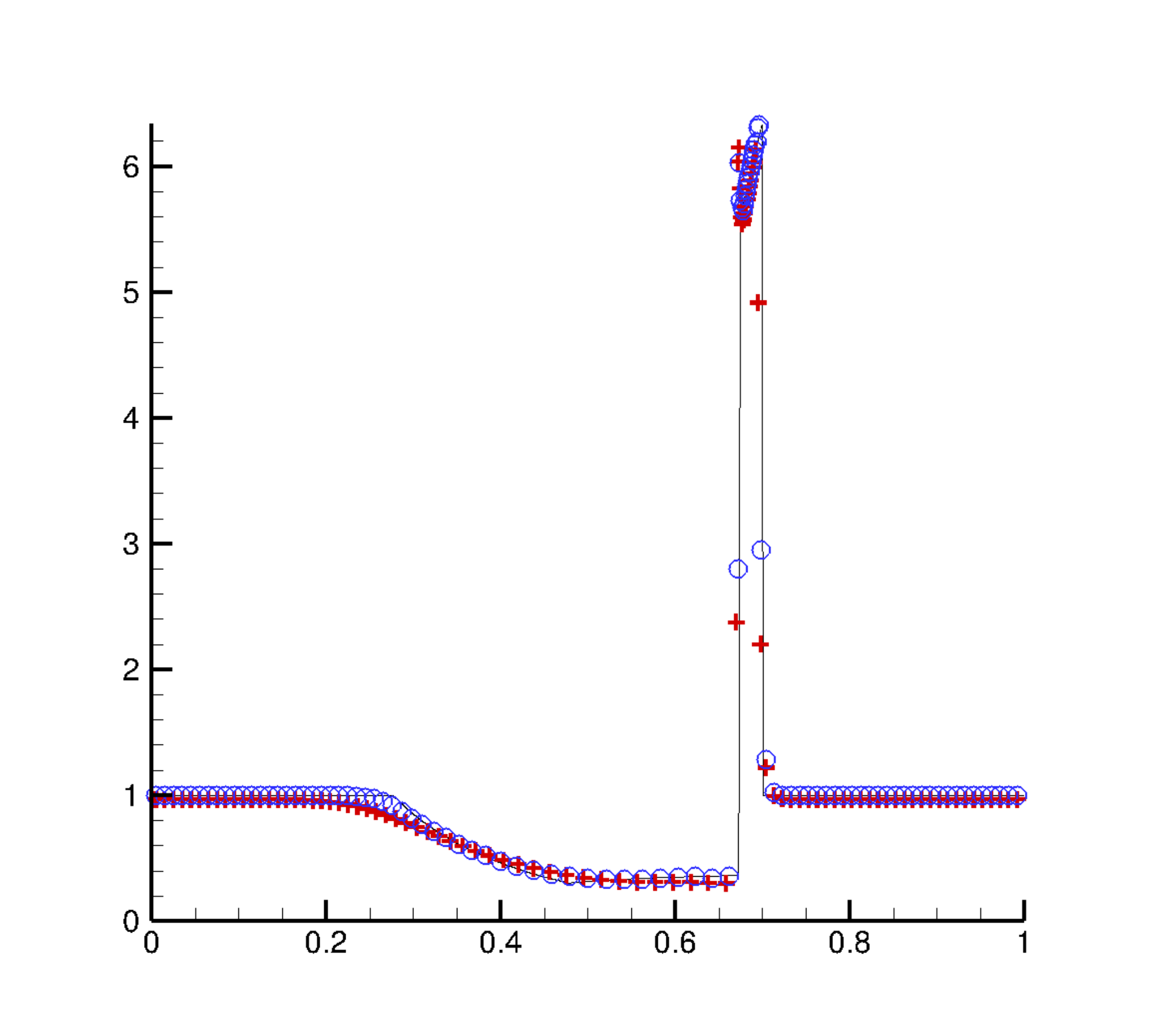}
  }
  \subfigure[Close-up  of $\rho(r)$] {
    \includegraphics[width=0.45\textwidth, trim=40 30 40 40, clip]{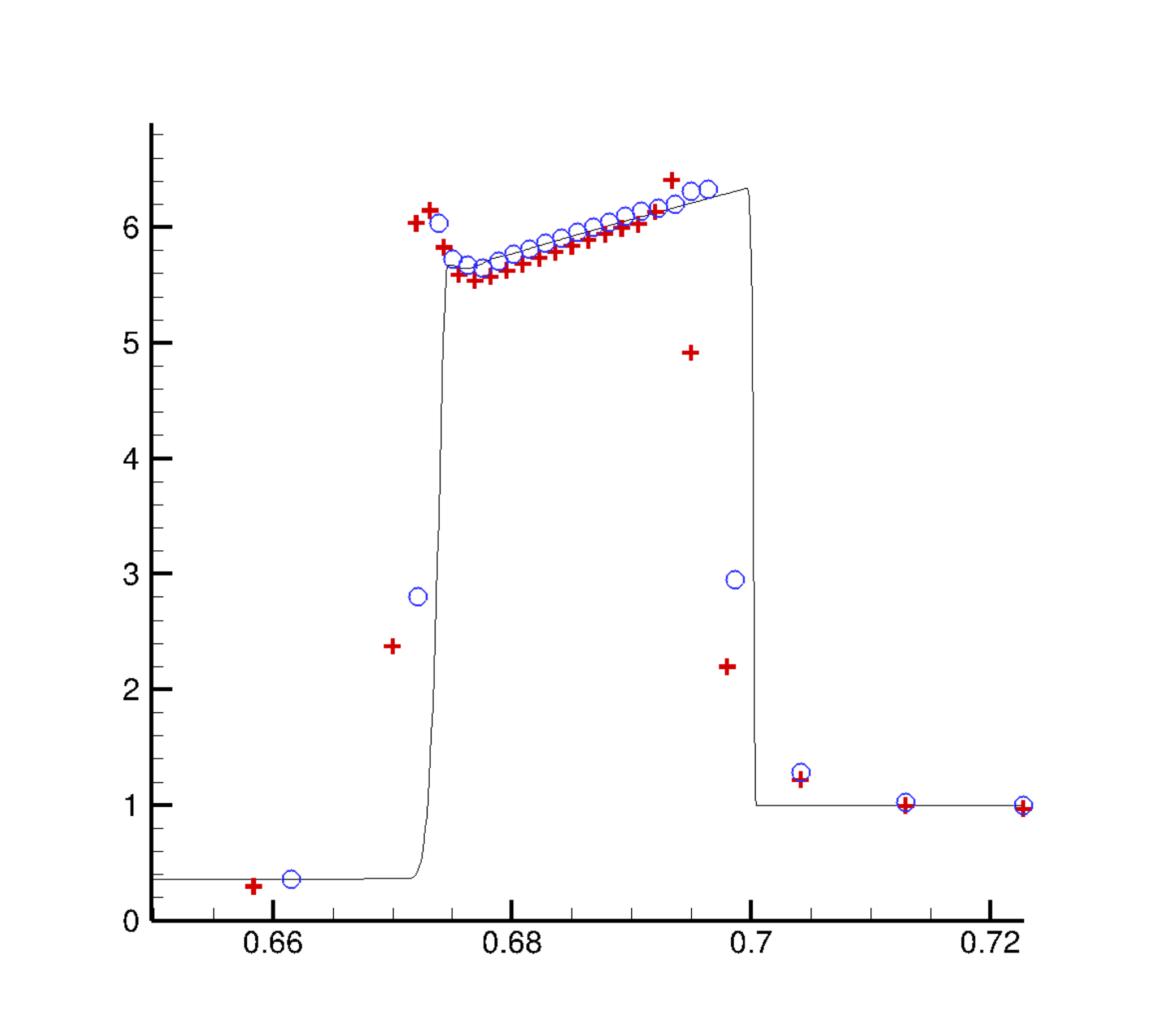}
  }

    \subfigure[$p(r)$]{
    \includegraphics[width=0.45\textwidth, trim=40 30 60 40, clip]{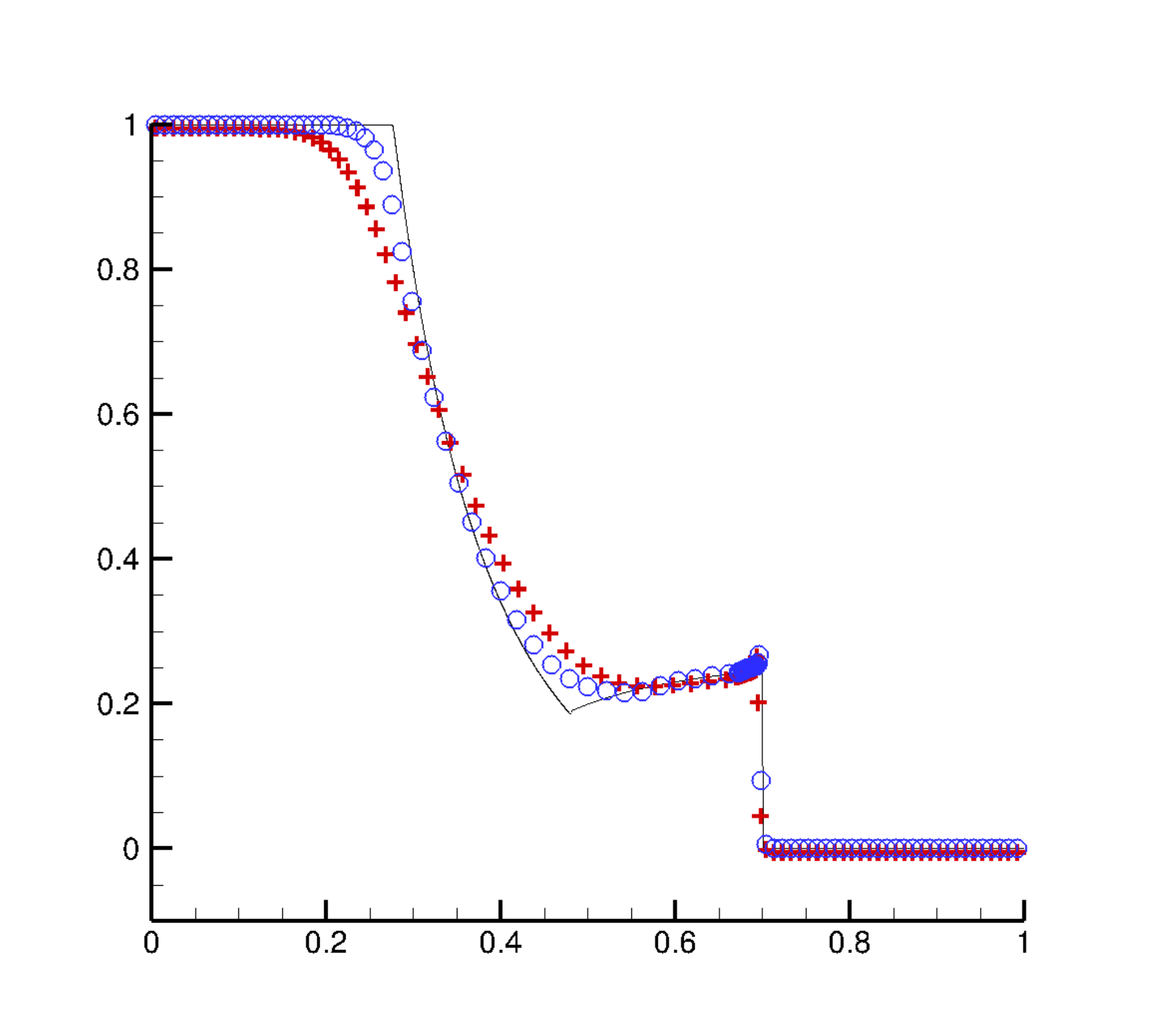}
  }
  \subfigure[Mesh]{
    \includegraphics[width=0.45\textwidth, trim=40 20 30 40, clip]{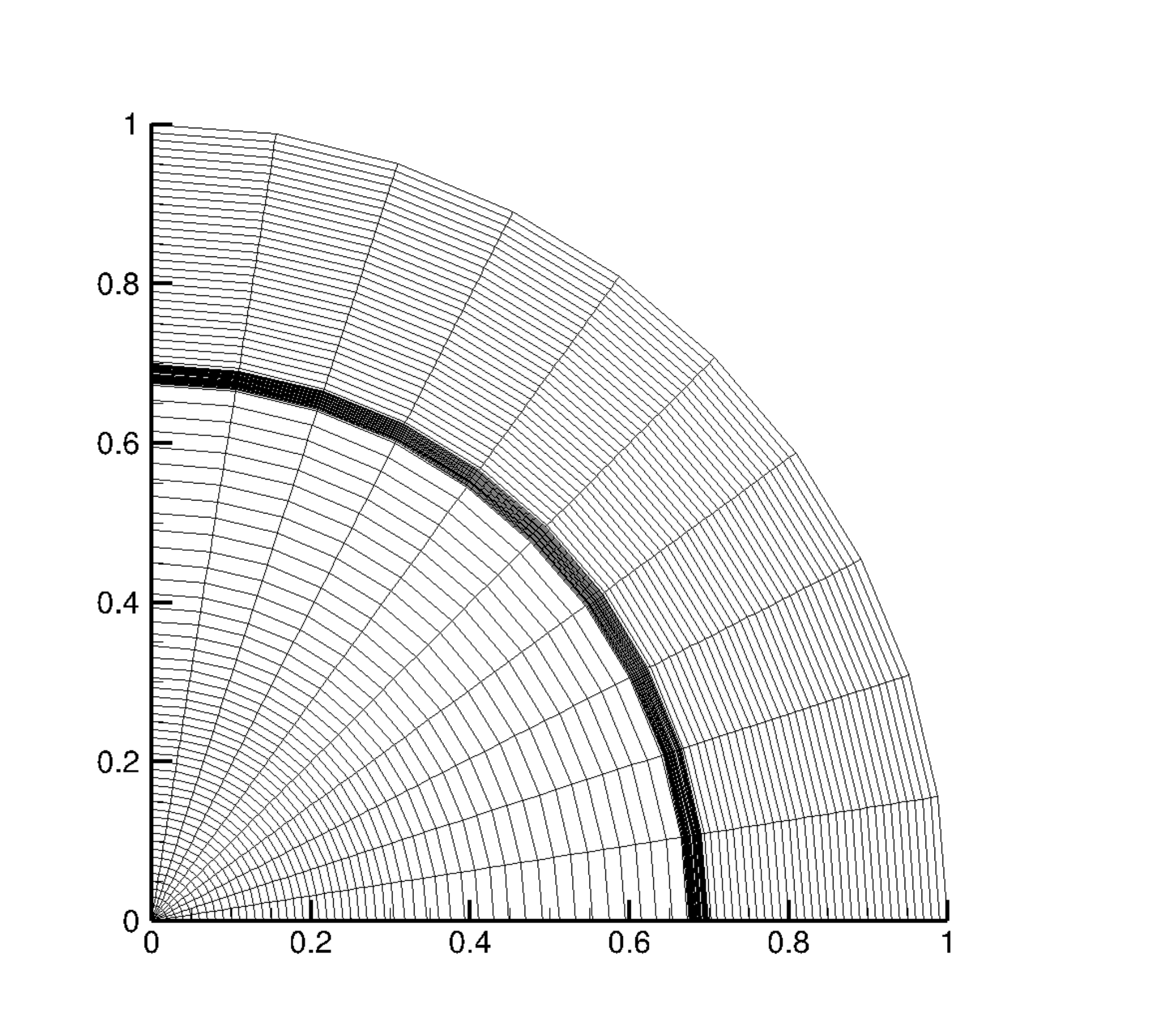}
  }
  \caption{Example \ref{example4}: Solutions and mesh at $t=0.4$.
    The  symbols
    ``$+$'' and   ``$\circ$'' denote the solutions obtained by using the first-
    and second-order PCP Lagrangian schemes, respectively.}
  \label{fig:SB_polar}
\end{figure}

\begin{example}[Implosion problem]\label{example5}
It is an implosion problem, similar to the Noh problem in the non-relativistic case.
The computational domain in
the polar coordinates $(r,\theta)$ is chosen as  $[0,1]\times [0,\pi/2]$
and divided into an initial equal-angled polar mesh.
The initial density and pressure of the fluid in the domain is $1$ and $10^{-12}$
respectively,  and an inward radial velocity  is set as $0.9$.
A strong shock wave is generated by bringing the cold gas to rest at the
origin, and will converge to the origin.  Because the low pressure will appear
in the solution, the negative pressure may be easily produced by  using the non-PCP scheme.

Figure \ref{fig:Noh}(a)-(b)
shows the   mesh with $100\times 10$ cells and the density contour at $t=0.6$
obtained by  the second-order PCP Lagrangian scheme,
while Figure \ref{fig:Noh}(c)-(d)  plots
the radial density and pressure obtained by using the first- and second-order
schemes, where
the solid line is the reference solution obtained by using a first-order Eulerian scheme with Lax-Friedrichs flux in the cylindrical coordinate and with $10000$ cells.
 It is shown that the PCP scheme can be
successfully  simulate such extreme relativistic fluid flow, and  the resolution
of the second-order scheme is obviously better than the first-order even though
there exists a wall heating phenomenon.
\end{example}

\begin{figure}[h!]
  \centering
  \subfigure[Mesh]{
    \includegraphics[width=0.45\textwidth, trim=40 30 60 40, clip]{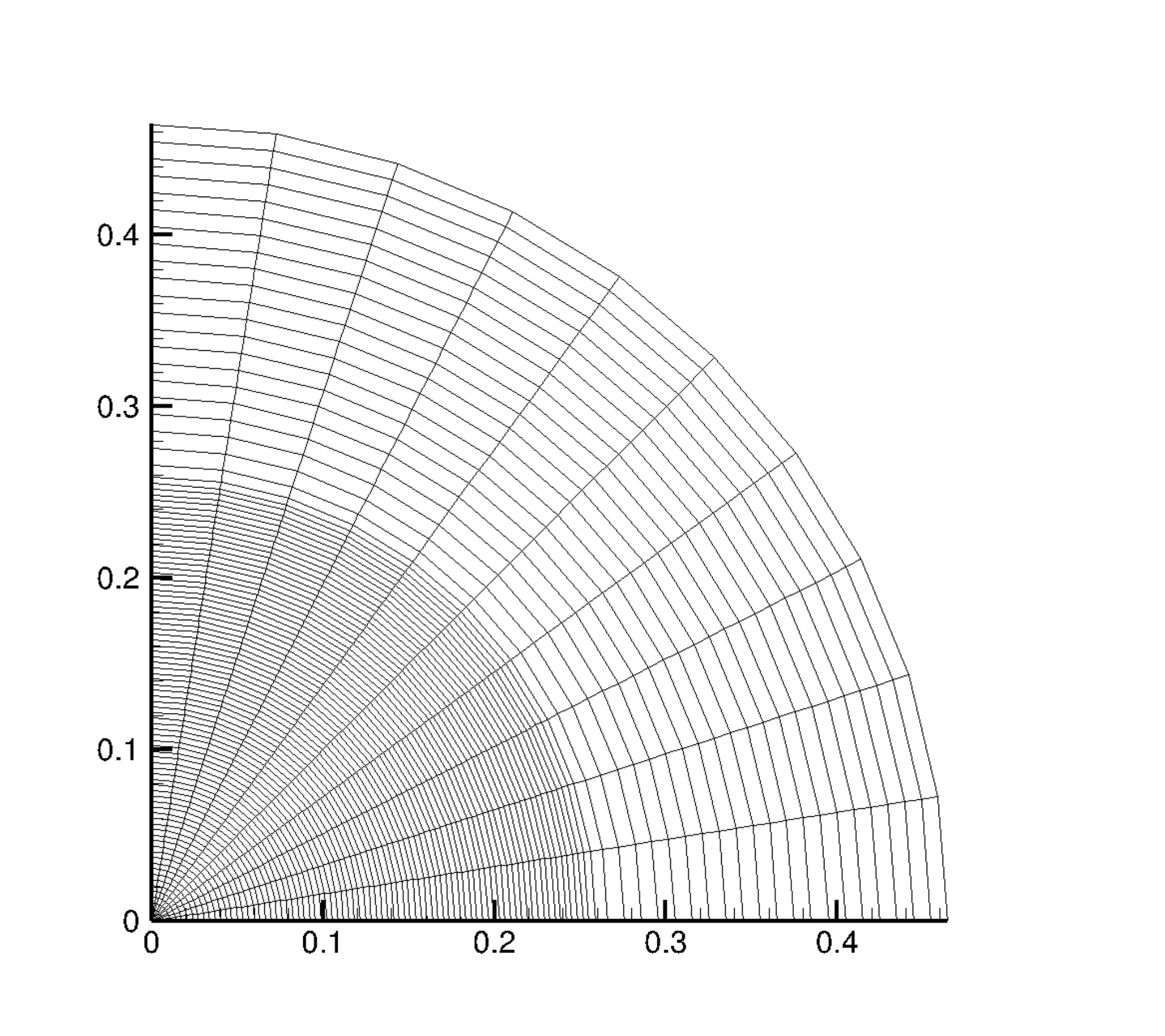}
  }
  \subfigure[$\rho(x,y)$]{
    \includegraphics[width=0.45\textwidth, trim=40 30 60 40, clip]{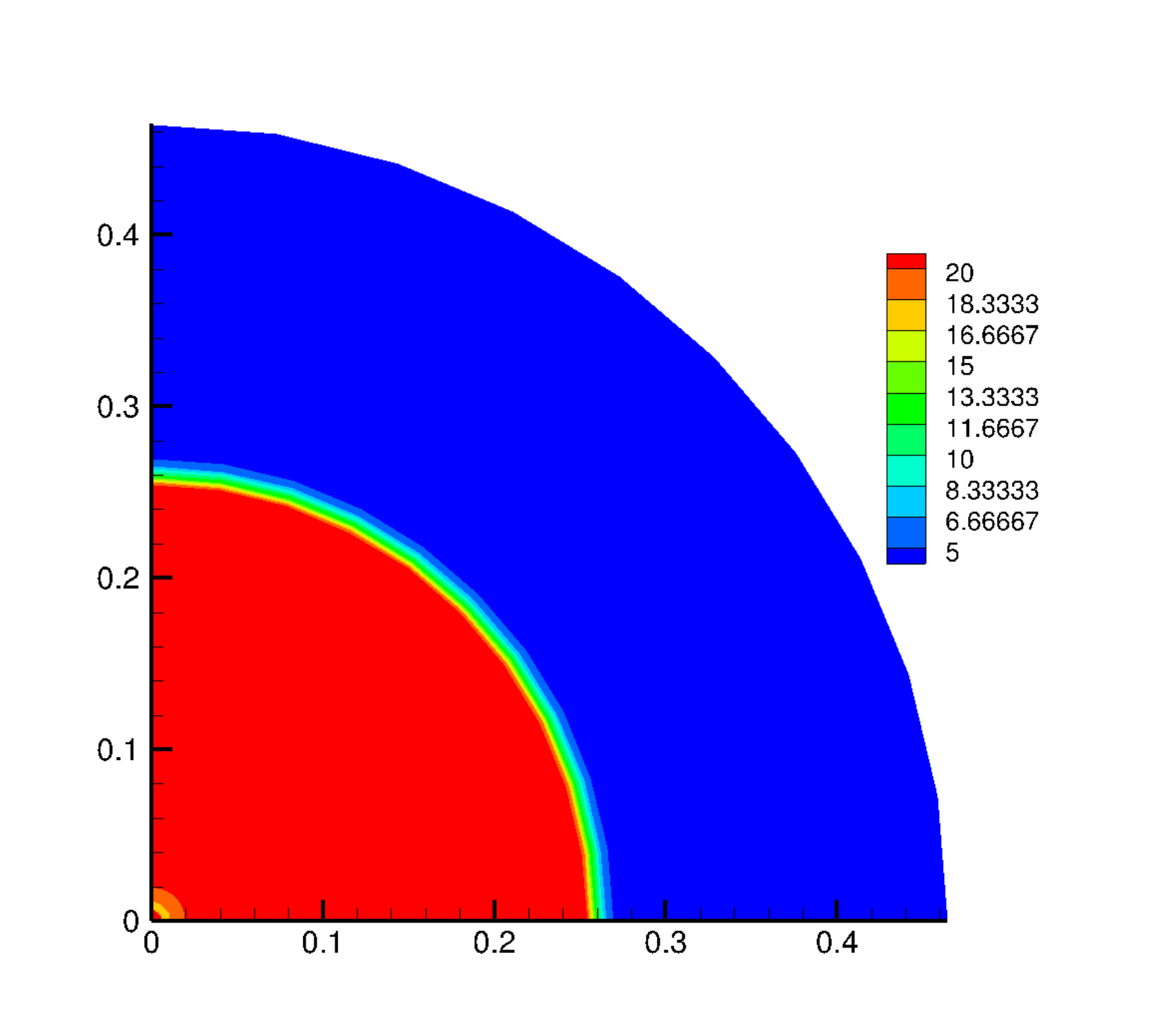}
  }
  \subfigure[$\rho(r)$]{
    \includegraphics[width=0.45\textwidth, trim=40 30 60 40, clip]{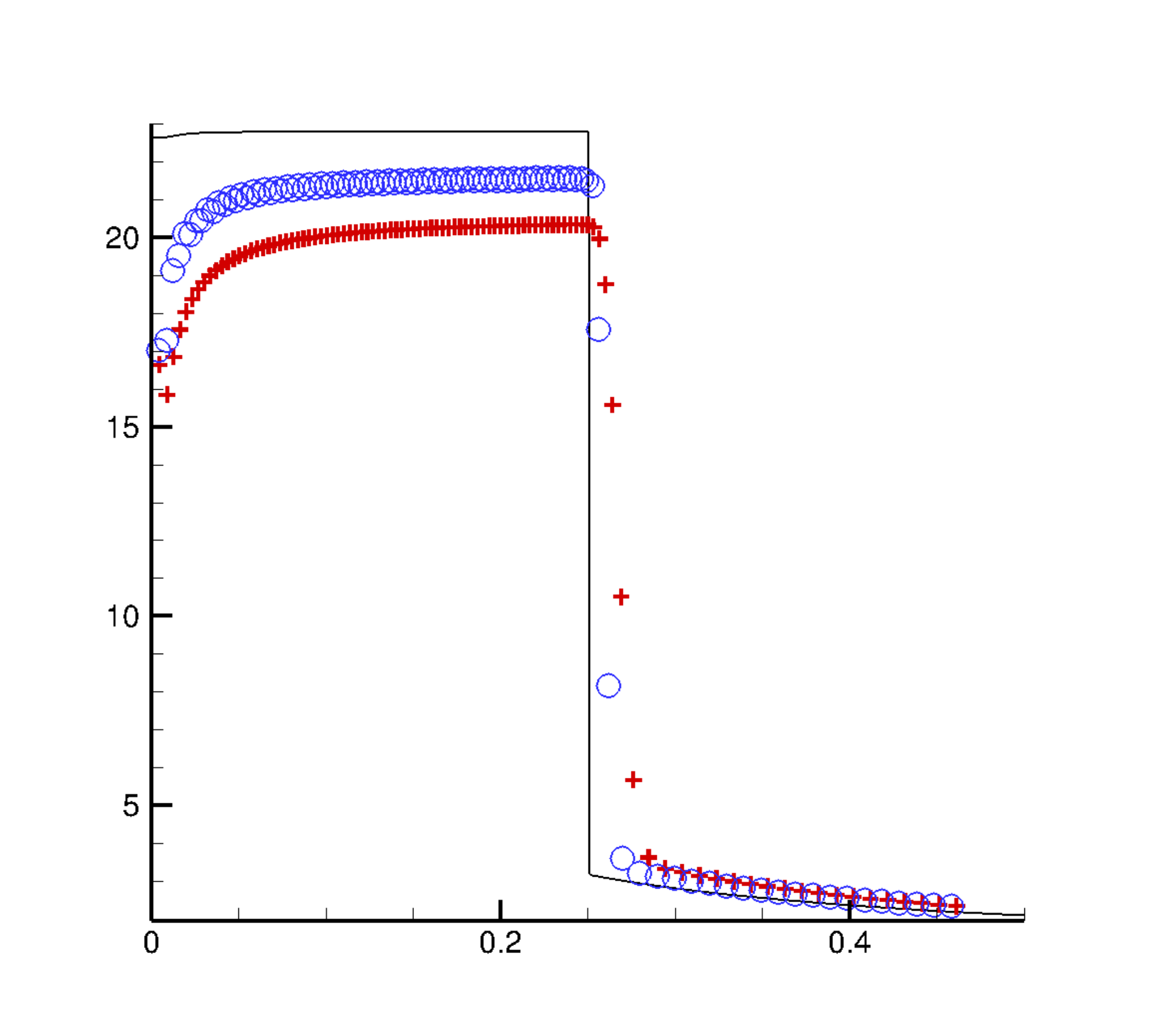}
  }
  \subfigure[$p(r)$]{
    \includegraphics[width=0.45\textwidth, trim=40 30 60 40, clip]{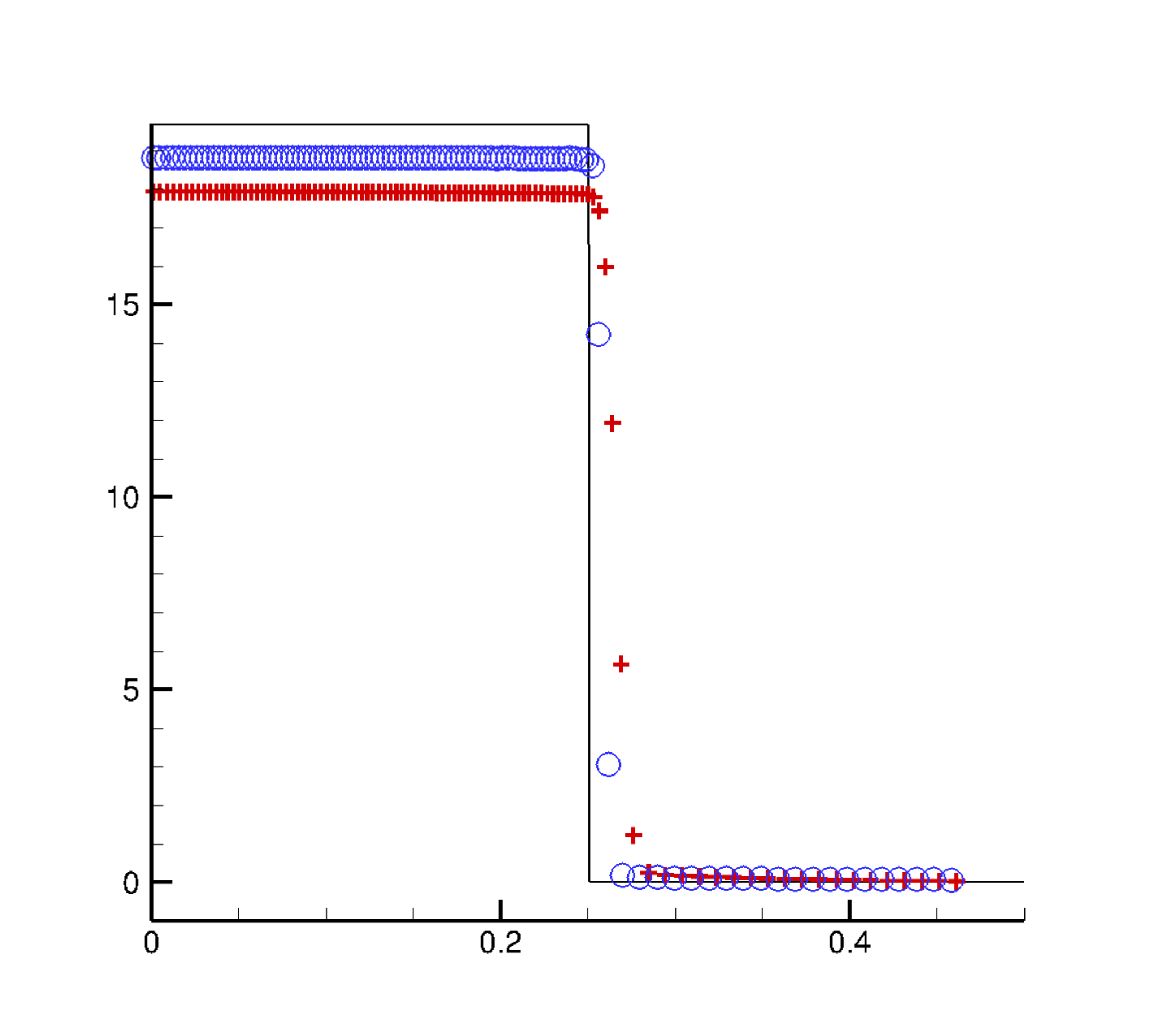}
  }
  \caption{Example \ref{example5}:  Mesh with $100\times 10$ cells and solutions at $t=0.6$.
     The  symbols
    ``$+$'' and ``$\circ$'' denote the solutions obtained by using
    the first- and second-order PCP Lagrangian schemes, respectively.}
  \label{fig:Noh}
\end{figure}

\begin{example}[ICF-like test]\label{example6}
The last problem is about a ICF-like test, which is similar to Example 6.7 in \cite{maire}.
The initial target is defined in the polar coordinates $(r,\theta)$
by $[0,1]\times [0,\pi/2]$.
There are two kinds of materials in the target with the same adiabatic
index $\Gamma=5/3$.
The target is divided into the internal part whose radius is $0.9$ and the external shell.
The initial condition is specified as follows
\begin{equation}
  (\rho,u_x,u_y,p)=\begin{cases}
    (0.01,0,0,5\times 10^{9}),       &  r<0.9,\\
    (1,0,0,10^{11}), & r>0.9,
  \end{cases}
\end{equation}
and on the external boundary of the shell, an additional pressure is added
as \begin{equation}
  \widetilde{p}=10^{12}\times(1+0.05\cos(6\theta)).
\end{equation}
The period of the perturbation of $\widetilde{p}$ is $\pi/3$.

 Figures \ref{fig:ICF1}-\ref{fig:ICF2} show the mesh, density $\rho$,
  radial velocity $u_r$, and the pressure $p$ at $t=1$ obtained
 by using the first- and second-order PCP Lagrangian schemes with $100\times 30$ cells, respectively.
We can see that the outmost mesh becomes sunken periodically
due to the perturbation in the additional pressure, and the mesh points gather
near the interface between the internal part and the external
shell. Obviously, the second-order PCP Lagrangian scheme gives a better resolution
near the interface than the first-order.
\end{example}

\begin{figure}[h!]
  \centering
  \subfigure[Mesh]{
    \includegraphics[width=0.45\textwidth, trim=10 10 30 30, clip]{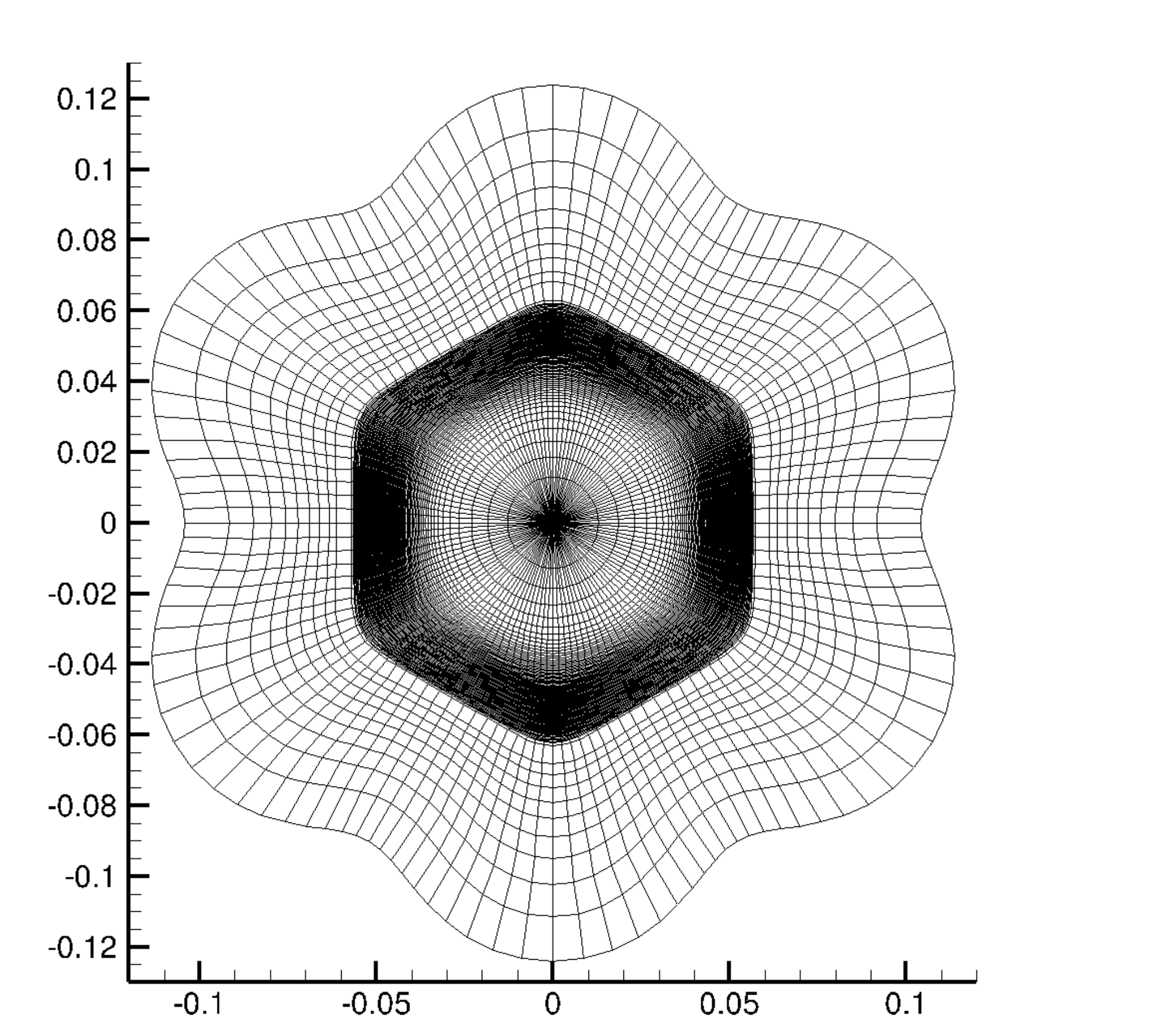}
  }
  \subfigure[$\rho$]{
    \includegraphics[width=0.45\textwidth, trim=10 10 30 30, clip]{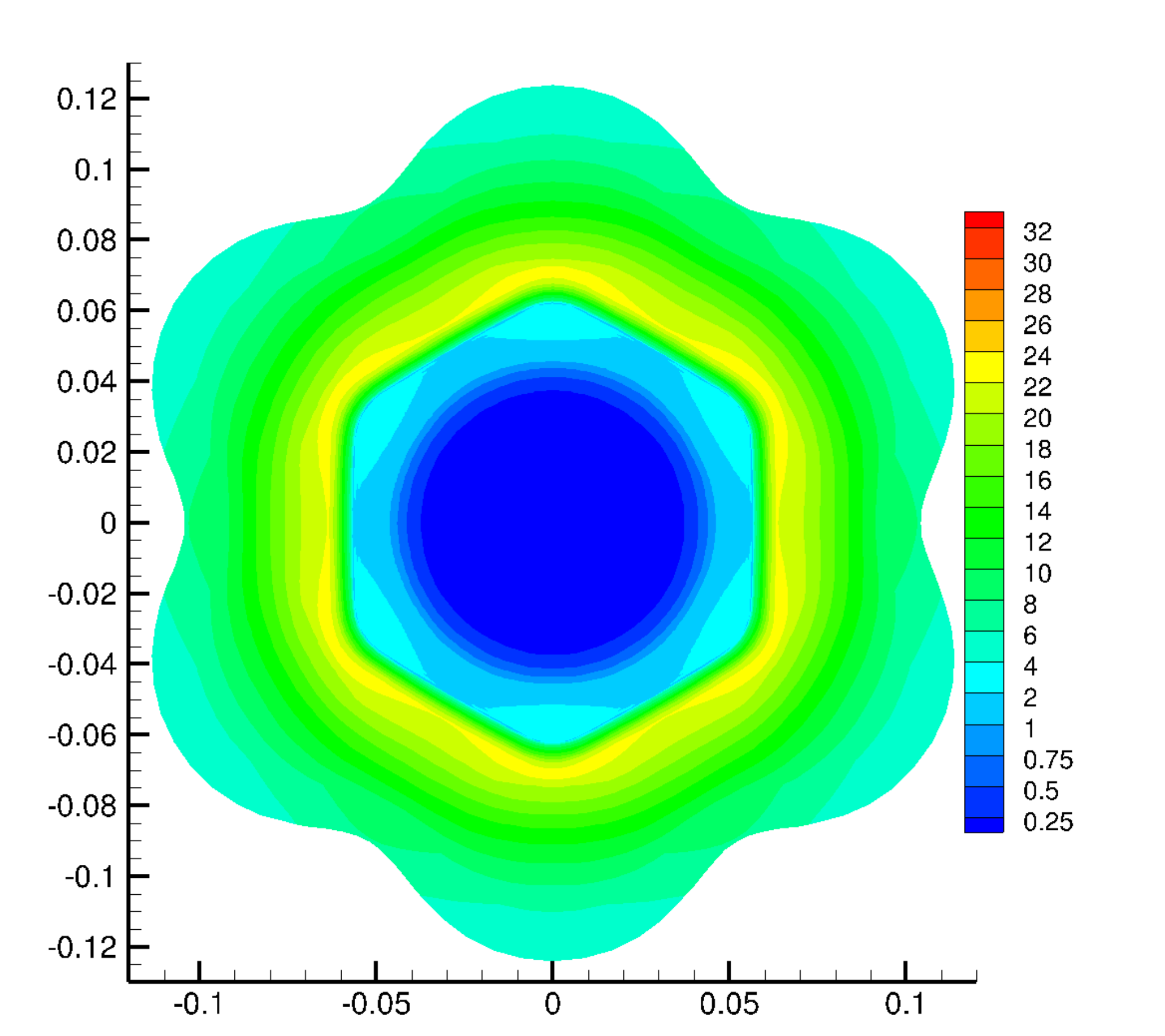}
  }
  \subfigure[$\sqrt{u^2+v^2}$]{
    \includegraphics[width=0.45\textwidth, trim=10 10 30 30, clip]{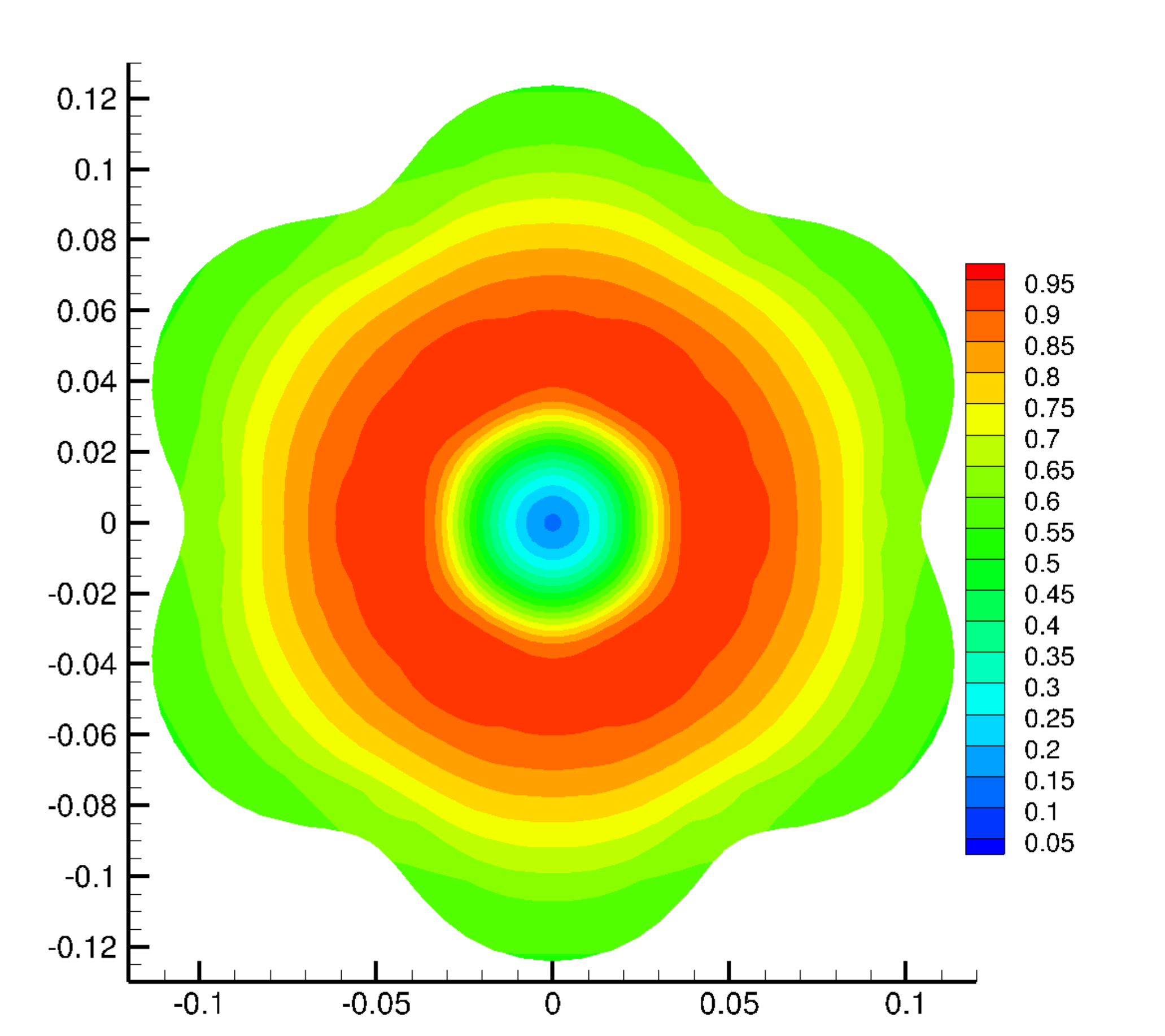}
  }
  \subfigure[$p$]{
    \includegraphics[width=0.45\textwidth, trim=10 10 20 30, clip]{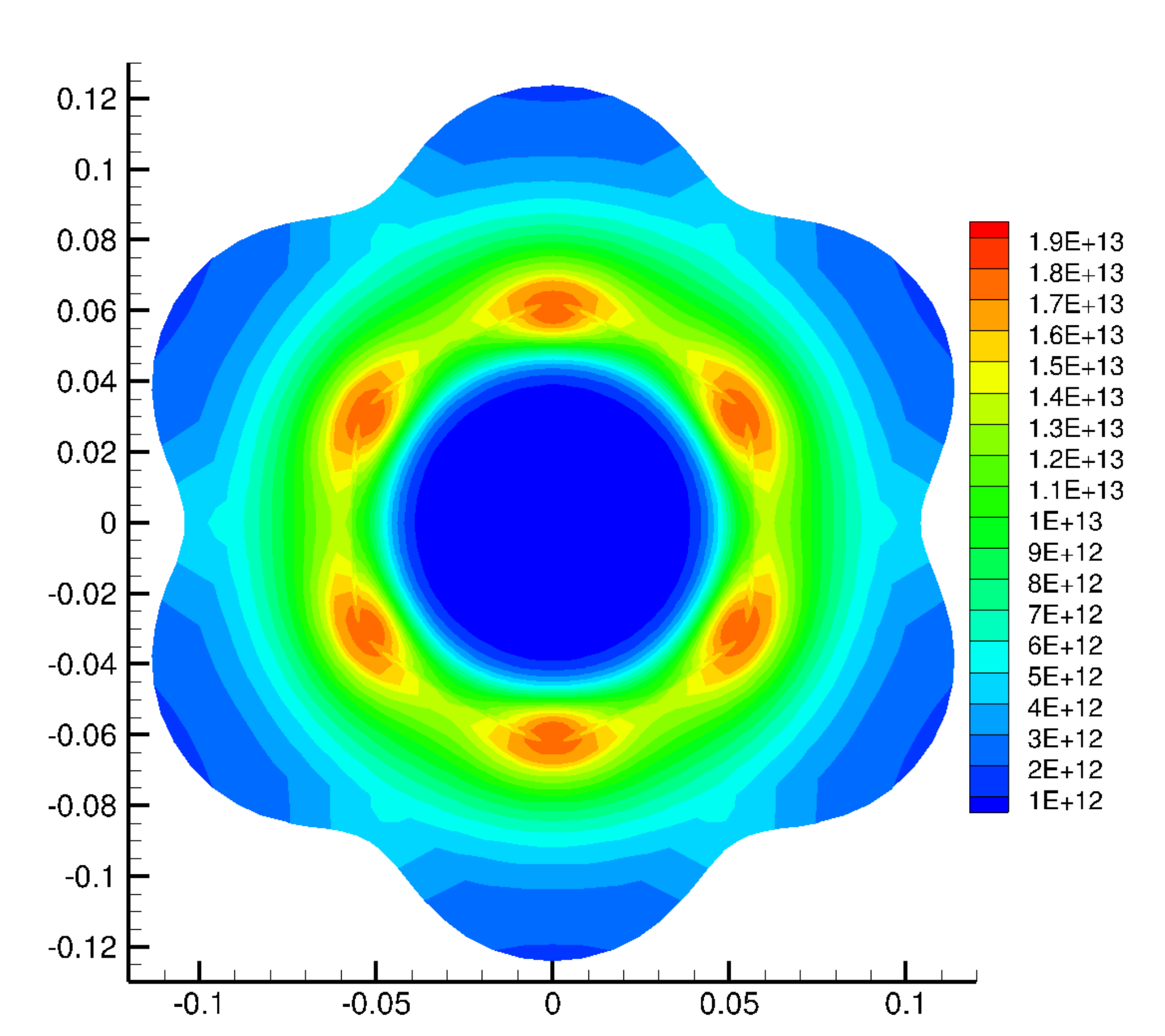}
  }
  \caption{Example \ref{example6}: Mesh and solutions at $t=1$ obtained by using the first-order
  scheme with  $100\times 30$ cells.}
  \label{fig:ICF1}
\end{figure}

\begin{figure}[h!]
  \centering
  \subfigure[Mesh]{
    \includegraphics[width=0.45\textwidth, trim=10 10 30 30, clip]{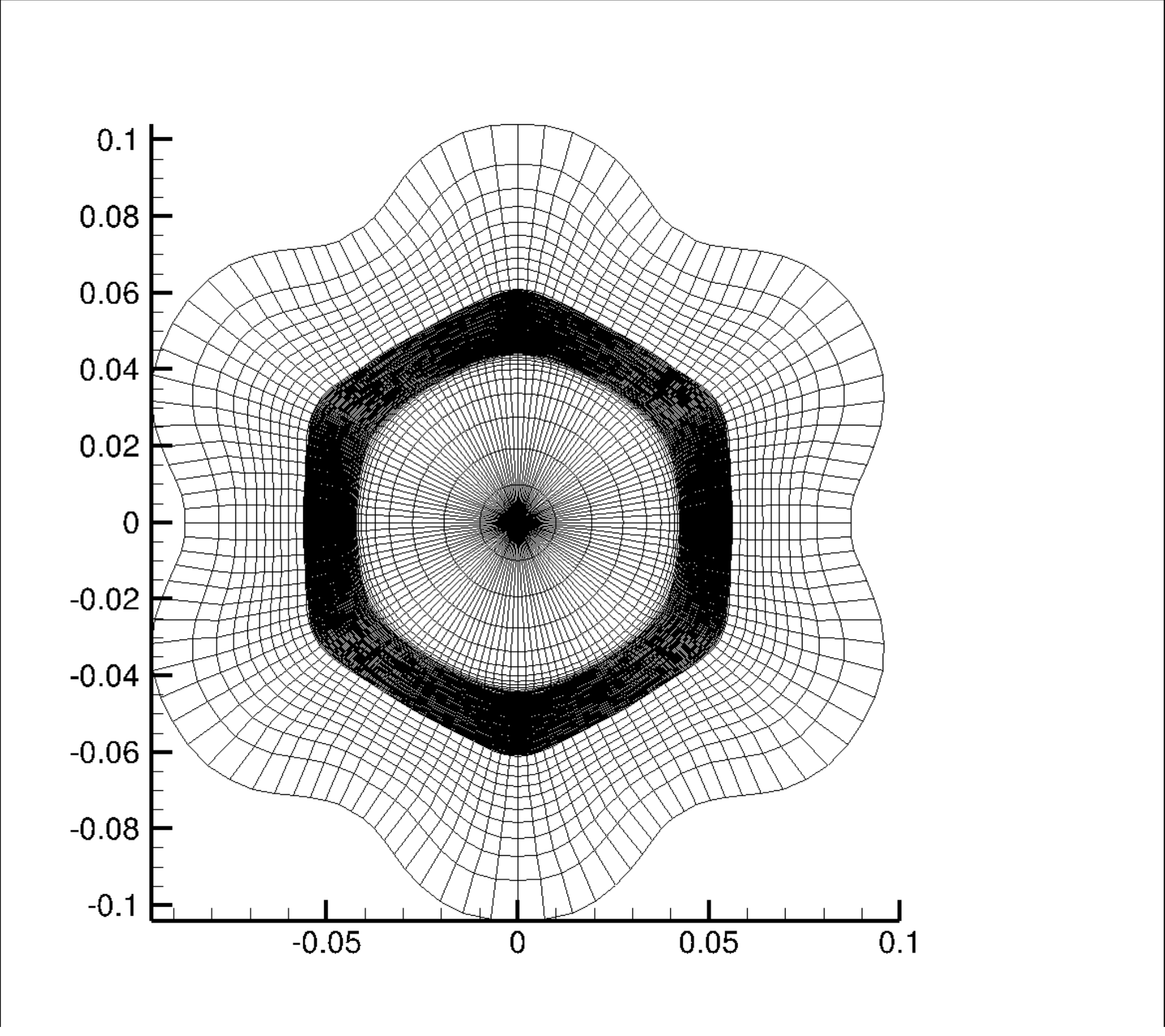}
  }
  \subfigure[$\rho$]{
    \includegraphics[width=0.45\textwidth, trim=10 10 30 30, clip]{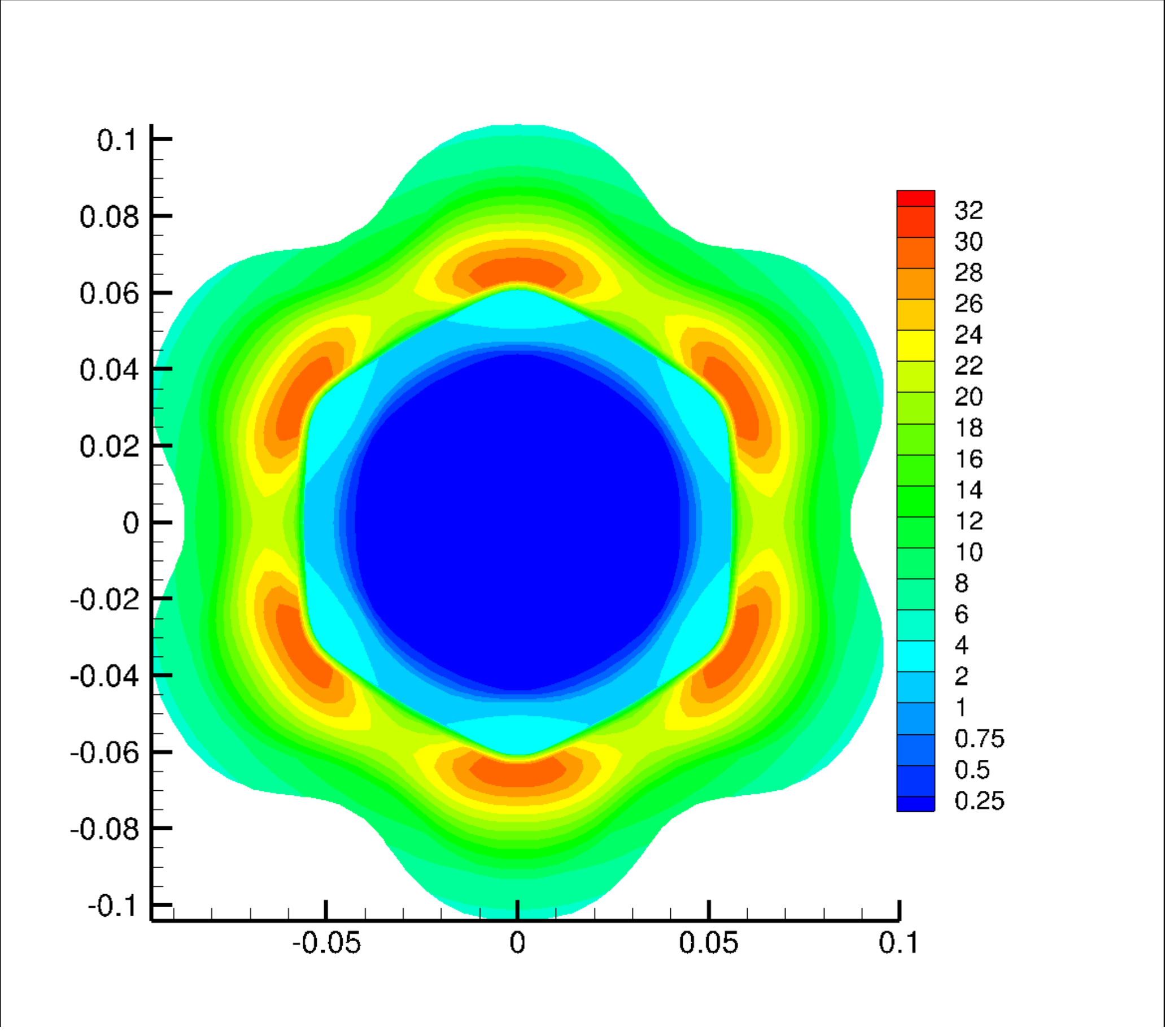}
  }
  \subfigure[$\sqrt{u^2+v^2}$]{
    \includegraphics[width=0.45\textwidth, trim=10 10 30 30, clip]{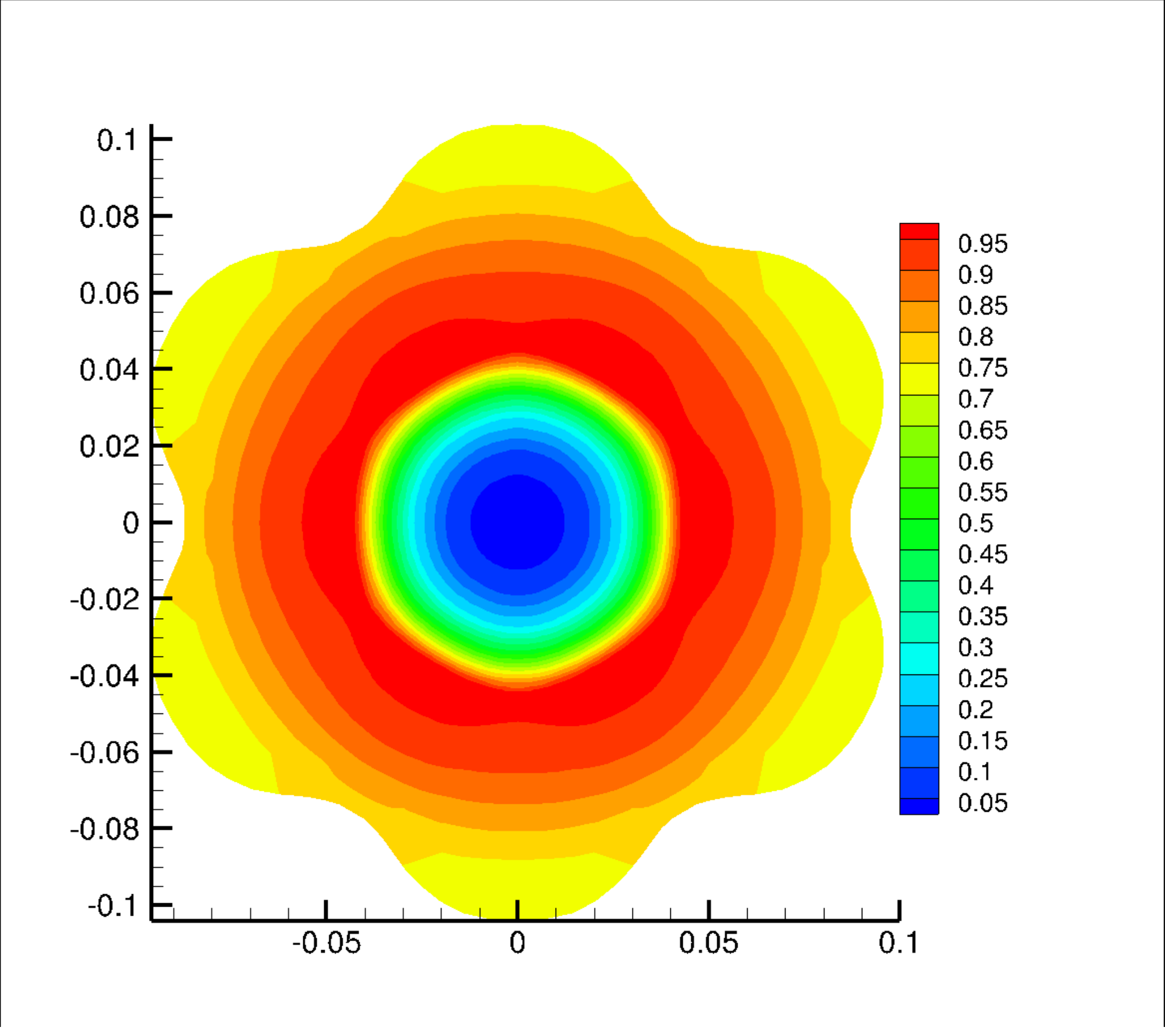}
  }
  \subfigure[$p$]{
    \includegraphics[width=0.45\textwidth, trim=10 10 20 30, clip]{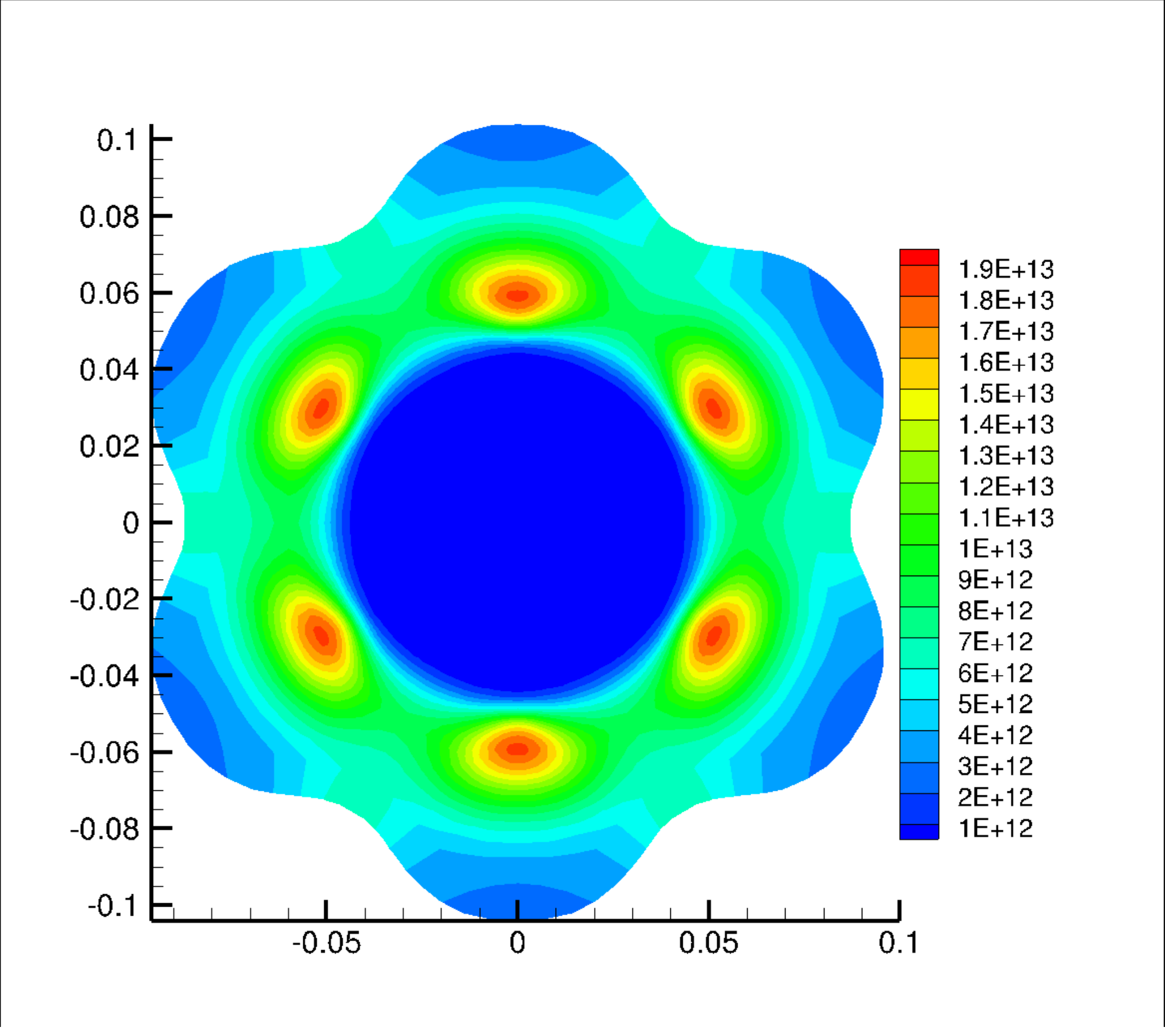}
  }
  \caption{Same as Figure \ref{fig:ICF1} except for the second-order scheme.}
  \label{fig:ICF2}
\end{figure}

\section{Conclusions}\label{section:conclusion}

This paper studied the physical-constraints-preserving (PCP) Lagrangian finite volume schemes for
one- and two-dimensional special relativistic hydrodynamic (RHD) equations.
Our attention was paid to the Lagrangian schemes with the HLLC Riemann solver.
First, we proved that the intermediate states in the HLLC Riemann solver were
 admissible or PCP (that is, the rest-mass density and pressure are positive
 and the fluid velocity magnitude is less than the speed of light) when the HLLC wave speeds were estimated suitably.
 It was worth noting that such PCP property has been observed in \cite{mignone},
 but  no rigorously mathematical proof was given there since
 it was confront with considerable difficulty and  pretty challenging.
Then we showed that the first-order accurate Lagrangian scheme with the HLLC
Riemann solver and forward Euler time discretization was PCP
and developed  the higher-order accurate PCP Lagrangian schemes
by using the high-order accurate strong stability preserving (SSP)
time discretizations, the WENO reconstruction procedure and the scaling PCP limiter.
Finally, several one- and two-dimensional numerical experiments were conducted to
demonstrate the accuracy and the effectiveness of the PCP Lagrangian schemes in solving
the special RHD problems involving  large Lorentz factor, or low
rest-mass density or low pressure or strong discontinuities etc.

\section*{Acknowledgements}
This work was partially supported by the Special Project on High-performance Computing
under the National Key R\&D Program (No. 2016YFB0200603), Science Challenge Project
(No. JCKY2016212A502), and the National Natural Science Foundation of China (Nos.
91330205, 91630310, 11421101, 11871111).

\begin{appendix}
\renewcommand{\thesection}{\Alph{section}}

\section{Proof of Lemma \ref{lem3}}\label{appendix-A}
\renewcommand{\thesection}{\Alph{section}}

This appendix proves the properties (\romannumeral1)-(\romannumeral7) presented in Lemma \ref{lem3} orderly.
For the sake of simplicity, denote
$s_{\min}^{\pm}=s_{\min}(\vec{U}^{\pm})$ and $s_{\max}^{\pm}=s_{\max}(\vec{U}^{\pm})$.  

(\romannumeral1) Since  $A^{\pm}$ defined in \eqref{eq67} can be rewritten as
\begin{equation*}
A=sE-m=(s-u)E-up=\frac{p}{c_s^2(1-u^2)}\big((\Gamma-c_s^2+c_s^2u^2)s-\Gamma u\big),
\end{equation*}
where the superscript $\pm$ has been omitted,  for left  and right states state $\vec{U}^{\mp}$
one has
\begin{align*}
A^{-}&=\frac{p^{-}}{(c_s^{-})^2\big(1-(u^{-})^2\big)}\bigg(\big(\Gamma-(c_s^{-})^2-(c_s^{-}u^{-})^2\big)s^{-}-\Gamma u^{-}\bigg)\\
&\le\frac{p^{-}}{(c_s^{-})^2\big(1-(u^{-})^2\big)}\bigg(\big(\Gamma-(c_s^{-})^2-(c_s^{-}u^{-})^2\big)s_{\min}^{-}-\Gamma u^{-}\bigg)\\
&=\frac{p^{-}}{(c_s^{-})^2\big(1-(u^{-})^2\big)}\bigg(\big(\Gamma-(c_s^{-})^2-(c_s^{-}u^{-})^2\big)
\frac{u^{-}-c_s^{-}}{1-c_s^{-}u^{-}}-\Gamma u^{-}\bigg)\\
&=\frac{p^{-}}{c_s^{-}\big(1-c_s^{-}u^{-}\big)}\big(-\Gamma+(c_s^{-})^2-c_s^{-}u^{-}\big)
<\frac{p^{-}}{c_s^{-}\big(1-c_s^{-}u^{-}\big)}\big(-\Gamma+\Gamma-1-c_s^{-}u^{-}\big)<0,
\end{align*}
and
\begin{align*}
A^{+}&=\frac{p^{+}}{(c_s^{+})^2\big(1-(u^{+})^2\big)}\bigg(\big(\Gamma-(c_s^{+})^2-(c_s^{+}u^{+})^2\big)s^{+}-\Gamma u^{+}\bigg)\\
&\ge\frac{p^{+}}{(c_s^{+})^2\big(1-(u^{+})^2\big)}\bigg(\big(\Gamma-(c_s^{+})^2-(c_s^{+}u^{+})^2\big)s_{\max}^{+}-\Gamma u^{+}\bigg)\\
&=\frac{p^{+}}{(c_s^{+})^2\big(1-(u^{+})^2\big)}\bigg(\big(\Gamma-(c_s^{+})^2-(c_s^{+}u^{+})^2\big)
\frac{u^{+}+c_s^{+}}{1+c_s^{+}u^{+}}-\Gamma u^{+}\bigg)\\
&=\frac{p^{+}}{c_s^{+}\big(1+c_s^{+}u^{+}\big)}\big(\Gamma-(c_s^{+})^2-c_s^{+}u^{+}\big)
>\frac{p^{+}}{c_s^{+}\big(1-c_s^{+}u^{+}\big)}\big(\Gamma-\Gamma+1-c_s^{-}u^{+}\big)>0.
\end{align*}

(\romannumeral2) Due to the definition of $A^{\pm}$ and $B^{\pm}$  in \eqref{eq67},
  it can have 
\begin{align*}
A^{-}-B^{-}&=s^{-}E^{-}-m^{-}-m^{-}(s^{-}-u^{-})+p^{-}\\
&=\frac{p^{-}}{(c_s^{-})^2(1+u^{-})}\bigg(s^{-}\big(\Gamma-(c_s^{-})^2(1+u^{-})\big)-\Gamma u^{-}
+(c_s^{-})^2(1+u^{-})\bigg)\\
&\le\frac{p^{-}}{(c_s^{-})^2(1+u^{-})}\bigg(s_{\min}^{-}\big(\Gamma-(c_s^{-})^2(1+u^{-})\big)-\Gamma u^{-}
+(c_s^{-})^2(1+u^{-})\bigg)\\
&=\frac{p^{-}}{(c_s^{-})^2(1+u^{-})}\bigg(\frac{\big(\Gamma-(c_s^{-})^2(1+u^{-})\big)(u^{-}-c_s^{-})}
{1-c_s^{-}u^{-}}-\Gamma u^{-}+(c_s^{-})^2(1+u^{-})\bigg)\\
&=\frac{p^{-}(1-u^{-})}{c_s^{-}(1-c_s^{-}u^{-})}\bigg(-\Gamma+c_s^{-}+(c_s^{-})^2\bigg)
 <\frac{p^{-}(1-u^{-})}{c_s^{-}(1-c_s^{-}u^{-})}\bigg(-\Gamma+c_s^{-}+\Gamma-1\bigg)<0,
\end{align*}
and 
\begin{align*}
A^{+}+B^{+}&=s^{+}E^{+}-m^{+}+m^{+}(s^{+}-u^{+})-p^{-}\\
&=\frac{p^{+}}{(c_s^{+})^2(1-u^{+})}\bigg(s^{+}\big(\Gamma-(c_s^{+})^2(1-u^{+})\big)-\Gamma u^{+}
-(c_s^{+})^2(1-u^{+})\bigg)\\
&\ge\frac{p^{+}}{(c_s^{+})^2(1-u^{+})}\bigg(s_{\max}^{+}\big(\Gamma-(c_s^{+})^2(1-u^{+})\big)-\Gamma u^{+}
-(c_s^{+})^2(1-u^{+})\bigg)\\
&=\frac{p^{+}}{(c_s^{+})^2(1-u^{+})}\bigg(\frac{\big(\Gamma-(c_s^{+})^2(1-u^{+})\big)(u^{+}+c_s^{+})}
{1+c_s^{+}u^{+}}-\Gamma u^{+}-(c_s^{+})^2(1-u^{+})\bigg)\\
&=\frac{p^{+}(1+u^{+})}{c_s^{+}(1+c_s^{+}u^{+})}\bigg(\Gamma-c_s^{-}-(c_s^{-})^2\bigg) 
 >\frac{p^{+}(1+u^{+})}{c_s^{+}(1+c_s^{+}u^{+})}\bigg(\Gamma-c_s^{-}-\Gamma+1\bigg)>0.
\end{align*}

(\romannumeral3) For the states $\vec{U}^{\pm}$, it is easy to verify
\begin{align*}
s^{\pm}A^{\pm}-B^{\pm}&=s^{\pm}(s^{\pm}E^{\pm}-m^{\pm})-m^{\pm}(s^{\pm}-u^{\pm})+p^{\pm}\\
&=(s^{\pm})^2E^{\pm}-2s^{\pm}u^{\pm}(E^{\pm}+p^{\pm})+(u^{\pm})^2(E^{\pm}+p^{\pm})+p^{\pm}\\
&=(s^{\pm}-u^{\pm})^2E^{\pm}+\big(1+(u^{\pm})^2-2s^{\pm}u^{\pm}\big)p^{\pm}\\
&=(s^{\pm}-u^{\pm})^2E^{\pm}+\big(1-(s^{\pm})^2+(s^{\pm}-u^{\pm})^2\big)p^{\pm}>0.
\end{align*}

(\romannumeral4) Because $A^{-}<0$,  one has
\begin{align*}
s^{+}A^{-}-B^{-}&\le s^{-}_{\max}A^{-}-B^{-}
=\frac{u^{-}+c_s^{-}}{1+c_s^{-}u^{-}}\big[s^{-}E^{-}-u^{-}(E^{-}+p^{-})\big]-u^{-}(s^{-}-u^{-})(E^{-}+p^{-})+p^{-}\\
&=\frac{p^{-}}{1+c_s^{-}u^{-}}\bigg[s^{-}(u^{-}+c_s^{-})\bigg(\frac{\Gamma}{(c_s^{-})^2\big(1-(u^{-})^2\big)}-1\bigg)
-u^{-}(u^{-}+c_s^{-})\frac{\Gamma}{(c_s^{-})^2\big(1-(u^{-})^2\big)}\\
&~~~-u^{-}(s^{-}-u^{-})\frac{\Gamma(1+c_s^{-}u^{-})}{(c_s^{-})^2\big(1-(u^{-})^2\big)}+1+c_s^{-}u^{-}\bigg]\\
&=\frac{p^{-}}{c_s^{-}(1+c_s^{-}u^{-})}\bigg(s^{-}\big(\Gamma-c_s^{-}u^{-}-(c_s^{-})^2\big)-\Gamma u^{-}+c_s^{-}(1+c_s^{-}u^{-})\bigg)\\
&\le\frac{p^{-}}{c_s^{-}(1+c_s^{-}u^{-})}\bigg(s_{\min}^{-}\big(\Gamma-c_s^{-}u^{-}-(c_s^{-})^2\big)
-\Gamma u^{-}+c_s^{-}(1+c_s^{-}u^{-})\bigg)\\
&=\frac{p^{-}}{c_s^{-}(1+c_s^{-}u^{-})}\bigg(\frac{(u^{-}-c_s^{-})\big(\Gamma-c_s^{-}u^{-}-(c_s^{-})^2\big)}{1-c_s^{-}u^{-}}
-\Gamma u^{-}+c_s^{-}(1+c_s^{-}u^{-})\bigg)\\
&=\frac{p^{-}\big(1-(u^{-})^2\big)}{1-(c_s^{-}u^{-})^2}\bigg(-\Gamma+1+(c_s^{-})^2\bigg)<0.
\end{align*}
Similarly,  the fact $A^{+}>0$  gives
\begin{align*}
s^{-}A^{+}-B^{+}&\le s^{+}_{\min}A^{+}-B^{+} 
 =\frac{u^{+}-c_s^{+}}{1-c_s^{+}u^{+}}\big[s^{+}E^{+}-u^{+}(E^{+}+p^{+})\big]-u^{+}(s^{+}-u^{+})(E^{+}+p^{+})+p^{+}\\
&=\frac{p^{+}}{1-c_s^{+}u^{+}}\bigg[s^{+}(u^{+}-c_s^{+})\bigg(\frac{\Gamma}{(c_s^{+})^2\big(1-(u^{+})^2\big)}-1\bigg)
-u^{+}(u^{+}-c_s^{+})\frac{\Gamma}{(c_s^{+})^2\big(1-(u^{+})^2\big)}\\
&~~~-u^{+}(s^{+}-u^{+})\frac{\Gamma(1-c_s^{+}u^{+})}{(c_s^{+})^2\big(1-(u^{+})^2\big)}+1-c_s^{+}u^{+}\bigg]\\
&=\frac{p^{+}}{c_s^{+}(1-c_s^{+}u^{+})}\bigg(-s^{+}\big(\Gamma+c_s^{+}u^{+}-(c_s^{+})^2\big)
+\Gamma u^{+}+c_s^{+}(1-c_s^{+}u^{+})\bigg)\\
&\le\frac{p^{+}}{c_s^{+}(1-c_s^{+}u^{+})}\bigg(-s_{\max}^{+}\big(\Gamma+c_s^{+}u^{+}-(c_s^{+})^2\big)
+\Gamma u^{+}+c_s^{+}(1-c_s^{+}u^{+})\bigg)\\
&=\frac{p^{+}}{c_s^{+}(1-c_s^{+}u^{+})}\bigg(-\frac{(u^{+}+c_s^{+})\big(\Gamma+c_s^{+}u^{+}-(c_s^{+})^2\big)}{1+c_s^{+}u^{+}}
+\Gamma u^{+}+c_s^{+}(1-c_s^{+}u^{+})\bigg)\\
&=\frac{p^{+}\big(1-(u^{+})^2\big)}{1-(c_s^{+}u^{+})^2}\bigg(-\Gamma+1+(c_s^{+})^2\bigg)<0.
\end{align*}

(\romannumeral5) Because
\begin{align*}
s^{-}-u^{-}&\le s_{\min}^{-}-u^{-}=\frac{u^{-}-c_s^{-}}{1-c_s^{-}u^{-}}-u^{-}=-\frac{c_s^{-}\big(1-(u^{-})^2\big)}{1-c_s^{-}u^{-}}<0,\\
s^{+}-u^{-}&\ge s_{\max}^{-}-u^{-}=\frac{u^{-}+c_s^{-}}{1+c_s^{-}u^{-}}-u^{-}=\frac{c_s^{-}
\big(1-(u^{-})^2\big)}{1+c_s^{-}u^{-}}>0,
\end{align*}
for the left state $\vec{U}^{-}$ 
and
\begin{align*}
s^{-}-u^{+}&\le s_{\min}^{+}-u^{+}=\frac{u^{+}-c_s^{+}}{1-c_s^{+}u^{+}}-u^{+}=-\frac{c_s^{+}\big(1-(u^{+})^2\big)}{1-c_s^{+}u^{+}}<0,\\
s^{+}-u^{+}&\ge s_{\max}^{+}-u^{+}=\frac{u^{+}+c_s^{+}}{1+c_s^{+}u^{+}}-u^{+}=\frac{c_s^{+}\big(1-(u^{+})^2\big)}{1+c_s^{+}u^{+}}>0.
\end{align*}
for the right state $\vec{U}^{+}$,
it holds
$$s^{-}<u^{\pm}<s^{+}.$$

The remaindering is to prove show the inequality  
\begin{equation}\label{eq:A1}
s^{-}<s^{\ast}<s^{+}.
\end{equation}
 Because (\ref{eq16}) gives
\begin{equation*}
(1-s^{+}s^{\ast})(s^{\ast}A^{-}-B^{-})
=(1-s^{-}s^{\ast})(s^{\ast}A^{+}-B^{+}),
\end{equation*}
and the wave speed is less than the speed of light $c=1$, 
 two terms
$s^{\ast}A^{-}-B^{-}$ and $s^{\ast}A^{+}-B^{+}$ should have the same sign. 
It   means that
\begin{equation*}
(s^{\ast}A^{-}-B^{-})(s^{\ast}A^{+}-B^{+})\geq 0,
\end{equation*}
which gives
\begin{equation}\label{eq:A5}
\frac{B^{+}}{A^{+}}\leq s^{\ast}\leq \frac{B^{-}}{A^{-}},
\ \ \ \mbox{or}\
\frac{B^{-}}{A^{-}}\leq s^{\ast}\leq \frac{B^{+}}{A^{+}}.
\end{equation}
On the other hand, the properties (\romannumeral1) and (\romannumeral3) give
\begin{equation*}
s^{-}<\frac{B^{-}}{A^{-}},~~s^{+}>\frac{B^{+}}{A^{+}},
\end{equation*}
and  the properties (\romannumeral1) and (\romannumeral4) lead to
\begin{equation*}
s^{-}<\frac{B^{+}}{A^{+}},~~s^{+}>\frac{B^{-}}{A^{-}},
\end{equation*}
so one has 
\begin{equation}\label{eq:A6}
s^{-}<\min\left(\frac{B^{-}}{A^{-}},\frac{B^{+}}{A^{+}}\right),~~~~
s^{+}>\max\left(\frac{B^{-}}{A^{-}},\frac{B^{+}}{A^{+}}\right).
\end{equation}
Combing \eqref{eq:A5} with \eqref{eq:A6} complete the proof of \eqref{eq:A1}.

(\romannumeral6)  The left hand side of the inequality (\romannumeral6) can be recast into
\begin{equation*}
4(A^{\pm})^2-(s^{\pm}A^{\pm}+B^{\pm})^2=-(s^{\pm}A^{\pm}+B^{\pm}
+2A^{\pm})(s^{\pm}A^{\pm}+B^{\pm}-2A^{\pm})=-f_1^{\pm}\cdot f_2^{\pm},
\end{equation*}
with
\begin{equation*}
\begin{aligned}
f_1^{\pm}&=s^{\pm}A^{\pm}+B^{\pm}+2A^{\pm}=(s^{\pm})^2E^{\pm}+2s^{\pm}E^{\pm}-(2+u^{\pm})m^{\pm}-p^{\pm},\\
f_2^{\pm}&=s^{\pm}A^{\pm}+B^{\pm}-2A^{\pm}=(s^{\pm})^2E^{\pm}-2s^{\pm}E^{\pm}+(2-u^{\pm})m^{\pm}-p^{\pm}.
\end{aligned}
\end{equation*}
Thus, one has to prove  $f_1^{\pm} f_2^{\pm}<0$ in order to draw the conclusion (\romannumeral6).

For the left state $\vec U^-$, because $s^{-}\in(-1,s_{\min}^{-}]$ with $s_{\min}^{-}=\frac{u^{-}-c_s^{-}}{1-c_s^{-}u^{-}}<1$, one has
\begin{align*}
f_1^{-}&=(s^{-})^2E^{-}+2s^{-}E^{-}-(2+u^{-})m^{-}-p^{-} \\
&\le (s_{\min}^{-})^2E^{-}+2s_{\min}^{-}E^{-}-(2+u^{-})m^{-}-p^{-} \\
&=(s_{\min}^{-}+u^{-}+2)(s_{\min}^{-}-u^{-})E^{-}-(1+u^{-})^2p^{-} \\
&=-\frac{c_s^{-}(1+u^{-})\big(1-(u^{-})^2\big)\big(2-c_s^{-}(1+u^{-})\big)}
{(1-c_s^{-}u^{-})^2}E^{-}-(1+u^{-})^2p^{-}<0,
\end{align*}
and
\begin{align*}
f_2^{-}&=(s^{-})^2E^{-}-2s^{-}E_{-}+(2-u^{-})m^{-}-p^{-}\\\
&\ge(s_{\min}^{-})^2E^{-}-2s_{\min}^{-}E^{-}+(2-u^{-})m^{-}-p^{-}\\
&=(s_{\min}^{-}+u^{-}-2)(s_{\min}^{-}-u^{-})E^{-}-(1-u^{-})^2p^{-}\\
&=\frac{c_s^{-}(1-u^{-})^2(1+u^{-})\big(2+c_s^{-}(1-u^{-})\big)}{(1-c_s^{-}u^{-})^2}E^{-}-(1-u^{-})^2p^{-}\\
&=C_1\left(c_s^{-}(1+u^{-})\big(2+c_s^{-}(1-u^{-})\big)
\left(\frac{\Gamma}{(c_s^{-})^2\big(1-(u^{-})^2\big)}-1\right)-(1-c_s^{-}u^{-})^2\right)\\
&=C_1\left(\frac{\big(2+c_s^{-}(1-u^{-})\big)\Gamma}
{c_s^{-}(1-u^{-})}-2c_s^{-}(1+u^{-})-(c_s^{-})^2\big(1-(u^{-})^2\big)-(1-c_s^{-}u^{-})^2\right)\\
&=C_1\left(\frac{2\Gamma}{c_s^{-}(1-u^{-})}+\Gamma-(1+c_s^{-})^2\right)\notag\\
&=C_1\left(\frac{2\big(\Gamma-(c_s^{-})^2+(c_s^{-})^2u^{-}\big)}{c_s^{-}(1-u^{-})}
+\Gamma-1-(c_s^{-})^2\right)>0,
\end{align*}
where
$$C_1=\frac{(1-u^{-})^2p^{-}}{(1-c_s^{-}u^{-})^2}>0.$$

Similarly, for the right state $\vec U^+$, the fact that $s^{+}\in [s_{\max}^{+},1)$ with
$s_{\max}^{+}=\frac{u^{+}+c_s^{+}}{1+c_s^{+}u^{+}}>-1$ means
\begin{align*}
f_1^{+}&=(s^{+})^2E^{+}+2s^{+}E^{+}-(2+u^{+})m^{+}-p^{+} \\
&\ge(s_{\max}^{+})^2E^{+}+2s_{\max}^{+}E^{+}-(2+u^{+})m^{+}-p^{+}  \\
&=(s_{\max}^{+}+u^{+}+2)(s_{\max}^{+}-u^{+})E^{+}-(1+u^{+})^2p^{+} \\
&=\frac{c_s^{+}(1+u^{+})^2(1-u^{+})\big(2+c_s^{+}(1+u^{+})\big)}{(1+c_s^{+}u^{+})^2}E^{+}-(1+u^{+})^2p^{+} \\
&=C_2\left(c_s^{+}(1-u^{+})\big(2+c_s^{+}(1+u^{+})\big)\left(\frac{\Gamma}
{(c_s^{+})^2\big(1-(u^{+})^2\big)}-1\right)-(1+c_s^{+}u^{+})^2\right) \\
&=C_2\left(\frac{\big(2+c_s^{+}(1+u^{+})\big)\Gamma}
{c_s^{+}(1+u^{+})}-2c_s^{+}(1-u^{+})-(c_s^{+})^2\big(1-(u^{+})^2\big)-(1+c_s^{+}u^{+})^2\right) \\
&=C_2\left(\frac{2\Gamma}{c_s^{+}(1+u^{+})}+\Gamma-(1+c_s^{+})^2\right) \\
&=C_2\left(\frac{2(\Gamma-(c_s^{+})^2-(c_s^{+})^2u^{+})}{c_s^{+}(1+u^{+})}+\Gamma-1-(c_s^{+})^2\right)>0,
\end{align*}
and
\begin{align*}
f_2^{+}&=(s^{+})^2E^{+}-2s^{+}E^{+}+(2-u^{+})m^{+}-p^{+} \\
&\le(s_{\max}^{+})^2E^{+}-2s_{\max}^{+}E^{+}+(2-u^{+})m^{+}-p^{+}  \\
&=(s_{\max}^{+}+u^{+}-2)(s_{\max}^{+}-u^{+})E^{+}-(1-u^{+})^2p^{+}  \\
&=-\frac{c_s^{+}(1-u^{+})(1-(u^{+})^2)\big(2-c_s^{+}(1-u^{+})\big)}{(1+c_s^{+}u^{+})^2}E^{+}-(1-u^{+})^2p^{+}<0,
\end{align*}
where
$$C_2=\frac{(1+u^{+})^2p^{+}}{(1+c_s^{+}u^{+})^2}>0.$$
In a word, the conclusion (\romannumeral6) holds.

(\romannumeral7) The left hand side of the inequality (\romannumeral7) corresponds to the value  of  the following
quadratic function   $f^{\pm}(s)$  at $s=s^\pm$
\begin{align*}
f^{\pm}(s)&:=(A^{\pm})^2-(B^{\pm})^2-(D^{\pm})^2(s-u^{\pm})^2\notag\\
&=(s-u^{\pm})^2\big((E^{\pm})^2-(D^{\pm})^2-(m^{\pm})^2-(p^{\pm})^2\big)+(p^{\pm})^2(s^2-1).
\end{align*}
Because
\begin{align*}
&(E^{\pm})^2-(D^{\pm})^2-(m^{\pm})^2=\frac{(p^{\pm})^2}{1-(u^{\pm})^2}\left(1-(u^{\pm})^2+\frac{2\Gamma}
{(c_s^{\pm})^2(\Gamma-1)}-\frac{\Gamma^2}{(\Gamma-1)^2}\right)>0,\\
&(E^{\pm})^2-(D^{\pm})^2-(m^{\pm})^2-(p^{\pm})^2=\frac{(p^{\pm})^2}{1-(u^{\pm})^2}\left(\frac{2\Gamma}
{(c_s^{\pm})^2(\Gamma-1)}-\frac{\Gamma^2}{(\Gamma-1)^2}\right)>0,
\end{align*}
 the parabolas $f^{\pm}(s)$ are   convex and symmetric with the lines
$$s=s_0^{\pm}=\frac{u^{\pm}\big[(E^{\pm})^2-(D^{\pm})^2-(m^{\pm})^2-(p^{\pm})^2\big]}{(E^{\pm})^2-(D^{\pm})^2-(m^{\pm})^2}.$$
For the left state $\vec U^-$, one has
\begin{align*}
s_0^{-}-s_{\min}^{-}&=\frac{u^{-}\big[(E^{-})^2-(D^{-})^2-(m^{-})^2-(p^{-})^2\big]}{(E^{-})^2-(D^{-})^2-(m^{-})^2}
-\frac{u^{-}-c_s^{-}}{1-c_s^{-}u^{-}}\\
&=C_3\left(-u^{-}+c_s^{-}(u^{-})^2+\frac{c_s^{-}}{(p^{-})^2}\big[1-(u^{-})^2\big]\big[(E^{-})^2-(D^{-})^2-(m^{-})^2\big]\right)\\
&=C_3\left(-u^{-}+c_s^{-}(u^{-})^2+c_s^{-}\left(1-(u^{-})^2+\frac{2\Gamma}{(c_s^{-})^2(\Gamma-1)}
-\frac{\Gamma^2}{(\Gamma-1)^2}\right)\right)\\
&=C_3\left(-u^{-}+\frac{2}{c_s^{-}}+\frac{2}{c_s^{-}(\Gamma-1)}-\frac{2c_s^{-}}{\Gamma-1}-\frac{c_s^{-}}{(\Gamma-1)^2}\right)\\
&=\frac{C_3}{c_s^{-}(\Gamma-1)^2}\bigg(-(\Gamma-1)^2c_s^{-}u^{-}+2(\Gamma-1)^2+2(\Gamma-1)-(2\Gamma-1)(c_s^{-})^2\bigg)\\
&>\frac{C_3}{c_s^{-}(\Gamma-1)^2}\bigg(-(\Gamma-1)^2+2(\Gamma-1)^2+2(\Gamma-1)-(2\Gamma-1)(\Gamma-1)\bigg)\\
&=\frac{C_3}{c_s^{-}(\Gamma-1)}(2-\Gamma)>0,
\end{align*}
with
$$C_3=\frac{(p^{-})^2}{(1-c_s^{-}u^{-})\big[(E^{-})^2-(D^{-})^2-(m^{-})^2\big]}>0.$$
Hence
 $f^{-}(s)$ is monotonically
decreasing with $s\in(-1,s_{\min}^{-}]$. It means that $f^{-}(s)\ge f^{-}(s_{\min}^{-})$ for any
$s\in(-1,s_{\min}^{-}]$ and thus
$f^{-}(s^{-})\ge f^{-}(s_{\min}^{-})$ since $s^-=\min(s_{\min}^{-},s_{\min}^{+})$.
To get the inequality (\romannumeral7) for the left state $\vec U^-$,
the remaining is to prove $f^{-}(s_{\min}^{-})>0$. In fact, it is true because
\begin{align*}
 f^-(s_{\min}^{-})&=\frac{(c_s^{-})^2\big[1-(u^{-})^2\big]^2}
{(1-c_s^{-}u^{-})^2}\big[(E^{-})^2-(D^{-})^2-(m^{-})^2-(p^{-})^2\big]\\
&~~+(p^{-})^2\left(\frac{(u^{-}-c_s^{-})^2}{(1-c_s^{-}u^{-})^2}-1\right)\\
&=\frac{(p^{-})^2\big[1-(u^{-})^2\big]}{(1-c_s^{-}u^{-})^2}\left(\frac{2\Gamma}{\Gamma-1}
-\frac{\Gamma^2(c_s^{-})^2}{(\Gamma-1)^2}\right)+(p^{-})^2\frac{\big[1-(c_s^{-})^2\big]\big[(u^{-})^2-1\big]}{(1-c_s^{-}u^{-})^2}\\
&=\frac{(p^{-})^2\big[1-(u^{-})^2\big]}{(1-c_s^{-}u^{-})^2}\left(\frac{2\Gamma}{\Gamma-1}-\frac{\Gamma^2(c_s^{-})^2}{(\Gamma-1)^2}+(c_s^{-})^2-1\right)\\
&=\frac{(p^{-})^2\big[1-(u^{-})^2\big]}{(1-c_s^{-}u^{-})^2}\left(\frac{2}{\Gamma-1}+1+\left(1-\frac{\Gamma^2}{(\Gamma-1)^2}\right)(c_s^{-})^2\right)\\
&>\frac{(p^{-})^2\big[1-(u^{-})^2\big]}{(1-c_s^{-}u^{-})^2}\left(\frac{2}{\Gamma-1}+1+(\Gamma-1)\left(1-\frac{\Gamma^2}{(\Gamma-1)^2}\right)\right)\\
&=\frac{(p^{-})^2\big[1-(u^{-})^2\big]}{(1-c_s^{-}u^{-})^2}\left(\frac{2}{\Gamma-1}+1+(\Gamma-1)-\frac{\Gamma^2}{\Gamma-1}\right)>0.
\end{align*}

Similarly, for the right state $\vec U^+$,
$f^{+}(s)$ is monotonically increasing with $s\in[s_{\max}^{+},1)$ and then $f^{+}(s)\ge f^{+}(s_{\max}^{+})$
for any $s\in[s_{\max}^{+},1)$ since
\begin{align*}
s_0^{+}-s_{\max}^{+}&=\frac{u^{+}\big[(E^{+})^2-(D^{+})^2-(m^{+})^2-(p^{+})^2\big]}{(E^{+})^2-(D^{+})^2-(m^{+})^2}-\frac{u^{+}+c_s^{+}}{1+c_s^{+}u^{+}}\\
&=C_4\left(-u^{+}-c_s^{+}(u^{+})^2-\frac{c_s^{+}}{p_{+}^2}\big[1-(u^{+})^2\big]\big[(E^{+})^2-(D^{+})^2-(m^{+})^2\big]\right)\\
&=C_4\left(-u^{+}(1+c_s^{+}u^{+})-c_s^{+}\left(1-(u^{+})^2+\frac{2\Gamma}{(c_s^{+})^2(\Gamma-1)}
-\frac{\Gamma^2}{(\Gamma-1)^2}\right)\right)\\
&=C_4\left(-u^{+}-\frac{2}{c_s^{+}}-\frac{2}{c_s^{+}(\Gamma-1)}+\frac{2c_s^{+}}{\Gamma-1}+\frac{c_s^{+}}{(\Gamma-1)^2}\right)\\
&=\frac{C_4}{c_s^{+}(\Gamma-1)^2}\bigg(-(\Gamma-1)^2c_s^{+}u^{+}-2(\Gamma-1)^2-2(\Gamma-1)+(2\Gamma-1)(c_s^{+})^2\bigg)\\
&<\frac{C_4}{c_s^{+}(\Gamma-1)^2}\bigg((\Gamma-1)^2-2(\Gamma-1)^2-2(\Gamma-1)+(2\Gamma-1)(\Gamma-1)\bigg)<0,
\end{align*}
with
$$C_4=\frac{(p^{+})^2}{(1+c_s^{+}u^{+})\big[(E^{+})^2-(D^{+})^2-(m^{+})^2\big]}>0.$$
To get the inequality (\romannumeral7) for the left state $\vec U^+$,
the remaining is to prove   $f^{+}(s_{\max}^{+})>0$. It is true since
\begin{align*}
 f^{+}(s_{\max}^{+})&=\frac{(c_s^{+})^2\big[1-(u^{+})^2\big]^2}{(1+c_s^{+}u^{+})^2}
\big[(E^{+})^2-(D^{+})^2-(m^{+})^2-(p^{+})^2\big]\\
&~~+(p^{+})^2\left(\frac{(u^{+}+c_s^{+})^2}{(1+c_s^{+}u^{+})^2}-1\right)\\
&=\frac{(p^{+})^2\big[1-(u^{+})^2\big]}{(1+c_s^{+}u^{+})^2}\left(\frac{2\Gamma}{\Gamma-1}
-\frac{\Gamma^2(c_s^{+})^2}{(\Gamma-1)^2}\right)+(p^{+})^2\frac{\big[1-(c_s^{+})^2\big]\big[(u^{+})^2-1\big]}{(1+c_s^{+}u^{+})^2}\\
&=\frac{(p^{+})^2\big[1-(u^{+})^2\big]}{(1+c_s^{+}u^{+})^2}\left(\frac{2\Gamma}{\Gamma-1}
-\frac{\Gamma^2(c_s^{+})^2}{(\Gamma-1)^2}+(c_s^{+})^2-1\right)\\
&=\frac{(p^{+})^2\big[1-(u^{+})^2\big]}{(1+c_s^{+}u^{+})^2}\left(\frac{2}{\Gamma-1}+1+\left(1-\frac{\Gamma^2}
{(\Gamma-1)^2}\right)(c_s^{+})^2\right)\\
&>\frac{(p^{+})^2\big[1-(u^{+})^2\big]}{(1+c_s^{+}u^{+})^2}\left(\frac{2}{\Gamma-1}+1+(\Gamma-1)
\left(1-\frac{\Gamma^2}{(\Gamma-1)^2}\right)\right)\\
&=\frac{(p^{+})^2\big[1-(u^{+})^2\big]}{(1+c_s^{+}u^{+})^2}\left(\frac{2}{\Gamma-1}+1+(\Gamma-1)-\frac{\Gamma^2}{\Gamma-1}\right)\\
&=\frac{(p^{+})^2\big[1-(u^{+})^2\big]}{(1+c_s^{+}u^{+})^2}\left(\frac{2}{\Gamma-1}-\frac{\Gamma}{\Gamma-1}\right)>0.
\end{align*}
It completes the proof.
\qed

\renewcommand{\thesection}{\Alph{section}}
\section{Proof of Lemma \ref{lem4}}\label{appendix-B}
\renewcommand{\thesection}{\Alph{section}}
This appendix discusses the proof of Lemma \ref{lem4}.
Here only the properties  (\romannumeral6) and (\romannumeral7) are proved in detail, 
since the proof of the properties (\romannumeral1)-(\romannumeral5)  
is similar to that of Lemma \ref{lem3}.  
For the sake of simplicity, denote
$
s_{\min}^{\pm}=s_{\min}(\vec{U}^{\pm})$, $s_{\max}^{\pm}=s_{\max}(\vec{U}^{\pm})$,
and define
\begin{equation*}
\begin{aligned}
&\alpha^{\pm}=\sqrt{(1-|\vec{u}^{\pm}|^2)\big[1-(u_n^{\pm})^2+(c_s^{\pm}u_n^{\pm})^2-(c_s^{\pm})^2|\vec{u}^{\pm}|^2\big]},\\
&\beta^{\pm}=\frac{2\Gamma(\Gamma-1)-\Gamma^2(c_s^{\pm})^2}{(\Gamma-1)^2}.
\end{aligned}
\end{equation*}
Due to \eqref{ws2d}, $s_{\min}^{\pm}$ and $s_{\max}^{\pm}$ can be  expressed as
\begin{equation*}
s_{\min}^{\pm}=\frac{(1-(c_s^{\pm})^2)u_n^{\pm}-c_s^{\pm}\alpha^{\pm}}{1-(c_s^{\pm})^2|\vec{u}^{\pm}|^2},~~
s_{\max}^{\pm}=\frac{(1-(c_s^{\pm})^2)u_n^{\pm}+c_s^{\pm}\alpha^{\pm}}{1-(c_s^{\pm})^2|\vec{u}^{\pm}|^2}.
\end{equation*}
Moreover, it can be verified that
\begin{equation*}
1-|\vec{u}^{\pm}|^2\le\alpha^{\pm}\le1-|u_n^{\pm}|^2,~~\beta^{\pm}>1-(c_s^{\pm})^2.
\end{equation*}

(\romannumeral6) The left hand side of (\romannumeral6) can be decomposed into the two parts
\begin{equation*}
4(A^{\pm})^2-(s^{\pm}A^{\pm}+B^{\pm})^2=-(s^{\pm}A^{\pm}+B^{\pm}
+2A^{\pm})(s^{\pm}A^{\pm}+B^{\pm}-2A^{\pm})=-f_1^{\pm}\cdot f_2^{\pm},
\end{equation*}
where
\begin{equation*}
\begin{aligned}
f_1^{\pm}&=s^{\pm}A^{\pm}+B^{\pm}+2A^{\pm}=(s^{\pm})^2E^{\pm}+2s^{\pm}E^{\pm}-(2+u_n^{\pm})m_n^{\pm}-p^{\pm},\\
f_2^{\pm}&=s^{\pm}A^{\pm}+B^{\pm}-2A^{\pm}=(s^{\pm})^2E^{\pm}-2s^{\pm}E^{\pm}+(2-u_n^{\pm})m_n^{\pm}-p^{\pm}.
\end{aligned}
\end{equation*}
Thus the remaining is  to prove that  $f_1^{\pm} f_2^{\pm}<0$.

For the left state $\vec U^-$, because $s^{-}\in(-1,s_{\min}^{-}]$ with
$s_{\min}^{-}=\frac{u_n^{-}(1-(c_s^{-})^2)-c_s^{-}\alpha^{-}}{1-(c_s^{-})^2|\vec{u}^{-}|^2}<u_n^{-}$,
one has
\begin{align*}
f_1^{-}&=(s^{-})^2E^{-}+2s^{-}E^{-}-(2+u_n^{-})m_n^{-}-p^{-}\\
&\le (s_{\min}^{-})^2E_{-}+2s_{\min}^{-}E^{-}-(2+u_n^{-})m_n^{-}-p^{-}\\
&=\big[(s_{\min}^{-})^2+2s_{\min}^{-}-u_n^{-}(2+u_n^{-})\big]E^{-}-(1+u_n^{-})^2p^{-}\\
&=(s_{\min}^{-}+u_n^{-}+2)(s_{\min}^{-}-u_n^{-})E^{-}-(1+u_n^{-})^2p^{-}<0,
\end{align*}
and
\begin{align*}
f_2^{-}&=(s^{-})^2E^{-}-2s^{-}E^{-}+(2-u_n^{-})m_n^{-}-p^{-}\\
&\ge(s_{\min}^{-})^2E^{-}-2s_{\min}^{-}E^{-}+(2-u_n^{-})m_n^{-}-p^{-}\\
&=\big[(1-s_{\min}^{-})^2-(1-u_n^{-})^2\big]E^{-}-(1-u_n^{-})^2p^{-}\\
&=\frac{1}{(1-u_n^{-})^2}\left(\left(\frac{1-s_{\min}^{-}}{1-u_n^{-}}\right)^2E^{-}-(E^{-}+p^{-})\right)\\
&=\frac{1}{(1-u_n^{-})^2}\left(\left(\frac{1-u_n^{-}+c_s^{-}\alpha^{-}
+(c_s^{-})^2(u_n^{-}-|\vec{u}^{-}|^2)}{(1-u_n^{-})(1-(c_s^{-})^2|\vec{u}^{-}|^2)}\right)^2E^{-}-(E^{-}+p^{-})\right)\\
&>\frac{1}{(1+|\vec{u}^{-}|)^2}\left(\left(\frac{1+|\vec{u}^{-}|+c_s^{-}(1-|\vec{u}^{-}|^2)
+(c_s^{-})^2|\vec{u}^{-}|(-1-|\vec{u}^{-}|)}{(1+|\vec{u}^{-}|)
(1-(c_s^{-})^2|\vec{u}^{-}|^2)}\right)^2E^{-}-(E^{-}+p^{-})\right)\\
&=\frac{1}{(1+|\vec{u}^{-}|)^2}\left(\left(\frac{1+c_s^{-}}{1+c_s^{-}|\vec{u}^{-}|}\right)^2E^{-}-(E^{-}+p^{-})\right)\\
&=\frac{p^{-}}{(1+|\vec{u}^{-}|)^2(1+c_s^{-}|\vec{u}^{-}|)^2}\left((1+c_s^{-})^2\frac{\Gamma-(c_s^{-})^2(1-|\vec{u}^{-}|^2)}
{(c_s^{-})^2(1-|\vec{u}^{-}|^2)}-\frac{\Gamma(1+c_s^{-}|\vec{u}^{-}|)^2}{(c_s^{-})^2(1-|\vec{u}^{-}|^2)}\right)\\
&=C_5\big[2\Gamma+\Gamma c_s^{-}(1+|\vec{u}^{-}|)-c_s^{-}(1+|\vec{u}^{-}|)(1+c_s^{-})^2\big]\\
&=C_5\big[2\Gamma+c_s^{-}(1+|\vec{u}^{-}|)(\Gamma-(1+c_s^{-})^2)\big]\\
&>C_5\big[2\Gamma-2(c_s^{-})^2(1+|\vec{u}^{-}|)\big]
 >2C_5\big[\Gamma-(\Gamma-1)(1+|\vec{u}^{-}|)\big]\\
&=2C_5\big[1-(\Gamma-1)|\vec{u}^{-}|\big]>2C_5\big(1-|\vec{u}^{-}|\big)>0,
\end{align*}
where
$$C_5=\frac{p^{-}}{c_s^{-}(1+|\vec{u}^{-}|)^3(1+c_s^{-}|\vec{u}^{-}|)^2}>0.$$

Similarly, for the right state $\vec U^+$, because $s^{+}\in[s_{\max}^{+},1)$ with
$s_{\max}^{+}=\frac{u_n^{+}\big(1-(c_s^{+})^2\big)+c_s^{+}\alpha^{+}}{1-(c_s^{+})^2|\vec{u}^{+}|^2}>u_n^{+}$,
one has
\begin{align*}
f_1^{+}&=(s^{+})^2E^{+}+2s^{+}E^{+}-(2+u_n^{+})m_n^{+}-p^{+}\\
&=\big[(s^{+})^2+2s^{+}-u_n^{+}(2+u_n^{+})\big]E^{+}-(1+u_n^{+})^2p^{+}\\
&\ge\big[(s_{\max}^{+})^2+2s_{\max}^{+}-u_n^{+}(2+u_n^{+})\big]E^{+}-(1+u_n^{+})^2p^{+}\\
&=\big[(1+s_{\max}^{+})^2-(1+u_n^{+})^2\big]E^{+}-(1+u_n^{+})^2p^{+}\\
&=\frac{1}{(1+u_n^{+})^2}\left(\left(\frac{1+s_{\max}^{+}}{1+u_n^{+}}\right)^2E^{+}-(E^{+}+p^{+})\right)\\
&=\frac{1}{(1+u_n^{+})^2}\left(\left(\frac{1+u_n^{+}+c_s^{+}\alpha^{+}-(c_s^{+})^2(u_n^{+}+|\vec{u}^{+}|^2)}
{(1+u_n^{+})(1-(c_s^{+})^2|\vec{u}^{+}|^2)}\right)^2E^{+}-(E^{+}+p^{+})\right)\\
&>\frac{1}{(1+|\vec{u}^{+}|)^2}\left(\left(\frac{1+|\vec{u}^{+}|
+c_s^{+}(1-|\vec{u}^{+}|^2)-(c_s^{+})^2|\vec{u}^{+}|(1+|\vec{u}^{+}|)}
{(1+|\vec{u}^{+}|)(1-(c_s^{+})^2|\vec{u}^{+}|^2)}\right)^2E^{+}-(E^{+}+p^{+})\right)\\
&=\frac{1}{(1+|\vec{u}^{+}|)^2}\left(\left(\frac{1+c_s^{+}}{1+c_s^{+}|\vec{u}^{+}|}\right)^2E^{+}-(E^{+}+p^{+})\right)\\
&=\frac{p^{+}}{(1+|\vec{u}^{+}|)^2(1+c_s^{+}|\vec{u}^{+}|)^2}\left((1+c_s^{+})^2\frac{\Gamma-(c_s^{+})^2(1-|\vec{u}^{+}|^2)}
{(c_s^{+})^2(1-|\vec{u}^{+}|^2)}-\frac{\Gamma(1+c_s^{+}|\vec{u}^{+}|)^2}{(c_s^{+})^2(1-|\vec{u}^{+}|^2)}\right)\\
&=C_6\big[2\Gamma+\Gamma c_s^{+}(1+|\vec{u}^{+}|)-c_s^{+}(1+|\vec{u}^{+}|)(1+c_s^{+})^2\big]\\
&=C_6\big[2\Gamma+c_s^{+}(1+|\vec{u}^{+}|)(\Gamma-(1+c_s^{+})^2)\big]\\
&>C_6\big[2\Gamma-2(c_s^{+})^2(1+|\vec{u}^{+}|)\big]
 >2C_6\big[\Gamma-(\Gamma-1)(1+|\vec{u}^{+}|)\big]\\
&=2C_6\big[1-(\Gamma-1)|\vec{u}^{+}|\big]>2C_6(1-|\vec{u}^{+}|)>0,
\end{align*}
and
\begin{align*}
f_2^{+}&=(s^{+})^2E^{+}-2s^{+}E^{+}+(2-u_n^{+})m_n^{+}-p^{+}\\
&=\big[(s^{+})^2-2s^{+}+u_n^{+}(2-u_n^{+})\big]E^{+}-(1-u_n^{+})^2p^{+}\\
&\le\big[(s_{\max}^{+})^2-2s_{\max}^{+}+u_n^{+}(2-u_n^{+})\big]E^{+}-(1-u_n^{+})^2p^{+}\\
&=(s_{\max}^{+}+u_n^{+}-2)(s_{\max}^{+}-u_n^{+})E^{+}-(1-u_n^{+})^2p^{+}<0,
\end{align*}
where
$$C_6=\frac{p^{+}}{c_s^{+}(1+|\vec{u}^{+}|)^3(1+c_s^{+}|\vec{u}^{+}|)^2}>0.$$
Therefore, the inequality (\romannumeral6) is proved.

({\romannumeral7})  
Let us consider the quadratic functions in terms of $s$
\begin{align*}
f^{\pm}(s)&=(A^{\pm})^2-(B^{\pm})^2-(D^{\pm})^2(s-u_n^{\pm})^2-(m_{\tau}^{\pm})^2(s-u_n^{\pm})^2\notag\\
&=(s-u_n^{\pm})^2\big[(E^{\pm})^2-(D^{\pm})^2-|\vec{m}^{\pm}|^2-(p^{\pm})^2\big]+(s^2-1)(p^{\pm})^2,
\end{align*}
which
are symmetric with the lines
$$s=s_0^{\pm}=\frac{u_n^{\pm}\big[(E^{\pm})^2-(D^{\pm})^2-|\vec{m}^{\pm}|^2
-(p^{\pm})^2\big]}{(E^{\pm})^2-(D^{\pm})^2-|\vec{m}^{\pm}|^2}.$$
Moreover, $f^{\pm}(s)$ are convex because
\begin{align*}
(E^{\pm})^2-(D^{\pm})^2-|\vec{m}^{\pm}|^2
&=\frac{(p^{\pm})^2}{1-|\vec{u}^{\pm}|^2}\left(1-|\vec{u}^{\pm}|^2
+\frac{2\Gamma}{(c_s^{\pm})^2(\Gamma-1)}-\frac{\Gamma^2}{(\Gamma-1)^2}\right) \\
&>\frac{(p^{\pm})^2}{1-|\vec{u}^{\pm}|^2}\left(1-|\vec{u}^{\pm}|^2
+\frac{2\Gamma}{(\Gamma-1)^2}-\frac{\Gamma^2}{(\Gamma-1)^2}\right)>0, 
\end{align*}
and
\begin{align*}
(E^{\pm})^2-(D^{\pm})^2-|\vec{m}^{\pm}|^2-(p^{\pm})^2
&=\frac{(p^{\pm})^2}{1-|\vec{u}^{\pm}|^2}\left(\frac{2\Gamma}{(c_s^{\pm})^2(\Gamma-1)}
-\frac{\Gamma^2}{(\Gamma-1)^2}\right) \\
&>\frac{(p^{\pm})^2}{1-|\vec{u}^{\pm}|^2}\left(\frac{2\Gamma}
{(\Gamma-1)^2}-\frac{\Gamma^2}{(\Gamma-1)^2}\right)>0.
\end{align*}
For the left state $\vec U^-$, one has
\begin{align*}
s_0^{-}-s_{\min}^{-}&=\frac{u_n^{-}\big[(E^{-})^2-(D^{-})^2-|\vec{m}^{-}|^2-(p^{-})^2\big]}{(E^{-})^2-(D^{-})^2-|\vec{m}^{-}|^2}
-\frac{u_n^{-}(1-(c_s^{-})^2)-c_s^{-}\alpha^{-}}{1-(c_s^{-})^2|\vec{u}^{-}|^2}\\
&=C_7\bigg(-u_n^{-}(1-(c_s^{-})^2|\vec{u}^{-}|^2)\bigg)+(p^{-})^2\big[u_n^{-}(c_s^{-})^2(1-|\vec{u}^{-}|^2)+c_s^{-}\alpha^{-}\big]\\
&=C_7\left(u_n^{-}\big[\beta^{-}-1+(c_s^{-})^2\big]+\frac{\alpha^{-}}{1-|\vec{u}^{-}|^2}\left(c_s^{-}(1-|\vec{u}^{-}|^2)
+\frac{\beta^{-}}{c_s^{-}}\right)\right)\\
&>C_7\left(-\big[\beta^{-}-1+(c_s^{-})^2\big]+c_s^{-}(1-|\vec{u}^{-}|^2)+\frac{\beta^{-}}{c_s^{-}}\right)\\
&>C_7\left(1-(c_s^{-})^2-\beta^{-}+c_s^{-}(1-|\vec{u}^{-}|^2)+\frac{\beta^{-}}{c_s^{-}}\right)\\
&=C_7\left(1-(c_s^{-})^2+c_s^{-}(1-|\vec{u}^{-}|^2)+\left(\frac{1}{c_s^{-}}-1\right)\beta^{-}\right)>0,
\end{align*}
which implies that $f^{-}(s)$ is monotonically
decreasing with $s\in(-1,s_{\min}^{-}]$, where
$$C_7=\frac{(p^{-})^2}{(1-(c_s^{-})^2|\vec{u}^{-}|^2)\big[(E^{-})^2-(D^{-})^2-|\vec{m}^{-}|^2\big]}>0.$$
Thus $f^{-}(s)\ge f^{-}(s_{\min}^{-})$ for $s\in(-1,s_{\min}^{-}]$, so that
$f^{-}(s^{-})\ge f^{-}(s_{\min}^{-})$.
The remaining is to prove $f^{-}(s_{\min}^{-})>0$. In fact, one has
\begin{align*}
f^{-}(s_{\min}^{-})&=\left(s_{\min}^{-}-u_n^{-}\right)^2\big[(E^{-})^2-(D^{-})^2-|\vec{m}^{-}|^2-(p^{-})^2\big]
+(p^{-})^2\big[(s_{\min}^{-})^2-1\big]\\
&=C_8\left(\frac{\beta^{-}\big[c_s^{-}u_n^{-}(1-|\vec{u}^{-}|^2)+\alpha^{-}\big]^2}{1-|\vec{u}^{-}|^2}
+\big[(1-(c_s^{-})^2)u_n^{-}-c_s^{-}\alpha^{-}\big]^2-\big[1-(c_s^{-})^2|\vec{u}^{-}|^2\big]^2\right)\\
&=C_8\bigg(\beta^{-}\left((c_s^{-}u_n^{-})^2\big(1-|\vec{u}^{-}|^2\big)+2\alpha^{-}c_s^{-}u_n^{-}
+\frac{(\alpha^{-})^2}{1-|\vec{u}^{-}|^2}\right)+\big[(1-(c_s^{-})^2)u_n^{-}-c_s^{-}\alpha^{-}\big]^2\\
&~~~-\big[1-(c_s^{-})^2|\vec{u}^{-}|^2\big]^2\bigg)\\
&=C_8\bigg((c_s^{-}u_n^{-})^2\big[\beta^{-}-1+(c_s^{-})^2\big]-\beta^{-}(c_s^{-}u_n^{-})^2|\vec{u}^{-}|^2
+2\alpha^{-}c_s^{-}u_n^{-}\big[\beta^{-}-1+(c_s^{-})^2\big]\\
&~~~+(\alpha^{-})^2\left(\frac{\beta^{-}}{1-|\vec{u}^{-}|^2}+(c_s^{-})^2\right)+(u_n^{-})^2-(c_s^{-}u_n^{-})^2
-\big[1-(c_s^{-})^2|\vec{u}^{-}|^2\big]^2\bigg)\\
&=C_8\bigg(\big[\beta^{-}-1+(c_s^{-})^2\big]\big[(c_s^{-}u_n^{-})^2+2\alpha^{-}c_s^{-}u_n^{-}+(\alpha^{-})^2\big]
-\beta^{-}(c_s^{-}u_n^{-})^2|\vec{u}^{-}|^2\\
&~~~+(\alpha^{-})^2\left(1+\frac{\beta^{-}|\vec{u}^{-}|^2}{1-|\vec{u}^{-}|^2}\right)+(1-(c_s^{-})^2)(u_n^{-})^2
-\big[1-(c_s^{-})^2|\vec{u}^{-}|^2\big]^2\bigg)\\
&=C_8\left(\big[\beta^{-}-1+(c_s^{-})^2\big](\alpha^{-}+c_s^{-}u_n^{-})^2
+\left(1+\frac{\beta^{-}|\vec{u}^{-}|^2}{1-|\vec{u}^{-}|^2}\right)(\alpha^{-})^2\right.\\
&~~~+(u_n^{-})^2-(c_s^{-}u_n^{-})^2-1+2(c_s^{-})^2|\vec{u}^{-}|^2-(c_s^{-})^4|\vec{u}^{-}|^4
-\beta^{-}(c_s^{-}u_n^{-})^2|\vec{u}^{-}|^2\bigg)\\
&=C_8\bigg(\big[\beta^{-}-1+(c_s^{-})^2\big](\alpha^{-}+c_s^{-}u_n^{-})^2+(c_s^{-})^2|\vec{u}^{-}|^2-(c_s^{-})^4|\vec{u}^{-}|^4
-(c_s^{-}u_n^{-})^2|\vec{u}^{-}|^2\\
&~~~+(\beta^{-}-1)|\vec{u}^{-}|^2\big[1-(u_n^{-})^2-(c_s^{-})^2|\vec{u}^{-}|^2\big]\bigg)\\
&=C_8\bigg(\big[\beta^{-}-1+(c_s^{-})^2\big](\alpha^{-}+c_s^{-}u_n^{-})^2
+|\vec{u}^{-}|^2\big[\beta^{-}-1+(c_s^{-})^2\big]\big[1-(u_n^{-})^2-(c_s^{-})^2|\vec{u}^{-}|^2\big]\bigg)\\
&=C_8\big[\beta^{-}-1+(c_s^{-})^2\big]\bigg((c_s^{-}u_n^{-}+\alpha^{-})^2+|\vec{u}^{-}|^2\big[1-(u_n^{-})^2-(c_s^{-})^2|\vec{u}^{-}|^2\big]\bigg),
\end{align*}
where
$$C_8=\frac{(p^{-})^2}{\big[1-(c_s^{-})^2|\vec{u}^{-}|^2\big]^2}>0,$$
and
\begin{align*}
&~~~(c_s^{-}u_n^{-}+\alpha^{-})^2+|\vec{u}^{-}|^2\big(1-(u_n^{-})^2-(c_s^{-})^2|\vec{u}^{-}|^2)\\
&=(c_s^{-}u_n^{-})^2+(1-|\vec{u}^{-}|^2)\big[1-(u_n^{-})^2-(c_s^{-})^2|\vec{u}^{-}|^2+(c_s^{-}u_n^{-})^2\big]+2\alpha^{-}c_s^{-}u_n^{-}\\
&~~~+|\vec{u}^{-}|^2\big[1-(u_n^{-})^2-(c_s^{-})^2|\vec{u}^{-}|^2\big]\\
&=(c_s^{-}u_n^{-})^2(1-|\vec{u}^{-}|^2)+\big[1-(u_n^{-})^2-(c_s^{-})^2|\vec{u}^{-}|^2+(c_s^{-}u_n^{-})^2\big]+2\alpha^{-}c_s^{-}u_n^{-}\\
&\ge(c_s^{-}u_n^{-})^2(1-|\vec{u}^{-}|^2)+\big[1-(u_n^{-})^2-(c_s^{-})^2|\vec{u}^{-}|^2+(c_s^{-}u_n^{-})^2\big]-2\alpha^{-}c_s^{-}|u_n^{-}|\\
&=(c_s^{-}u_n^{-})^2(1-|\vec{u}^{-}|^2)+\big[1-(u_n^{-})^2-(c_s^{-})^2|\vec{u}^{-}|^2+(c_s^{-}u_n^{-})^2\big]\\
&~~~-2c_s^{-}|u_n^{-}|\sqrt{(1-|\vec{u}^{-}|^2)\big[1-(u_n^{-})^2-(c_s^{-})^2|\vec{u}^{-}|^2+(c_s^{-}u_n^{-})^2\big]}\\
&=\bigg(c_s^{-}u_n^{-}\sqrt{1-|\vec{u}^{-}|^2}-\sqrt{1-(u_n^{-})^2-(c_s^{-})^2|
\vec{u}^{-}|^2+(c_s^{-}u_n^{-})^2}\bigg)^2\ge0.
\end{align*}
Hence it is true that $f^{-}(s_{\min}^{-})>0$ and thus
the inequality  (\romannumeral7) for $\vec U^-$ is proved.

For the right state $\vec U^+$,   since
\begin{align*}
s_0^{+}-s_{\max}^{+}&=\frac{u_n^{+}\big[(E^{+})^2-(D^{+})^2-|\vec{m}^{+}|^2-(p^{+})^2\big]}{(E^{+})^2-(D^{+})^2-|\vec{m}^{+}|^2}
-\frac{u_n^{+}\big(1-(c_s^{+})^2\big)+c_s^{+}\alpha^{+}}{1-(c_s^{+})^2|\vec{u}^{+}|^2}\\
&=C_9\bigg(-u_n^{+}\big[1-(c_s^{+})^2|\vec{u}^{+}|^2\big]\bigg)+(p^{+})^2\big[u_n^{+}(c_s^{+})^2(1-|\vec{u}^{+}|^2)
-c_s^{+}\alpha^{+}\big]\\
&=C_9\left(u_n^{+}\big[\beta^{+}-1+(c_s^{+})^2\big]-\frac{\alpha^{+}}{1-|\vec{u}^{+}|^2}\left(c_s^{+}(1-|\vec{u}^{+}|^2)
+\frac{\beta^{+}}{c_s^{+}}\right)\right)\\
&<C_9\left(\beta^{+}-1+(c_s^{+})^2-\left(c_s^{+}(1-|\vec{u}^{+}|^2)-\frac{\beta^{+}}{c_s^{+}}\right)\right)\\
&=C_9\left(\left(1-\frac{1}{c_s^{+}}\right)\beta^{+}-(1-c_s^{+})-c_s^{+}(1-|\vec{u}^{+}|^2)\right)<0,
\end{align*}
with
$$C_9=\frac{(p^{+})^2}{\big[1-(c_s^{+})^2|\vec{u}^{+}|^2\big]\big[(E^{+})^2-(D^{+})^2-|\vec{m}^{+}|^2\big]}>0,$$
$f^{+}(s)$ is monotonically increasing
 with $s\in[s^{+},1)\subset[s_{\max}^{+},1)$
and then we have $f^{+}(s^{+})>f^{+}(s_{\max}^{+})$.
Hence to prove the inequality  (\romannumeral7) for $\vec U^+$, it suffices to prove $f^{+}(s_{\max}^{+})>0$.
In fact, one has
\begin{align*}
f^{+}(s_{\max}^{+})&=(s_{\max}^{+}-u_n^{+})^2\big[(E^{+})^2-(D^{+})^2-|\vec{m}^{+}|^2-(p^{+})^2\big]+(p^{+})^2\big[(s_{\max}^{+})^2-1\big]\\
&=C_{10}\left(\frac{\beta^{+}\big[c_s^{+}u_n^{+}(1-|\vec{u}^{+}|^2)-\alpha^{+}\big]^2}{1-|\vec{u}^{+}|^2}
+\big[u_n^{+}(1-(c_s^{+})^2)+c_s^{+}\alpha^{+}\big]^2-\big[1-(c_s^{+})^2|\vec{u}^{+}|^2\big]^2\right)\\
&=C_{10}\left(\beta^{+}(c_s^{+}u_n^{+})^2+\big[(\alpha^{+})^2-2\alpha^{+}c_s^{+}u_n^{+}\big]\big[\beta^{+}-1+(c_s^{+})^2\big]
+\left(1+\frac{\beta^{+}|\vec{u}^{+}|^2}{1-|\vec{u}^{+}|^2}\right)(\alpha^{+})^2\right.\\
&~~~+(u_n^{+})^2-2(c_s^{+}u_n^{+})^2+(c_s^{+})^4(u_n^{+})^2-1+2(c_s^{+})^2|\vec{u}^{+}|^2
-(c_s^{+})^4|\vec{u}^{+}|^4-\beta^{+}(c_s^{+}u_n^{+})^2|\vec{u}^{+}|^2\bigg)\\
&=C_{10}\left(\big[\beta^{+}-1+(c_s^{+})^2\big]\big[(c_s^{+}u_n^{+})^2-2\alpha^{+}c_s^{+}u_n^{+}+(\alpha^{+})^2\big]
+\left(1+\frac{\beta^{+}|\vec{u}^{+}|^2}{1-|\vec{u}^{+}|^2}\right)(\alpha^{+})^2\right)\\
&~~~+(u_n^{+})^2-(c_s^{+}u_n^{+})^2-1+2(c_s^{+})^2|\vec{u}^{+}|^2-(c_s^{+})^4|\vec{u}^{+}|^4-\beta^{+}(c_s^{+}u_n^{+})^2|\vec{u}^{+}|^2\bigg)\\
&=C_{10}\left((\beta^{+}-1+(c_s^{+})^2)(c_s^{+}u_n^{+}-\alpha^{+})^2
+\left(1+\frac{\beta^{+}|\vec{u}^{+}|^2}{1-|\vec{u}^{+}|^2}\right)(\alpha^{+})^2+(u_n^{+})^2-(c_s^{+}u_n^{+})^2-1\right.\\
&~~~+2(c_s^{+})^2|\vec{u}^{+}|^2-(c_s^{+})^4|\vec{u}^{+}|^4-\beta^{+}(c_s^{+}u_n^{+})^2|\vec{u}^{+}|^2\bigg)\\
&=C_{10}\big[\beta^{+}-1+(c_s^{+})^2\big]\bigg(\big(c_s^{+}u_n^{+}-\alpha^{+}\big)^2
+|\vec{u}^{+}|^2\big[1-(u_n^{+})^2-(c_s^{+})^2|\vec{u}^{+}|^2\big]\bigg),
\end{align*}
where
$$C_{10}=\frac{(p^{+})^2}{\big[1-(c_s^{+})^2|\vec{u}^{+}|^2\big]^2}>0,$$
and
\begin{align*}
&~~~(c_s^{+}u_n^{+}-\alpha^{+})^2+|\vec{u}^{+}|^2\big[1-(u_n^{+})^2-(c_s^{+})^2|\vec{u}^{+}|^2\big]\\
&=(c_s^{+}u_n^{+})^2+(\alpha^{+})^2-2\alpha^{+}c_s^{+}u_n^{+}
+|\vec{u}^{+}|^2\big[1-(u_n^{+})^2-(c_s^{+})^2|\vec{u}^{+}|^2\big]\\
&=(c_s^{+}u_n^{+})^2+(1-|\vec{u}^{+}|^2)\big[1-(u_n^{+})^2-(c_s^{+})^2|\vec{u}^{+}|^2+(c_s^{+}u_n^{+})^2\big]-2\alpha^{+}c_s^{+}u_n^{+}\\
&~~~+|\vec{u}^{+}|^2\big[1-(u_n^{+})^2-(c_s^{+})^2|\vec{u}^{+}|^2\big]\\
&=(c_s^{+}u_n^{+})^2(1-|\vec{u}^{+}|^2)+\big[1-(u_n^{+})^2-(c_s^{+})^2|\vec{u}^{+}|^2+(c_s^{+}u_n^{+})^2\big]-2\alpha^{+}c_s^{+}u_n^{+}\\
&\ge(c_s^{+}u_n^{+})^2(1-|\vec{u}^{+}|^2)+\big[1-(u_n^{+})^2-(c_s^{+})^2|\vec{u}^{+}|^2+(c_s^{+}u_n^{+})^2\big]-2\alpha^{+}c_s^{+}|u_n^{+}|\\
&=(c_s^{+}u_n^{+})^2(1-|\vec{u}^{+}|^2)+\big[1-(u_n^{+})^2-(c_s^{+})^2|\vec{u}^{+}|^2+(c_s^{+}u_n^{+})^2\big]\\
&~~~-2c_s^{+}|u_n^{+}|\sqrt{(1-|\vec{u}^{+}|^2)\big[1-(u_n^{+})^2-(c_s^{+})^2|\vec{u}^{+}|^2+(c_s^{+}u_n^{+})^2\big]}\\
&=\bigg(c_s^{+}u_n^{+}\sqrt{1-|\vec{u}^{+}|^2}-\sqrt{1-(u_n^{+})^2-(c_s^{+})^2|\vec{u}^{+}|^2+(c_s^{+}u_n^{+})^2}\bigg)^2\ge0.
\end{align*}
Hence it is true that $f^{+}(s_{\max}^{+})>0$ and thus the inequality  (\romannumeral7) for $\vec U^+$
holds.
The proof is completed.
\qed
\end{appendix}

\bibliographystyle{plain}

\end{document}